\documentclass[a4paper,11 pt]{amsart}

\usepackage{geometry} 
\usepackage{overpic}
\usepackage{amssymb}  
\usepackage{mathtools}      
\usepackage{mathabx}        
\usepackage[bb=fourier,cal=euler,scr=rsfs]{mathalfa}	
\usepackage{enumitem}       
\usepackage[colorlinks=true,backref=page]{hyperref} 

\usepackage[ansinew]{inputenc}
\usepackage[dvipsnames]{xcolor}

\renewcommand*{\backref}[1]{}
\renewcommand*{\backrefalt}[4]{\quad \tiny
  \ifcase #1 (\textbf{NOT CITED.})%
  \or    (Cited on page~#2.)%
  \else   (Cited on pages~#2.)%
  \fi}

\hypersetup{bookmarksdepth = 3} 
\numberwithin{equation}{section}     

\setlist[enumerate,1]{label={\upshape(\roman*)},ref=\roman*}
\setlist[enumerate,2]{label={\upshape(\alph*)},ref=\alph*}

\newcommand{\Z}{\mbox{$\mathbb{Z}$}}

\newcommand{\R}{\mbox{$\mathbb{R}$}}

\newcommand{\tild}[1]{\widetilde{#1}}

 \def\RR{{\mathbb R}}  
   
 \def\ZZ{{\mathbb Z}}
\def\cA{\mathcal{A}}  \def\cG{\mathcal{G}}  \def\cS{\mathcal{S}}
\def\cB{\mathcal{B}}    \def\cT{\mathcal{T}}
\def\cC{\mathcal{C}}  \def\cI{\mathcal{I}}  
\def\cD{\mathcal{D}}  \def\cJ{\mathcal{J}}  \def\cV{\mathcal{V}}
    \def\cW{\mathcal{W}}
\def\cF{\mathcal{F}}  \def\cL{\mathcal{L}}  
\def\cV{\mathcal{V}}

\newtheorem*{teo*}{Theorem}

\newtheorem{teo}{Theorem}[section]

\newtheorem{addendum}[teo]{Addendum}

\newtheorem{cor}[teo]{Corollary}

\newtheorem{af}[teo]{Claim}
\newtheorem{claim}[teo]{Claim}
\newtheorem{lema}[teo]{Lemma}
\newtheorem{lemma}[teo]{Lemma}
\newtheorem{prop}[teo]{Proposition}

\newtheorem{thmintro}{Theorem}

\theoremstyle{definition}
\newtheorem{defi}[teo]{Definition}
\theoremstyle{remark}

\newtheorem{remark}[teo]{Remark}
\newtheorem{notation}[teo]{Notation}
\newtheorem{example}[teo]{Example}

\newcommand{\eps}{\varepsilon}

\newcommand{\mt}{\widetilde{M}}

\newcommand{\fh}{\hat{f}}
\newcommand{\gh}{\hat{g}}
\newcommand{\ft}{\widetilde{f}}
\newcommand{\gt}{\widetilde{g}}
\newcommand{\Ft}{\widetilde{\cF}}
\newcommand{\cFe}{\widetilde{\cF}_{\eps}}
\newcommand{\cs}{\cW^{cs}}

\newcommand{\cu}{\cW^{cu}}

\newcommand{\ws}{\cW^{s}}
\newcommand{\wc}{\cW^{c}}

\newcommand{\wcs}{\tild{\cs}}
\newcommand{\wcsb}{\tild{\cs_2}}
\newcommand{\wcub}{\tild{\cu_2}}
\newcommand{\wcu}{\tild{\cu}}
\newcommand{\wwc}{\tild{\wc}}
\newcommand{\wwca}{\tild{\wc_1}}
\newcommand{\wwcb}{\tild{\wc_2}}
\newcommand{\wws}{\tild{\ws}}

\newcommand{\red}[1]{{\color{red} #1}}

\newcommand{\wT}{\widetilde{\cT}}
\newcommand{\Su}{S^1_{univ}}

\newcommand{\blue}[1]{{\color{blue} #1}}

\title[Partial hyperbolicity and pseudoAnosov dynamics]{Partial hyperbolicity and pseudo-Anosov dynamics}

\author[S.~Fenley]{Sergio R.\ Fenley} 
\address{Florida State University, Tallahassee, FL 32306, United States}
\email{fenley@math.fsu.edu}

\author[R.~Potrie]{Rafael Potrie} 
\address{Centro de Matem\'atica, Universidad de la Rep\'ublica, Uruguay}
\email{rpotrie@cmat.edu.uy}
\urladdr{http://www.cmat.edu.uy/~rpotrie/}

\thanks{S.F was partially supported by Simons Foundation grants numbered 280429
and 637554, by National Science Foundation grant
DMS-2054909, and by the Institute for
Advanced Study. R. P. was partially supported by CSIC. Part of this work was completed while R.P. was a Von Neumann fellow at the Institute for Advanced Study, funded by Minerva Research Fundation Membership Fund and NSF DMS-1638352. We would like to thank the IAS for the great working environment.}

\begin{document}

\begin{abstract}
We show that if a hyperbolic 3-manifold admits a partially hyperbolic diffeomorphism then it also admits an Anosov flow. 
Moreover, we give a complete classification of partially hyperbolic diffeomorphisms in hyperbolic 3-manifolds as well as partially hyperbolic diffeomorphisms in Seifert manifolds inducing pseudo-Anosov dynamics in the base. 
This classification is given in terms of the structure of their center (branching) foliations and the notion of collapsed Anosov flows. 
\end{abstract}

\maketitle

\section{Introduction}
A diffeomorphism $f:M \to M$ of a closed 3-manifold is \emph{partially hyperbolic} if its tangent bundle $TM$ splits as a $Df$-invariant sum $TM =E^s \oplus E^c \oplus E^u$ of one-dimensional continuous subbundles and there exists $\ell>0$ so that if $v^s,v^c,v^u$ are unit vectors in $E^s(x)$, $E^c(x)$ and $E^u(x)$ respectively, then:$$ \|Df^\ell v^s \| < \min \{1, \|Df^\ell v^c\| \} \qquad \text{and } \qquad \|Df^\ell v^u \| > \max \{ 1, \|Df^\ell v^c\| \}. $$
This paper is concerned with the classification problem of partially hyperbolic diffeomorphisms in dimension 3. 

It has become apparent that there is a strong link between partially hyperbolic diffeomorphisms and Anosov flows in dimension 3, at least when the manifold is ``sufficiently large''. This goes back at least to Pujals' conjecture \cite{BW} $-$ which roughly states that under certain very general conditions,
the diffeomorphism is a variable time map of a topological
Anosov flow. Recently new examples \cite{BPP,BGP,BGHP} have been constructed which fail Pujals' conjecture, for instance
in Seifert manifolds.  This has challenged our understanding of the topological structure of these systems.   
This paper aims to solve the classification problem (as formulated
in \cite[Question 1]{BFP}) 
completely for some particularly relevant classes of manifolds and isotopy classes of maps. 

The first result concerns the problem of finding topological obstructions,
or in other words to determine exactly
when  a manifold admits a partially hyperbolic diffeomorphism. This problem is well understood when the manifold has (virtually solvable) fundamental group \cite{HP-survey}, or when it is Seifert fibered under some
assumptions \cite{HaPS}. 
It is always possible to construct a partially hyperbolic diffeomorphism from an Anosov flow (in any manifold) by taking its time-one map. 
However it is expected that partially hyperbolic diffeomorphisms
are much more abundant amongst manifolds than Anosov flows.
For example the $3$ torus $\mathbb{T}^3$ or nil manifolds admit
partially hyperbolic diffeomorphisms, but do not admit Anosov flows.
In fact in $\mathbb{T}^3$ there are infinitely many essentially
distinct partially hyperbolic diffeomorphisms.
The reason these manifolds do not admit Anosov flows is because
the fundamental group does not have exponential  
growth, a necessary condition for the existence of an Anosov flow,
by work of Margulis \cite{Margulis}.

One big focus of this paper is
the case of hyperbolic 3-manifolds, that is, those homeomorphic to a quotient of $\mathbb{H}^3$ by a cocompact group of isometries. 
These 3-manifolds are by far the most abundant in the class of
closed, irreducible $3$-manifolds with infinite fundamental group,
by the famous work of Thurston and Perelman. Our first result is the following:

\begin{thmintro}\label{teoA}
Let $M$ be a closed hyperbolic 3-manifold admitting a partially hyperbolic diffeomorphism. Then, $M$ admits an Anosov flow.
\end{thmintro}

One consequence of Theorem A is that
it gives a complete set of obstructions up to the problem of determining which hyperbolic 3-manifolds admit Anosov flows. 
It is unknown which hyperbolic 3-manifolds admit Anosov flows, though some obstructions and examples are known \cite{Calegari} (see \S \ref{s.context} for more discussion).

The results that follow, of which Theorem \ref{teoA} is a consequence, pertain the topological classification of partially hyperbolic diffeomorphisms. 
More specifically we analyze the structure of a partially hyperbolic
diffeomorphism in a closed hyperbolic 
$3$-manifold and we show that the structure
we obtain allows us to construct a topological Anosov flow in
the manifold. We obtain this even if the diffeomorphism is not
at all the time one (or variable time) map of a flow.
The classification in hyperbolic $3$-manifolds
is built upon our previous work 
\cite{BFFP-announce,BFFP-2,BFFP-3} which deals, amongst many other
things, with general partially hyperbolic diffeomorphisms in hyperbolic 3-manifolds; and in addition
\cite{BFFP-4} which considers partially hyperbolic diffeomorphisms in certain isotopy classes of diffeomorphisms of Seifert manifolds. Our presentation aims to give a unified framework for both hyperbolic and Seifert $3$-manifolds, which can also be applied in other situations (see \S \ref{ss.generalres}). Note that while the previous work was mostly reliant on the homotopic to the identity hypothesis, the ideas and tools here as well as the overall strategy for classification that we introduce are applicable to wider situations.


Motivated by the present work, in \cite{BFP} we propose a notion of \emph{collapsed Anosov flows}. This relates partially hyperbolic diffeomorphisms with Anosov flows and their self-orbit equivalences. The class of
collapsed Anosov flows covers all known examples of partially hyperbolic diffeomorphisms in manifolds
with non solvable fundamental group
\cite{BFP}.
The concept of collapsed Anosov flow generalizes the notion of leaf conjugacy to the case when $f$ may not be \emph{dynamically coherent}\footnote{Dynamical coherence means that there exist $f$-invariant foliations $\cF^{cs}$ and $\cF^{cu}$ tangent respectively to $E^{cs}=E^s \oplus E^c$ and $E^{cu}=E^c \oplus E^u$. This notion has been used in several works. In this paper we obtain more precise integrability or non-integrability statements, so we will not need to use this notion.}. In other
words $f$ may not admit 
$f$-invariant foliations tangent to the center stable and center unstable bundles. The dynamically incoherent situation is unavoidable
and very common, as shown for example in Seifert manifolds
\cite{BGHP} (see also \cite{BFFP-2,BFFP-3,PotICM}).  
These works show that one needs new tools and models to attack a complete classification of partially hyperbolic diffeomorphisms. Our results provide a complete topological classification of partially hyperbolic diffeomorphisms up to the center direction in both hyperbolic 3-manifolds and some isotopy classes of partially hyperbolic diffeomorphisms in Seifert manifolds.

\subsection{Collapsed  Anosov flows}

Recall that an \emph{Anosov flow} is a $C^1$-flow $\phi_t: M \to M$ whose time one map is a partially hyperbolic diffeomorphism (see \S \ref{ss.anosovPA} for the standard definition and other properties). A \emph{self orbit equivalence} of $\phi_t$ is a homeomorphism $\beta: M \to M$ such that sends orbits of $\phi_t$ to orbits of $\phi_t$ preserving
orientation. These notions are discussed in more 
detail in \cite{BFP} (see \S \ref{ss.discretisedandcollapsed}
of this article). 

\begin{defi}[Collapsed Anosov flow]\label{def.cafw} 
A partially hyperbolic diffeomorphism $f: M \to M$ is a \emph{collapsed Anosov flow} if there is a topological Anosov flow $\phi_t: M \to M$, a self orbit equivalence $\beta: M \to M$ of $\phi_t$ and a continuous map $h: M \to M$ homotopic to the identity such that: 
\begin{enumerate}
\item $h$ maps orbits of the flow injectively onto $C^1$ curves tangent\footnote{With the results in this article we also prove 
that the map $h$ maps weak stable and weak unstable leaves 
of $\phi_t$ into $C^1$ surfaces tangent respectively to $E^{cs}$ 
and $E^{cu}$. See the discussion in \cite{BFP}. For the purposes of this introduction, this definition will be ok.} to the center direction $E^c$ of $f$,
\item one has that $f \circ h = h \circ \beta$. 
\end{enumerate}
\end{defi}

For a precise definition of a topological
Anosov flow, see \S~\ref{ss.discretisedandcollapsed}. 
It is a generalization of an Anosov flow.

In \cite{BFP} we studied Definition \ref{def.cafw},
and many of its possible variants. We showed that there are many examples of partially hyperbolic diffeomorphisms in 3-manifolds verifying this definition. We also studied different equivalent formulations and conditions that ensure that a partially hyperbolic diffeomorphism verifies this property. 
As we already remarked the conception of the idea of a collapsed
Anosov flow originated directly from the results and properties
proved in this article.

\subsection{Statements}

In this paper we will show the following result:

\begin{thmintro}\label{teoB}
Let $f: M \to M$ be a partially hyperbolic diffeomorphism on a hyperbolic 3-manifold. Then, it is a collapsed Anosov flow. 
\end{thmintro}

Theorem \ref{teoB} builds on \cite{BFFP-3} where a dichotomy is given for partially hyperbolic diffeomorphisms 
in a closed hyperbolic $3$-manifold: an iterate of
$f$ is either a \emph{discretized Anosov flow} (cf. \S \ref{ss.discretisedandcollapsed}), 
or is virtually a \emph{double translation} (cf. \S \ref{ss.dichotomy}). 
In a hyperbolic $3$-manifold any homeomorphism has a finite iterate
which is homotopic to the identity, so if needed we consider an
iterate of the original map.
A discretized Anosov flow is a map $f = \phi_{t(x)}(x)$ where
$\phi_t$ is a topological Anosov flow.
This is a generalization
of the time one map of an Anosov flow.
In particular, partially hyperbolic diffeomorphisms
in  dimension $3$ which are 
discretized Anosov flows are collapsed Anosov flows.
In this paper we further study the other possibility.
In other words we study the
double translation case to obtain that in this case it must also be a collapsed Anosov flow. We remark that in the double translation case $f$ cannot be dynamically coherent and the topological
Anosov flow we will construct is $\RR$-covered
\cite{BFFP-3}. 

Any topological Anosov flow in an atoroidal manifold is
transitive \cite{Mos}. Shannon \cite{Shannon} showed that any transitive topological Anosov flow
is orbitally equivalent to an Anosov flow.
This implies
that Theorem \ref{teoA} is a direct consequence of Theorem \ref{teoB}. 
Two flows are orbitally equivalent if there is a 
homeomorphim sending orbits of the first into orbits of the second
and preserving flow direction.

It is important to emphasize here that among the difficulties in 
showing Theorem \ref{teoB}  is the need to show that one does not need to take a finite cover or an iterate of $f$ 
to obtain the result. In order to deal with this problem, we  need to obtain strong uniqueness properties of the curves tangent to the center direction. 
More specifically, the results are obtained using branching foliations.
The fundamental results of Burago and Ivanov \cite{BI} show that these
exist for an iterate of $f$ lifted to a finite cover. The finite cover
has to do with orienting the bundles of the partially hyperbolic diffeomorphism and so that $Df$ preserves orientation. So a priori we obtain our
results on a finite cover of $M$ and for a lift of an iterate
of $f$. We show uniqueness of branching foliations in  this 
context, and this allows us to go back to $M$ to obtain the announced result in $M$ and for $f$ itself.

In \cite[Theorem A]{BFFP-3} we got a complete classification of partially hyperbolic diffeomorphisms on Seifert manifolds homotopic to the identity. Further results in this class of manifolds were obtained in \cite{BFFP-4} and will be used here. Here we treat new isotopy classes:

\begin{thmintro}\label{teoC}
Let $f: M \to M$ be a partially hyperbolic diffeomorphism on a 
Seifert manifold with hyperbolic (and orientable) base, so that $f$
acts as a pseudo-Anosov in the base, then, it is a collapsed Anosov flow. 
\end{thmintro}

We note that the class of diffeomorphisms in this result is non-empty \cite{BGHP}, and indeed some of the results, in particular the fact that center curves are quasigeodesics in their center stable
or center unstable leaves was known to be true for some open sets of examples (see\footnote{In  \cite[\S 10]{BFP} it is shown that not only small perturbations have this property but all partially hyperbolic diffeomorphisms that can be connected to the examples by a path of partially hyperbolic diffeomorphisms. It is unknown if the space of partially hyperbolic diffeomorphisms is connected in this isotopy class.} \cite[\S 5]{BGHP}) but unknown for general partially hyperbolic diffeomorphisms in these isotopy classes. 

As a consequence of these results 
we see that these partially hyperbolic diffeomorphisms must admit branching foliations (regardless of any a priori
orientability assumptions) and we will prove
uniqueness results for branching foliations in some settings, cf. \S \ref{s.uniqueness}. These results provide a complete classification of partially hyperbolic diffeomorphisms in these isotopy classes. 

One interesting point is that Theorems \ref{teoB} and \ref{teoC} admit mostly a unified proof, and we made an effort in presenting the unified point of view. 
We also show that this unified 
approach is helpful to study partially hyperbolic diffeomorphisms in other 3-dimensional manifolds. The difference between the proofs of
Theorems \ref{teoB} and \ref{teoC} has to do with how we show that certain general assumptions are met. Under these assumptions, we will get some even stronger results about classification (see \S \ref{s.hyperbolic} and \S \ref{s.seifert}). We also direct the reader to Theorem \ref{teo.curvesgeneral} for statements that include some surprising information about the strong stable and unstable foliations for certain examples (in contrast with the case of discretized Anosov flows, 
the stable and unstable foliations cannot look like horocycle foliations).  

\begin{remark}
In fact, for both Theorem \ref{teoB} and \ref{teoC} we obtain a stronger property which we called \emph{strong collapsed Anosov flow} in \cite{BFP}. See 
\S~\ref{s.CAF} for more discussion.
\end{remark}

\subsection{Tools developed in this article}

The main technical result 
we obtain in this paper is the following: under certain assumptions on the branching foliations $\cs,\cu$ preserved by a partially hyperbolic diffeomorphism $f$ of  a closed 3-manifold, we show that the leaves of the center foliation lifted to $\mt$ are quasigeodesics in the respective leaves of the 
foliations $\wcs, \wcu$. 
In \cite{BFP}  we call a partially hyperbolic diffeomorphim
which satisfies this quasigeodesic property a {\em quasigeodesic
partially hyperbolic diffeomorphism}. Under orientability conditions
this is one of the several equivalent definitions of forms
of collapsed Anosov flows, as proved in \cite{BFP}.

This eventually leads to the creation of a topological
Anosov flow which is associated with $f$. We note that this strategy is quite divergent with the strategy followed in \cite{BFFP-2,BFFP-3} which aimed to construct the topological Anosov flow by showing that a good lift in the universal cover would fix all leaves of $\wcs,\wcu$ and their intersections and then using partial hyperbolicity to conclude that the intersections should support a topological Anosov flow. In this paper we cannot assume that a lift of $f$ fixes leaves of $\wcs$ or $\wcu$, so, even if $f$ may be homotopic to the identity, we are forced to use the action of all lifts of $f$ in order to make our analysis. Also, we cannot rely on bounded deviations inside a leaf for the same reason, so we are forced to understand better the coarse geometry of the foliations and produce techniques to understand the behavior of their intersections. We point out that while very powerful, the strategy in \cite{BFFP-2,BFFP-3}, relying heavily on the existence of lifts of $f$ with special properties (fixing leaves), is unlikely to extend to other contexts. The strategy here is way more general, as shown for instance by the fact that we are able to obtain Theorem \ref{teoC}. 

The quasigeodesic property for center curves in leaves
of $\wcs, \wcu$,
and also ideas and
constructions of this article are what lead to the definition
of a collapsed Anosov flow (in its various forms),
which is done in \cite{BFP}.
In addition what is done in this paper is a general 
recipe to prove that a partially hyperbolic diffeomorphism
in any type of $3$-manifold, but admitting branching foliations
with Gromov hyperbolic leaves
is a collapsed Anosov flow (that is, it is a quasigeodesic
partially hyperbolic diffeomorphism). 
This establishes a program to study the structure
of general partially hyperbolic diffeomorphisms in dimension $3$. This is detailed in the next section (see also \cite{FP5,FP6} for recent progress in this program). 

In \cite{BFP} we explain
that all known examples of partially hyperbolic diffeomorphisms
with branching foliations having Gromov hyperbolic leaves
are collapsed Anosov flows.

\subsection{Idea of proof}

Let us discuss a bit the main difficulties we need to address and the new tools we develop to take care of them.

We focus on Theorem \ref{teoA} and forget at first about the orientability issues mentioned above which involve a different kind of problems that are discussed in \S \ref{s.funnel} and \S \ref{s.uniqueness}. 
In other words we assume the necessary orientability conditions.
As explained, from the work of \cite{BFFP-3} we can reduce to the 
double translation case: 
this means that the partially hyperbolic diffeomorphism preserves two (branching) foliations in $M$ that are uniform and $\RR$-covered. In addition
the lift $\ft$ of the homotopy of $f$ to the identity translates both such foliations. However, in principle we know nothing about how these foliations intersect, nor how they look at a big scale. The main driving goal we pursued in this project was the attempt to obtain geometric properties of the intersection of these (branching) foliations by showing that the intersected leaves are quasigeodesics in the leaves of each branching foliation when lifted to 
the universal cover. 

This strong geometric property is proved in steps.
We consider the center foliations in
(say) center stable leaves lifted to the universal cover. Each such
center stable leaf is Gromov hyperbolic and is compactified to a
closed disk with an ideal circle. 

\begin{enumerate}
\item We first show that for each ray in a center stable leaf $L$, the ray accumulates on a single
point in the ideal circle of $L$ in \S  \ref{ss.pAandsubfoliations} . We call this property \emph{landing} of rays. To show this we exploit the pseudo-Anosov behavior at infinity, introducing the notion of \emph{pseudo-Anosov pairs} ( \S~\ref{s.pApairs}). This works in quite some generality, not even partial hyperbolicity is used, only that a one-dimensional subfoliation is preserved by the map. The notion of pseudo-Anosov pair extends and subsumes some similar phenomena already appearing in \cite[\S 8]{BFFP-2} and \cite{BFFP-4}. 

\item Then we show that given a center
leaf $c$ in $L$, the ideal points of the two rays of $c$ are distinct
ideal points of $L$. This has several steps, the main of which is to establish a \emph{small visual measure property} for the invariant foliations. Here, partial hyperbolicity is used in a crucial way, but also a precise description of the obstruction is obtained (see  \S \ref{s.pApairsandPH}). This part requires some pseudo-Anosov pairs, but it is quite flexible (see \S \ref{s.further}). 

\item Finally we show that for any center stable leaf $L$ in the universal cover, then the leaf space of the
center foliation in $L$ is Hausdorff in \S \ref{s.Hausdorff}. Here is the main point where the full strength of the fact that we are working in hyperbolic 3-manifolds (or in the context of Theorem \ref{teoC}) is crucial. This motivates the notion of \emph{full pseudo-Anosov pair} and is related to specific properties of the laminations associated to the pseudo-Anosov elements that appear in these contexts. 

\item Together these properties
then imply that the centers are uniform quasigeodesics in
the center stable leaves as proved in  \S\ref{s.QG}. 
\end{enumerate}

It is worth mentioning that sections \ref{s.prelim} and \ref{s.pApairs} are quite heavy and are used to construct an abstract setting which is used to establish our results. The reader may find simpler to accept Proposition \ref{prop-dynamicspA} in a first read (and applying it to either Example \ref{example1} or \ref{example2}) and come back to those sections for specific definitions. For the specific results in hyperbolic 3-manifolds, what is needed from Proposition \ref{prop-dynamicspA} follows from \cite[\S 8]{BFFP-2}. 

We would like to mention another technical tool introduced here, which is the notion of {\em super attracting} fixed points
for actions on the universal circle in 
Subsection \ref{ss.superatracting}. This notion was first 
introduced in the setting of lifts of homeomorphisms of closed
surfaces in \cite{BFFP-4}. Here we generalize this notion to the
case of $\R$-covered, uniform foliations with Gromov hyperbolic
leaves. This notion plays a fundamental role in the definition
of pseudo-Anosov pairs and the properties that can be proved
from pseudo-Anosov pairs. We expect that it will also
be useful in other contexts. 

We mention that while showing the Hausdorff property relies very strongly on our topological assumptions (the ones in Theorem \ref{teoB} and \ref{teoC}), it is natural to expect it can be obtained by other means in other contexts (see for instance \cite{FP5}). We also mention that 
the last step, that is step (iv) in the
strategy above, has been recently extended to more generality \cite{FP6}. 

Finally we consider orientability issues: the results above use
branching foliations, which assume taking an iterate of $f$
lifted to a finite cover of $M$. We then prove invariance of these
branching foliations under deck transformations of
the finite cover. To obtain this we strongly use the quasigeodesic behavior we 
prove in the cover. The result is that the branching foliations
descend to $M$ and an iterate of $f$ satisfies all the orientability
conditions. Then using additional results of Burago and Ivanov
\cite{BI} we approximate the center stable and center unstable
foliations by foliations $\cF$ and $\cG$ which intersect
along a one dimensional oriented foliation, generating a flow.
We show that this flow is expansive, generating a topological
Anosov flow. By the result of Shannon \cite{Shannon} the flow is orbitally
equivalent to an Anosov flow finishing the proof of Theorem \ref{teoA}. 
We note that additional arguments are needed to obtain the proof
of Theorem \ref{teoB}. This can be achieved for instance
using the equivalence of the different
definitions of collapsed Anosov flows which is proved
in \cite{BFP}.

\subsection{More general results}\label{ss.generalres}

One consequence of our results is that we are able to solve the plaque expansivity conjecture (see \cite{HPS}) in hyperbolic 3-manifolds. See Corollary \ref{coro-plaqueexp}.

Also, some of the results in this work hold under weaker assumptions and these will be stated precisely in \S \ref{s.further}. For the interest of putting these results in value, we state some consequences of these results here. Note that some terms are undefined, but have been explained in the previous section. We refer the reader to \ref{s.further} for precise statements in some greater generality.

\begin{thmintro}\label{teoD}
Let $f: M \to M$ be a transitive partially hyperbolic diffeomorphism homotopic to the identity in a closed manifold $M$ with some atoroidal piece in its JSJ decomposition. Then, the  branching foliation $\cW^c$ has the small visual measure property in both $\cW^{cs}$ and $\cW^{cu}$. Moreover, either an iterate of $f$ is a discretized Anosof flow, or one of the foliations $\cW^s$ or $\cW^u$ has the small visual measure property on $\cW^{cs}$ (resp. $\cW^{cu}$).   
\end{thmintro}

We note that a similar result in the context of Seifert manifolds also holds:

\begin{thmintro}\label{teoE}
Let $f: M \to M$ be a transitive partially hyperbolic in a Seifert 3-manifold so that the action of $f$ in the base has a pseudo-Anosov component. Then, the  branching foliation $\cW^c$ has the small visual measure property in both $\cW^{cs}$ and $\cW^{cu}$. Moreover, both $\cW^s$ and $\cW^u$ have the small visual measure property on $\cW^{cs}$ and $\cW^{cu}$ respectively.   
\end{thmintro}

In particular, let us mention a suprising consequence of our techniques, stating that under the assumptions of these theorems, the strong stable and strong unstable foliations may have a behavior very different from the horocycle foliation of an Anosov flow (see Remark \ref{rem-horocycles}). As we mentioned, the missing step to have the collapsed Anosov flow property is the need to show the Hausdorff property for the center branching foliation in 
center stable and center unstable leaves
in the universal cover. This requires new ideas, but let us mention that recently, in \cite{FP5} this was achieved for unit tangent bundles of higher genus surfaces.

\subsection{Context and comments}\label{s.context} 
The problem of the topological classification of Anosov flows in 3-manifolds goes back to the seminal work of Margulis \cite{Margulis} and Plante-Thurston \cite{PlanteThurston} showing that a 3-manifold admitting an Anosov flow must have exponential growth of fundamental group. It is noteworthy that when 
these results appeared, the only known examples were 
orbitally equivalent to
geodesic flows in the unit tangent bundle of negatively curved manifolds and suspension of toral automorphisms. Since then, a myriad of new very different examples have appeared starting with the ones constructed by Franks-Williams \cite{Fr-Wi} and those by Handel-Thurston \cite{Ha-Th} and Goodman 
\cite{Go}  with somewhat different methods (we refer the reader to the introduction of \cite{BBY} for a list of known examples and constructions). This was just the beginning. In hyperbolic manifolds starting
with the fundamental work of Goodman \cite{Go}, new examples have continued to appear until very recently (see for instance \cite{FenleyAnosov,FoulonHasselblatt,BowdenMann}). Questions about which 3-manifolds support Anosov flows and how many orbitally
inequivalent ones such manifolds admit are still abundant. 
There has been considerable progress on the classification of Anosov flows in manifolds with non-trivial JSJ decomposition, we mention the recent work of Barbot and the first author in particular which gives a rather complete classification of what they call \emph{totally periodic Anosov flows} in 
graph manifolds as well as other classes, see \cite{BF1,BF2,BF3}. We refer the reader to \cite{BarbotHDR, BarthelmeAnosovsurv} for surveys about Anosov flows in dimension 3. 

The case of hyperbolic 3-manifolds is certainly the most mysterious. There are some known obstructions for hyperbolic manifolds to admit Anosov flows, and several constructions of such flows. Recently, some hyperbolic 3-manifolds have been shown to admit an arbitrary large number of
orbitally
inequivalent Anosov flows \cite{BowdenMann}.
 All these results make our results somewhat more interesting,
since it implies that we cannot 
compare our systems with some model (Anosov) systems in the manifold,
as is the
case for example in solvable manifolds (see for instance \cite{HP-survey}).
We point out in particular that the known topological obstructions to admit Anosov flows in hyperbolic 3-manifolds are very sensitive to taking finite covers (see \cite{RRS,CD}). Hence it is very important for us to obtain the results in Theorem \ref{teoA} without need to take finite covers (which introduces a big challenge, since our starting point is the existence of branching foliations from \cite{BI} which requires some orientablity assumptions). 

Let us first comment on our Theorems \ref{teoA} and \ref{teoB}. The first important thing to point out is that they both rely heavily on our previous work with Barthelm\'e and Frankel \cite{BFFP-3}: In that paper we showed that
a partially hyperbolic diffeomorphism of a hyperbolic 3-manifold (up to iterate so that it is homotopic to the
identity) is either a discretized Anosov flow or up to finite
cover admits two transverse taut (branching) foliations which were translated by the lift of the dynamics to
the universal cover. We point out that the existence of two transverse taut foliations (even with all possible orientability assumptions) in a hyperbolic manifold does not imply the existence of an Anosov flow, at least not one related to those foliations (see \cite{BBP}). Here, we analyze the
second case, and describe it completely. A posteriori this leads
to the existence of an Anosov flow in $M$.

We mention that the proof of Theorem \ref{teoB} is very similar to the proof of Theorem \ref{teoC} and we present it in a way that the only difference is in how one shows that certain conditions are met, that we do at the very end. On the other hand, Theorem \ref{teoC} is mostly self contained, since in the isotopy class under analysis, we only need to deal with that case. (The analogy would be that the discretized Anosov flow case is when $f$ is homotopic to the identity in a Seifert manifold, which is the case we treated in \cite{BFFP-3}.) 

Our results fit well in the program of classification of partially hyperbolic diffeomorphisms in dimension 3 and have motivated the definition of \emph{collapsed Anosov flows} which we believe may play an important role in this program. We refer the reader to \cite{PotICM,HP-survey} for recent surveys on the classification of partially hyperbolic diffeomorphisms in dimension 3. In \cite{BFP}, with Barthelm\'e, we have developed the notion of a collapsed Anosov flow that was suggested by this work and which provides a platform for classification.

\subsection{Organization of the paper} 
After giving some preliminaries in \S \ref{s.prelim} we introduce
the notion of super attracting fixed points in the
universal circle in Subsection \ref{ss.superatracting}.
In \S \ref{s.pApairs} 
we give a unified presentation of \cite[\S 8]{BFFP-2} and \cite[\S 11]{BFFP-3} as well as \cite{BFFP-4},
in particular using the notion of super attracting 
fixed points, and  which also works in other settings.
Part of section \S \ref{s.pApairs} is an extension of
previous work of 
\cite{BFFP-announce,BFFP-2,BFFP-3,BFFP-4} to 
a more general setting.

In \S \ref{ss.pAandsubfoliations} and \S \ref{s.pApairsandPH} we introduce new arguments that are presented in an abstract setting that has wider applicability and applications are given in \S \ref{s.further}. This serves two purposes, on the one hand it allows to obtain both results almost simultaneously; on the other hand, it also intends to express precisely what properties we use and where and allows to follow the arguments without prior knowledge on fine properties of hyperbolic 3-manifolds (the properties we will use only appear in \S~\ref{s.hyperbolic}). 

We use these results in \S \ref{s.Hausdorff} and \S\ref{s.QG} to make the key steps for our main results under more restricted assumptions that will be checked for our examples in \S \ref{s.hyperbolic} and \S \ref{s.seifert} (Theorems \ref{teoA} and \ref{teoB} are proved in \S \ref{s.hyperbolic} and Theorem \ref{teoC} is proved in \S \ref{s.seifert}). In \S \ref{s.CAF} we explain how the work done in sections \S\ref{ss.pAandsubfoliations}-\S\ref{s.QG} implies that certain partially hyperbolic diffeomorphisms are collapsed Anosov flows and in \S \ref{s.funnel} and \ref{s.uniqueness} we obtain some uniqueness results that will allow us to show our results without need to take finite covers and iterates. 

\subsection{Intersection and dependence on previous works}

This article intersects with \cite{BFFP-2,BFFP-3} and \cite{BFP},
so we expand on the interactions.
Sections \ref{s.pApairs} through \ref{s.QG} here are 
completely independent from previous works and stand on their
own (the setup in \S\ref{s.pApairs} is strongly motivated by the works \cite{BFFP-2,BFFP-3,BFFP-4}). It is in these sections that
we prove the quasigeodesic behavior
of center leaves in center stable and center unstable
leaves under certain conditions.

Section \ref{s.CAF} uses the property obtained in Section 
\ref{s.QG} to prove that certain partially hyperbolic
diffeomorphisms are collapsed Anosov flows. Here 
we directly quote a result from \cite{BFP}.
But the setup in this article is much simpler in certain aspects than
the general setup of \cite{BFP}, so in addition we provide
a short sketch and explanation of the proof in our more restricted setting.

Sections \ref{s.funnel} and \ref{s.uniqueness} are completely
independent of \cite{BFFP-2,BFFP-3,BFP} and stand on their
own. The same applies to Section \ref{s.seifert}, where
 we prove Theorem \ref{teoC}.

In Section \ref{s.hyperbolic} we prove Theorems \ref{teoA}
and \ref{teoB}. Here we rely on \cite{BFFP-3}
which describes the two possibilities for a partially
hyperbolic diffeomorphism in a closed hyperbolic manifold:
discretized Anosov flow and double translation. 
The first type is completely understood in \cite{BFFP-3}.
It is the second type we study here, and prove that it
satisfies the quasigeodesic behavior, in order to obtain
the collapsed Anosov flow property. This section's goal is to 
further understand what was known in \cite{BFFP-3},
 and how our general results can complete the picture 
when $M$ is hyperbolic.

Finally in Section \ref{s.further} we provide some 
extra consequences of our results, some of which
rely in part on \cite{BFFP-3,BFFP-2}.

In particular the whole strategy of proving that center
leaves are uniform quasigeodesics (and all its steps) is completely new and introduced in this article.

{\small \emph{Acknowledgements:} We would like to thank the referee for very useful feedback that helped us improve the presentation.} 


\section{Preliminaries and discussions on some notions}\label{s.prelim}

In this paper $M$ denotes a closed aspherical 3-manifold, and $\pi: \mt \to M$ the universal covering map. We will assume that the manifold does not have (virtually) solvable fundamental group. This allows to simplify some statements,  and the case of (virtually) solvable fundamental group
for partially hyperbolic diffeomorphisms
is already well understood (see \cite{HP-survey}). In some
sections at the end of the paper we will restrict further to $M$ being either Seifert or hyperbolic.  

Our results and statements will be independent of the chosen Riemannian metric, but we will fix one first for which the definition of partial hyperbolicity is given, and later we will change the metric so that the leaves of the (branching) foliations are negatively curved: this only changes definitions by bounded factors. 

In this section we introduce some preliminaries and fix notations which will be used later and relate with the objects introduced in the previous section. The reader familiar with \cite{BFFP-3,BFP} can safely skip this section,
except for Subsection \ref{ss.superatracting} where the notion of
super attracting fixed point in the universal circle is 
introduced. 

\subsection{Branching foliations}\label{ss.bran}
We will give a brief account on what we need about branching foliations introduced in \cite{BI} in our context. For a more detailed account we refer the reader to \cite[\S 3]{BFFP-3} or \cite{BFP}. 

Our definition will be a bit more restrictive (what we will define would be a \emph{Reebless branching foliation}) which is more than enough for our purposes and makes the definition easier to give. 

A \emph{branching ($2$-dimensional) foliation} on a closed 3-manifold $M$ is a collection of immersed surfaces $\cF$ tangent to a $2$-dimensional
continuous distribution $E$ of $TM$ such that if we consider $\Ft$ the lift of the collection to $\mt$ we have the following properties:

\begin{itemize}
\item[(i)] Every leaf $L \in \Ft$ is a properly embedded plane separating $\mt$ into two open connected components $L^{\oplus}$ and $L^{\ominus}$ depending on a fixed transverse orientation to $E$ lifted to $\mt$. Denote $L^+ = L \cup L^\oplus$ and $L^-=L \cup L^{\ominus}$. 
\item[(ii)] Every point $x \in \mt$ belongs to at least one leaf $L \in \Ft$. 
\item[(iii)] For every two leaves $L, F \in \Ft$ we have that either $F \subset L^+$ or $F\subset L^-$. This is the \emph{no topological crossings} condition.
\item[(iv)] If $x_n \to x$ and $L_n \in \Ft$ so that $x_n \in L_n$. Then every limit of $L_n$ in the compact-open topology belongs to $\Ft$. 
\end{itemize}

We will add an additional condition in the case that the distribution
$E$ is transversely oriented,
which is that every diffeomorphism preserving $E$ and its transverse orientation preserves the branching foliation in the sense that the image under $f$ of a leaf of $\cF$ is a leaf of $\cF$. 

A branching foliation is \emph{well approximated by foliations} if for every $\eps>0$ there is a true foliation $\cF_\eps$ tangent to a bundle $E_\eps$ and continuous maps $h_\eps: M \to M$ so that:
\begin{itemize}
\item The angle between $E_\eps$ and $E$ is smaller than $\eps$. 
\item The map $h_\eps$ is $\eps$-$C^0$-close to the identity (in particular, it is homotopic to the identity and therefore surjective) sending leaves of $\cF_\eps$ to leaves of $\cF$.
\item For every $L \in \cF$ there is a unique leaf $L_\eps \in \cF_\eps$ so that $h_\eps: L_\eps \to L$ is a local $C^1$-diffeomorphism so that $1-\eps < \|Dh_\eps^{-1}\|^{-1} \leq \|Dh_\eps\| <1+\eps$  (therefore, in $\mt$ it lifts to a diffeomorphism). 
\end{itemize}
Note that when a branching foliation is well approximated by foliations we can define a \emph{leaf space} $\cL_\cF$, for example  by identifying with the leaf space of some of the approximating foliations $\cL_{\cF_{\eps}} = \mt /_{\widetilde{\cF_{\eps}}}$. 
This uses the uniqueness result in the third item above, so there 
is a bijection between the leaf spaces of $\widetilde \cF$ 
and $\widetilde \cF_{\eps}$.

We will use the following result (the uniqueness statement for the approximating foliation is explained in \cite[Appendix A]{BFP}): 

\begin{teo}[\cite{BI}]\label{teo-bi}
Let $f: M \to M$ be a partially hyperbolic diffeomorphism of a closed 3-dimensional manifold so that the bundles $E^s, E^c, E^u$ are oriented and $Df$ preserves their orientation. 
Then, there exist branching foliations $\cs$ and $\cu$ tangent to $E^{cs}$ and $E^{cu}$ respectively. These branching foliations are well approximated by foliations. 
\end{teo}

We will denote by $\cL^{cs}$ and $\cL^{cu}$ the leaf 
spaces respectively of the lifts $\wcs, \wcu$ of the branching foliations. In our setting, we will be considering a special class of foliations where leaf spaces are easier to define. See \cite[\S 3.2.2]{BFFP-3} for the general case. 

We remark also that the intersection of $\wcs$ and $\wcu$ gives rise to a one-dimensional branching foliation $\widetilde{\wc}$ which also has a well defined leaf space (see \cite[\S 2.3]{BFP}). By \emph{one dimensional branching foliation} $\cT$ which subfoliates a foliation $\cF$ we mean a collection of $C^1$-curves such that in the universal cover, for every $L \in \Ft$ the curves of $\wT$ contained in $L$ have the same properties ($i$)-($iv$) defining two dimensional branching foliations (of course, in ($i$) one needs to change properly embedded plane to properly embedded line).

\subsection{$\R$-covered foliations and hyperbolic metrics}\label{ss.Candel}

A celebrated result of Candel \cite{Candel} states that under quite 
general conditions, given a foliation in a $3$-manifold, there is a 
metric on $M$ that makes every leaf a hyperbolic surface. 
In particular, it follows from
\cite[Theorem 5.1]{FPacc} that this is the case for every minimal foliation in manifolds with exponential growth of $\pi_1(M)$ as we will consider here. 
Here is why:
If there are no holonomy invariant transverse measures, then
this result follows directly from Candel's theorem \cite{Candel}.
If there is a holonomy invariant transverse measure, then
this result follows from \cite[Theorem 5.1]{FPacc}.
Since we will be mainly concerned with minimal $\RR$-covered foliations we will consider such a metric for some approximating foliations 
$\cF^{cs},\cF^{cu}$ to $\wcs,\wcu$ and this will induce a (coarsely) negatively curved metric on each leaf of both branching foliations, see the next subsection. 

We will say that the branching foliation $\cF$ is $\R$-\emph{covered} if for every pair of leaves $L, F \in \Ft$ we have that either $L^+ \subset F^+$ or $F^+ \subset L^+$. This allows to induce an order between leaves and therefore it is equivalent to having that the approximating foliation is $\R$-covered, that is, the leaf space $\cL_{\cF}$ is homeomorphic to $\RR$. Compare with \cite[\S 11.1, \S 11.2]{BFFP-3}. 

Most of our study will take place for \emph{uniform} and $\R$-covered branching foliations. A branching foliation $\cF$ is \emph{uniform} if given two leaves $L, F \in \Ft$ the Hausdorff distance between them in $\mt$ is finite. It follows from \cite[Theorem 1.1]{FPmin} (see also \cite[\S 6]{FPmin}) that a uniform branching foliation is $\R$-covered
(here we use the hypothesis that the approximating foliations
are Reebless). We will say that a branching foliation is by \emph{hyperbolic leaves} if the metric on $M$ makes all leaves uniformly Gromov hyperbolic (see \cite[\S A.3]{BFP}). 

\subsection{Boundary at infinity and visual metric}\label{ss.boundaryvisual}

Let $L$ be a complete simply connected surface which is Gromov hyperbolic. We can define $S^1(L)$ its \emph{visual (or Gromov) boundary} as the set of geodesic rays on $L$ identified by being equivalent if they are a finite
Hausdorff distance from each other in $L$.
An oriented bi-infinite
geodesic $\ell \in L$ has therefore two endpoints $\ell^{\pm}$ corresponding to the positively and negatively oriented rays of $\ell \setminus \{x\}$ for some $x \in \ell$. This is clearly independent of the point $x \in \ell$.

One can compactify $L$ to $\hat L = L \cup S^1(L)$ making it homeomorphic to the closed disk (see \cite[\S III.H.3]{BH} or \cite{FenleyQG}). Given a geodesic $\ell$ in $L$  and a point $\xi$ in $S^1(L) \setminus \{\ell^\pm\}$ we can define an open set $O_\ell(\xi)$ containing $\xi \in S^1(L)$ as the union of the open interval of $S^1(L)$ whose endpoints are $\ell^+$ and $\ell^-$ and contains $\xi$ and the connected component of $L \setminus \ell$ containing completely a geodesic ray $r$ representing $\xi$. Note that for every $\xi \setminus \{\ell^\pm\}$ there are rays in $L \setminus \ell$ representing it, and the definition of the open set $O_\ell (\xi)$ is independent of this choice. The open sets $O_\ell(\xi)$ with $\ell$ being geodesics in $\ell$ together with the open sets in $L$ give a topology in $\hat L$ making it homeomorphic to a closed disk. 

For several reasons, we will choose a metric in $M$ whose restriction to leaves of $\cF$ is $\mathrm{CAT}(\kappa)$ for some $\kappa<0$ which can be done using the Candel metric, where we approximate by
an actual foliation, as opposed to a branching foliation (see \cite[Proposition A.5]{BFP}). This is not so important, since there are canonical ways to define visual metrics, but it is convenient to have a geometric sense of what is happening.
Scaling the metric, we can always assume that all leaves
are $\mathrm{CAT}(-1)$.
This property implies uniqueness of geodesic segments, rays, or full 
geodesics given the endpoints or ideal points.
In particular the $\mathrm{CAT}(-1)$ property implies that for any
$x$ in a leaf $ L$ there is exactly one such ray starting at $x$ for every $\xi \in S^1(L)$ so one can identify $S^1(L)$ with $T_x^1L$. This defines, for each $x \in L$ a \emph{visual metric} on $S^1(L)$ 
given by measuring intervals
 in $S^1(L)$ by the angle in $T_x^1L$ measured with the Riemannian metric on $L$. 
A very important fact for us is the following:

\begin{remark} \label{rem:holder}
The visual metric in $S^1(L)$ is well defined up to H\"{o}lder equivalence \cite[\S III.H.3]{BH}. 
\end{remark}

These metrics in the leaves vary continuously with the points in $M$.
Then $S^1(L)$ is canonically identified with $T_x^1 L$.
Also when $x$ varies, the sets  $T_x^1 L_x$ vary continuously. It follows that
the visual metrics in $S^1(L)$ vary continuously with the points.

We refer the reader to \cite{Thurston,Fen2002,CalegariPA,Calegari} for more general statements.  

\subsection{The universal circle}\label{ss.universalS1}
 In this subsection we describe what we need of the universal
circle of the foliation which allows us 
to determine the dynamics at infinity of a good pair. What is described here is done with much more detail
and richer properties in
\cite{Thurston,Fen2002,CalegariPA}. Our construction is from scratch with only the properties we will need, see also \cite[\S 2.5]{FPmin}.

 Let $\cF$ be a uniform $\RR$-covered branching foliation on $M$ by hyperbolic leaves and $\Ft$ its lift to $\mt$. For each $L \in \Ft$ we consider its visual boundary $S^1(L)$ as in \S \ref{ss.boundaryvisual}.  
 
 Let $\cA = \bigsqcup_{L \in \Ft} S^1(L)$ which can be given a topology 
from the collection of $T^1 \Ft|_{\tau}$ where $\tau$ is a transversal
compact segment to $\Ft$.
Then if $L_n \to L$ in $\cL = \mt/_{\Ft}$ then $S^1(L_n) \to S^1(L)$. With this topology $\cA$ is an open annulus since it is a circle bundle over the leaf space $\cL$ of $\Ft$.

Recall that a \emph{quasi-isometry} between two metric spaces $(X,d_X)$ and $(Y, d_Y)$ is a map $q: X \to Y$ such that there exists $C>1$ such that for every $x,x' \in X$ we have: 

$$  C^{-1} d_X(x,x') - C \leq  d_Y(q(x),q(x')) \leq C d_X(x,x') + C . $$

We do not require $q$ to be continuous, the constant $C$ is called a \emph{quasi-isometry constant} for $q$. A \emph{quasigeodesic} in $(X,d_X)$ is a quasi-isometric map from $\RR$ with its usual distance into $X$ and a quasigeodesic ray is a quasigeodesic map from $[0,\infty)$ to $X$. The Morse Lemma (see eg. \cite[Theorem III.H.1.7]{BH}) states that if $L$ is a Gromov hyperbolic simply connected surface then for every $C$ there is some $K$ such that if $r$ is a $C$-quasigeodesic (resp. $C$-quasigeodesic ray), then there exists a true geodesic (resp. geodesic ray) at Hausdorff distance less than $K$ from $r$. This also holds with the same constants for quasigeodesic segments and the geodesic joining the endpoints.

\begin{prop}\label{prop.ext} Let $f: M \to M$ be a homeomorphism preserving $\cF$, it follows that any lift $\fh$ extends naturally to a homeomorphism of $\cA$ that we still call $\fh$ and preserves the fibers (i.e. it is a homeomorphism from $S^1(L)$ to $S^1(\fh(L))$. 
\end{prop}

Note that, by taking $f = \mathrm{id}$ this includes action by deck transformations in $\mt$.

\begin{proof}
Since $f$ is a homeomorphism of $M$ which is compact, then any given lift $\fh$ induces quasi-isometries\footnote{One can cover the manifold by finitely many sufficiently small foliations charts. A homeomorphism verifies that the image of a plaque in a foliation chart can intersect only finitely many (uniformly bounded number of) foliation charts. Since plaques in the chart have size uniformly bounded  from above and below, one deduces that $\fh$ must be a uniform quasi-isometry between leaves of the foliation.} from $L$ to $\fh(L)$ for every $L \in \Ft$ so it maps geodesic rays into quasigeodesic rays. The Morse Lemma implies that these are bounded distance away from a well defined geodesic ray up to bounded distance. This induces a continuous map from $S^1(L)$ to $S^1(\fh(L))$ and $\fh^{-1}$ induces its inverse so it is a homeomorphism.  
\end{proof}
 
It is many times useful to collate all circles in $S^1(L)$ by constructing a \emph{universal circle}, introduced by
Thurston.  There are standard constructions, which in the setting of $\RR$-covered uniform foliations gets simplified \cite[\S 5]{Thurston} (see \cite{Fen2002,CalegariPA,Calegari} for more details and more general constructions). 

To do this for $\RR$-covered uniform foliations, it is important to construct a natural way to identify leaves of $\Ft$. Intuitively, one can think as if there is a flow $\Phi_t$ in $M$ which is transverse and regulating to $\cF$: this means that if one considers two leaves $L_1, L_2 \in \Ft$ then the time it takes the flow $\Phi_t$ to take a point of $L_1$ to a leaf in $L_2$ is bounded above by a constant only depending on $L_1$ and $L_2$. This flow can be extended to $\cA$ and gives a way to identify fibers. Such a flow exists
for general transversely oriented, 
uniform $\R$-covered foliations \cite{FenleyFlow}.
To construct the identification between distinct circles
at infinity, less is needed: 

\begin{prop}\label{prop.quasiisome}
There is a family $\{\tau_{L,L'}: L \cup S^1(L) \to L' \cup S^1(L')\}$ for $L, L' \in \Ft$ with the following properties: 
\begin{enumerate}
\item the map $\tau_{L,L'}|_{L}: L \to L'$ is a quasi-isometry whose constant depends only on the Hausdorff distance between $L$ and $L'$,
\item the map $\tau_{L,L'}|_{S^1(L)} : S^1(L) \to S^1(L')$ is a homeomorphism, 
\item one has that $$ \tau_{L',L''}|_{S^1(L')} \circ  \tau_{L,L'}|_{S^1(L)} = \tau_{L, L''}|_{S^1(L)}.$$
\end{enumerate}
\end{prop}

This statement can be found in \cite[\S 5]{Thurston}, \cite[Corollary 5.3.16]{CalegariPA} or \cite[Proposition 3.4]{Fen2002} and the proof works exactly the same for branching foliations\footnote{Or can be deduced for them by using approximating foliations.}.
The quasi-isometries $\tau_{L,L'}: L \mapsto L'$ are
coarsely well defined, in the sense that given $L, L'$ there
is a constant $b$ which depends only on the Hausdorff distance
between $L, L'$ so that for any $x$ in $L$, then 
$d_{\mt}(x, \tau_{L,L'}(x)) < b$.
For $\R$-covered foliations this implies that if $\tau'_{L,L'}$ is
another such map then 

$$d_{L'}(\tau_{L,L'}(x),\tau'_{L,L'}(x)) \ \ < \ \ b_1,$$

\noindent
for a constant $b_1$ that depends only on $b$. So $\tau_{L,L'}$ is
coarsely defined.

We can now define $\Su$, the \emph{universal circle} of the foliation $\cF$, as $\cA/_{\sim}$ where we identify the circles $S^1(L)$ and $S^1(L')$ via the maps $\tau_{L,L'}$ from the proposition. It is important to remark that the universal circle depends on the foliation, and so, when several foliations are involved (as is the case of partially hyperbolic diffeomorphisms) we will make an effort to make clear which circle we are considering. For any $L$ in $\widetilde \cF$ define 
$$\Theta_L : \Su \mapsto S^1(L)$$
\noindent
the map that associates to a point in $\Su$ its representative in $S^1(L)$.
Notice that for any leaves $L, E$ of $\Ft$ then 
$$\Theta_E \ \ = \ \ \tau_{L,E} \circ \Theta_L$$

A useful property for us is that the following extension of Proposition \ref{prop.ext} holds: 

\begin{prop}\label{prop.extuniv}
Let $f: M \to M$ be a homeomorphism preserving an $\RR$-covered uniform branching foliation $\cF$ by hyperbolic leaves and $\fh$ a lift to $\mt$. Then, there is a well defined action $\fh_\infty$ of $\fh$ on $\Su$ given by $\fh_\infty=\Theta_L^{-1} \circ \tau_{\fh(L),L} \circ \fh \circ \Theta_L : \Su \to \Su$, where $L$ is an arbitrary leaf of $\widetilde \cF$.
In other words the map $\fh_\infty$ is independent of the choice of $L$. 
\end{prop}
 
 \begin{proof} 
We sketch the proof, for more details see the proof of  \cite[Proposition 3.4]{Fen2002}.
Let $p$ in $\Su$. 
Choose an arbitrary leaf $L$ of $\Ft$ to start with.
The point $p$ in $\Su$ is associated with
a point $q$ in $S^1(L)$, $q = \Theta_L(p)$. Let $r$ be a geodesic ray
in $L$ with ideal point $q$. Then since $\hat f$ is a quasi-isometry
from $L$ to $\hat f(L)$, it follows that $\hat f(r)$ is a quasigeodesic
ray and has a unique ideal point $q_0$ in $S^1(\hat f(L))$. 
Any other geodesic ray $r'$ in $L$ with ideal point $q$ in $L$,
$r'$ is asymptotic to $r$, hence $\hat f(r')$ is a finite
distance from $\hat f(r)$ and defines the same ideal point
in $\hat f(L)$. 

Finally we need to show that the map is independent of the choice of $L$, that is, that 

$$ \fh \circ \tau_{L, E}|_{S^1(L)}= \tau_{\fh(L), \fh(E)} \circ \fh|_{S^1(L)}. $$

\noindent
for any leaf $E$ of $\Ft$.

For this, let $E$ be another leaf of $\Ft$. The map $\tau_{L,E}: L \mapsto E$
is a quasi-isometry so that for any $x$ in $L$, then 
$d_{\mt}(x, \tau_{L,E}(x)) < b$ for $b$ which depends only
on the Hausdorff distance between $L, E$.
It follows that $\tau_{L,E}(r)$ is a quasigeodesic ray in $E$ which is
a bounded distance in $\mt$ from $r$. 
The quasigeodesic ray $\tau_{L,E}(r)$ is also a bounded
distance in $E$ from a geodesic ray in $E$ (this bound only depends
on the quasi-isometry constant of $\tau_{L,E}$). Hence there is
a geodesic ray $r'$ in $E$ which is a bounded distance in $\mt$
from $r$. If $q'$ is the ideal point of $r'$ in $E$, then
by definition $\tau_{L,E}(q) = q'$.
Taking the image of both $r, r'$ by $\fh$ we obtain
quasigeodesic rays $\fh(r), \fh(r')$ in $\fh(L), \fh(E)$
respectively, which are a bounded
distance from each other in $\mt$. The ideal point of $\fh(r)$ is
$\fh (q)$. The ideal point of the second is $\fh \circ \tau_{L,E}(q)$.
Since these quasigeodesic rays in $\fh(L), \fh(E)$ respectively
are a bounded distance from each other in $\mt$,
they define the same point in the universal circle, in other words

$$\tau_{\fh(L),\fh(E)}( \fh(q)) \ \ = \ \ \fh(\tau_{L,E}(q)),$$

\noindent
which is exactly what we wanted to prove.
 \end{proof}
 
\subsection{Visual metrics on the universal circle $\Su$}\label{ss.superatracting}

In the previous section we described visual metrics in 
individual leaves of $\widetilde \cF$.
It will be useful to have metrics on $\Su$ to talk about
super attracting fixed points of homeomorphisms acting on $\Su$.

Consider first a leaf $L$ of $\widetilde \cF$. 
There is a well defined bijection $\Theta_L : \Su \to S^1(L)$. The ideal circle $S^1(L)$ has
visual metrics: given $x_0$ in $L$ there is a bijection
between the unit tangent vectors to $T \widetilde \cF$ at $x$
and the points in $L$. The angle metric in $T^1_{x_0} \widetilde \cF$
induces a metric on $S^1(L)$. 
When one changes the basepoint in $L$ the visual metric in $S^1(L)$
changes by a H\"{o}lder homemorphism as noted in Remark \ref{rem:holder},
see \cite[\S III.H.3]{BH}. 

A map $g: (A,d) \mapsto (B,d')$ between metric spaces
is \emph{quasisymmetric} if 
\begin{itemize}
\item $g$ is an embedding, 
\item there is a homeomorphism $\eta: [0,\infty) \mapsto
[0,\infty)$ so that if $x,y,z$ are distinct points in $A$,
then 

$$\frac{d'(g(z),g(x))}{d'(g(y),g(x))} \ \ \leq \ \ 
\eta\left(\frac{d(z,x)}{d(y,x)} \right)$$
\end{itemize}
\noindent
See \cite[Definition 2.1]{Hai}.

When one changes from one leaf $L$ to other leaf $E$, one needs to understand the metric properties of the map $\tau_{L,E}$ restricted to $S^1(L)$. It is the identification
associated with the universal circle $\Su$. 

Recall from Proposition \ref{prop.quasiisome} that $\tau_{L,E}$ is a quasi-isometry from $L$
to $E$. Quasi-isometries between Gromov hyperbolic spaces induce quasi-symmetric homeomorphisms of the boundaries.

One can obtain this from the proofs of \cite[Propositions 5.15 and 6.6]{Gh-Ha}.
However they only explicitly talk about quasiconformal behavior,
which in dimension $1$ does not provide much information.
In \cite[Theorem 3.1]{Hai}
there is an explicit proof that a $C$-quasi-isometry between Gromov hyperbolic spaces
induces an ideal map $g$ which is quasisymmetric with constants related to $C$, the quasi-isometry constant.
He proves that it is quasi-M\"{o}bius (which we do not
define here), which implies quasisymmetric. We note that the quasi-isometry constant $C$ depends on the Hausdorff distance between leaves but we do not have much control on it. However, some metric properties make sense. We can now introduce super attracting fixed points.

First we choose a metric in $\Su$:
let $d$ be a visual metric in $\Su$ given say by
identification with $S^1(L)$ using a point in $L$ for some $L$
of $\widetilde \cF$.

\begin{defi}\label{def.superattract}
Let $f$ be a homeomorphism of $\Su$ which
fixes a point $\xi$ in $\Su$. We say that $\xi$ is a super
attracting fixed point for $f$ if 

$$\lim_{x \rightarrow \xi}  \ \frac{d(h(x),\xi)}{d(x,\xi)} \ \ = \ \ 0,$$

\noindent
where $h$ is the expression of $f$ using the
identification of $\Su$ with $S^1(L)$ for some leaf $L$ of 
$\widetilde \cF$.
\end{defi}

Compare with the definition given in \cite[Appendix A]{BFFP-4}
 which is done in a special situation. The name is inspired in complex dynamics where super attracting points are those whose derivative vanishes. Here, even if we cannot define the derivative since maps are only continuous, the quasi-symmetric structure allows `zero derivative' as in Definition \ref{def.superattract} to make sense as we will prove next.

\begin{lemma} \label{lem.superattracting}
The property of $\xi$ being a super attracting fixed point
for a homeomorphism
$f: \Su \mapsto \Su$ is independent
of the leaf $L$ in $\widetilde \cF$.
\end{lemma}

\begin{proof}
Let $h$ be the expression of the homeomorphism $f$ using
the identification $\Theta_L: \Su \to S^1(L)$.
If $\Su$ is identitifed with $S^1(E)$ for another leaf $E$ of 
$\widetilde \cF$, then the visual metrics $d$ in $\Su$ 
coming  from identification with $S^1(L)$ and $d'$ from
identification with $S^1(E)$ are quasisymmetric using
the map $g$ which is $\tau_{L,E}$ restricted to $S^1(L)$.
Let $\eta$ be the quasisymmetric function associated to $g$.
Then 

$$ \frac{d'(g(h(x),g(\xi))}{d'(g(x),g(\xi))} \ \ 
\leq \ \ \eta \left( \frac{d(h(x),\xi)}{d(x,\xi)}\right).$$

\noindent
Here $d'$ is the metric in $\Su$ coming from identification
with $S^1(E)$.
As $\eta$ is a homeomorphism with $\eta(0) = 0$, it follows that
$$\lim_{x \rightarrow \xi} \frac{d'(g(h(x),g(\xi))}{d'(g(x),g(\xi))}
= 0.$$
\noindent
Let $z = g(x)$. Since $g$ is a homeomorphism, then $x$ limits
to $\xi$ if and only if $z$ limits to $g(\xi)$. So we obtain

$$\lim_{z \rightarrow g(\xi)} \frac{d'(g(h(g^{-1}(z))), g(\xi))}
{d'(z,g(\xi))} \ = \ 0.$$

\noindent
But $g \circ h \circ g^{-1}$ is the expression of $f$ using
$E$ instead of $L$. 
This proves the lemma.
\end{proof}

 \subsection{Discretized and collapsed Anosov flows}\label{ss.discretisedandcollapsed}
We refer the reader to the paper \cite{BFP} which discusses in detail these concepts, as well as equivalences, examples and properties. Here we just give some quick definitions and properties that we will use to prove our results. 
Let $M$ be a closed 3-manifold. A non-singular flow $\phi_t:M \to M$ generated by a vector field $X$ is said to be \emph{Anosov} if there is a $D\phi_t$-invariant splitting $TM = E^s \oplus \RR X \oplus E^u$ and $t_0>0$ such that if $v^\sigma \in E^\sigma$ is a unit vector ($\sigma=s,u$) then:

$$ \|D\phi_{t_0} v^s \| < \frac{1}{2} < 2 < \|D\phi_{t_0} v^u\| . $$

It is easy to show that a flow on $M$ is Anosov if and only if its time $1$ map (and therefore its time $t$-map for every $t$) is partially hyperbolic. We refer the reader to \cite{FenleyAnosov, BarbotHDR, BarthelmeAnosovsurv} for generalities on the topological properties of Anosov flows.

We also have to consider the topological versions of these objects. A \emph{topological Anosov flow} $\phi_t: M \to M$ is an expansive flow tangent to a continuous vector field $X$ which preserves two transverse foliations so that orbits of one of the foliations get contracted under forward flowing while orbits of the other foliation are contracted by backward flowing. See \cite[Appendix G]{BFFP-2} for more discussions. 

It has recently been established by Shannon that transitive topologically Anosov flows are orbit equivalent to true Anosov flows \cite{Shannon}. 

More generally, a \emph{pseudo-Anosov flow} is a flow $\phi_t: M\to M$ preserving two transverse singular foliations which is locally modelled in a topological Anosov flow away from finitely many periodic orbits on which it has singularities of prong type. See \cite{Calegari} for more details. We note that every expansive
flow is pseudo-Anosov \cite{InabaMatsumoto, Paternain}.

In this paper we will be mostly interested in what is called $\R$-\emph{covered Anosov flows}: that is, topological 
Anosov flows whose 
stable foliation $\cF^{ws}$ and 
unstable foliations $\cF^{wu}$ lifted to $\mt$ are $\R$-covered. There are two important classes of $\R$-covered foliations: suspensions and \emph{skewed Anosov flows}. A fundamental early result on Anosov flows is the following: 

\begin{teo}[\cite{FenleyAnosov,Barbot}]\label{teoAnosov} The orbit space of the lift $\tild{\phi_t}$ of an arbitrary Anosov flow to $\mt$ is homeomorphic to $\R^2$. The flow is $\R$-covered if and only if one of the foliations $\cF^{ws}$ or $\cF^{wu}$ is $\R$-covered. The foliations 
$\widetilde \cF^{ws}, \widetilde \cF^{wu}$ 
induce one-dimensional foliations on $\R^2$ and $\phi_t$ is a suspension if and only if the foliations have a global product structure. 
If $\phi_t$ is $\RR$-covered and not orbitally
equivalent to a suspension, then
$\phi_t$ is \emph{skewed}. Moreover, every $\R$-covered Anosov flow is transitive. 
\end{teo}

Notice that the previous theorem is shown for topological Anosov flows, so, combined with \cite{Shannon} it says that if a topological Anosov flow is $\R$-covered, then it is orbit equivalent to a true Anosov flow. Also, it follows from \cite{Brunella} that topological Anosov flows in atoroidal 3-manifolds are always transitive.

\section{Pseudo-Anosov good pairs}\label{s.pApairs}

Let $\cF$ be a Reebless branching foliation of a closed 3-manifold $M$ which is $\RR$-covered and uniform and by hyperbolic leaves.
Denote by $\Ft$ to the lift of $\cF$ to $\mt$. We choose a transverse orientation for $\Ft$ and for $L \in \Ft$ we denote by $L^+$ and $L^-$ the closed half spaces determined by $L$ in $\mt$. 

By the definition of branching foliation it follows that given $L \in \Ft$, every leaf $L' \in \Ft$ is contained in either $L^+$ or $L^-$ (and if it is contained in both, it must be $L$). We will denote by $\Su$ the universal circle of $\cF$ (cf. \S \ref{ss.universalS1}).

\subsection{Good pairs}

We will be interested in certain lifts of maps that preserve the foliation $\cF$. 

\begin{defi} \label{good}
Given $f,g: M \to M$ diffeomorphisms of $M$ preserving $\cF$, a pair $(\fh, \gh)$ where $\fh, \gh: \mt \to \mt$ are lifts of $f,g$ is called a \emph{good pair} if they commute, neither fixes a leaf of $\Ft$ and one of them acts as the identity on $\Su$. 
\end{defi}

Notice that this implies that both $\fh$ and $\gh$ act as a translation on the leaf space $\cL_{\cF} \cong \RR$.

\begin{remark} \label{rem.quotient}
If $(\fh, \gh)$ is a good pair, then we can consider the quotient $M_{\fh} = \mt/_{<\fh>}$ which is a solid torus trivially foliated by the leaves
of the induced foliation $\cF_{\fh}$. 
The leaf space\footnote{Formally, we need to take the approximating foliation to define the leaf space, but one can also define $\cL_{\cF_{\fh}}$ by using the action of $\fh$ in $\cL_{\cF}$.} $ \cL_{\cF_{\fh}} = M_{\fh}/_{\cF_{\fh}}$ is a circle where $\eta_{\fh}$, the action induced by $\fh$ in the quotient, acts as a homeomorphism.
The same can be done to produce $M_{\gh} = \mt/_{<\gh>}$. 
\end{remark}

We will mostly have in mind the following two examples on which our results will be applied and eventually specialize to these cases: 

\begin{example}\label{example1}
Let $M$ be a closed 3-manifold and $\cF$ be a minimal foliation in $M$ preserved by a diffeomorphism $f: M \to M$ homotopic to the identity. It follows from \cite[Corollary 4.7]{BFFP-3} that if we consider $\ft$ to be a good lift of $f$ (i.e. the lift obtained by lifting a homotopy to the identity, cf. \cite[Definition 2.3]{BFFP-2}) then either every leaf of $\Ft$ is fixed by $\ft$ or $\cF$ is $\RR$-covered and uniform and $\ft$ acts as a translation on the leaf space $\cL$ of $\Ft$. Since $\ft$ is a bounded distance from the identity, one can easily show that $\ft$ acts as the identity on $\Su$,
when $\cF$ is $\R$-covered.
Moreover, as $\ft$ commutes with all deck transformations (which are lifts of the identity that clearly preserves $\cF$), it is enough to find a deck transformation $\gamma$ which does not fix any leaf of $\Ft$ to obtain a good pair $(\ft, \gamma)$. Note that such deck transformations are quite abundant (see for instance \cite[\S 8]{BFFP-2} and \cite[\S 10]{BFFP-3} for the case where $M$ is a hyperbolic 3-manifold). In this paper we will 
later consider this setting when $\Ft$ 
acting  as a translation
and  $M$ is a hyperbolic 3-manifold to prove Theorem \ref{teoB}.
\end{example}

\begin{example}\label{example2}
Let $f: M \to M$ be a diffeomorphism of a Seifert manifold $M$ with hyperbolic base and preserving a horizontal (branching) foliation $\cF$. 
Horizontal means that the Seifert foliation is isotopic to one
which is transverse to $\cF$.
In particular $\cF$ is $\R$-covered and uniform.
(Since $M$ has hyperbolic base, it follows that it has
a unique Seifert fibration up to isotopy, see eg. \cite[Appendix A]{BFFP-2}.)
Suppose now that the Seifert fibration is orientable.
Then $\pi_1(M)$ has non trivial center, which is infinite
cyclic, and we can take $\gamma$ to be a generator of the center.
The center corresponds to the regular circle fibers,
$\gamma$ picks an orientation on these.

In that case $f$ 
preserves the center of $\pi_1(M)$, and  up to taking 
a square $f$ preserves
the conjugacy class of $\gamma$.
Note that $\gamma$ acts as the identity on $\Su$. In addition, if $\Ft$ is the lift of $\cF$ to $\mt$ it follows that $\gamma$ does not fix any leaf of $\Ft$ because $\cF$ is horizontal. Moreover, again taking
the square of $f$ if necessary, any lift $\ft$ of $f$ commutes with $\gamma$. If one fixes a lift $\ft$, it follows that for large enough $m>0$ the pair $(\gamma^m \ft, \gamma)$ is a good pair.  This setting will be considered when the action in the base is pseudo-Anosov to prove Theorem \ref{teoC}. 
\end{example}

\begin{notation}\label{not-hn}
Given a good pair $(\fh, \gh)$ and $m, n$ integers
 we denote the by $P$ to the diffeomorphism $P=\gh^{m}\circ\fh^n$ of $\mt$ and
by $P_{\infty}$ the induced homeomorphism of
$\Su$ (cf. Proposition \ref{prop.extuniv}). The values of $m,n$ will be clear in the context. 
\end{notation}

\subsection{Super attracting points}\label{ss.superatractingpoints} 

In this subsection we study what happens when a good pair has a super attracting fixed point in the universal circle (cf. Definition \ref{def.superattract}) and how this forces some behavior in $\mt$. 
We remark that such useful information in $\mt$ from the action
at infinity cannot in general be obtained from a merely attracting
fixed point in $\Su$.

We first need to describe some natural `neighborhoods' of points of $\Su$ inside $\mt$ adapted to a good pair $(\fh,\gh)$. 
We will assume in all this subsection that $(\fh,\gh)$ is a good pair preserving $\cF$ and that $\gh$ is the element of the pair which acts as the identity on $\Su$.  
Recall that for any $L$ in $\widetilde \cF$, we denote by $\Theta_L : \Su \mapsto S^1(L)$
the map that associates to a point in $\Su$ its representative in
$S^1(L)$. 

We will need to introduce some notations. Given an interval $I$ of $\Su$ containing $\xi$ in its interior and $L \in \Ft$, we  denote by $\ell^L_I$ the geodesic in $L$ joining the endpoints of $\Theta_L(I)$. 
Given a leaf $L \in \Ft$ we denote by $L_I^0$ to the closure of the connected component of $L \setminus \ell^L_I$ whose closure in $L \cup S^1(L)$ contains $\Theta(I)$. Given  $b>0$ we denote as $L_I^{+b} \subset L$ (resp. $L_I^{-b}$) to the union of $L_I^0$ with the $b$-neighborhood of $\ell_I^L$ (resp. the points in $L^0_I$ at distance larger than $b$ from $\ell_I^L$).

\begin{defi}\label{not-neigh}
Given $\xi \in \Su$ we say that an open set $U$ of $\mt$ is a \emph{neighborhood} of $\xi$ if it is $\hat g$ invariant and 
for every $L \in \Ft$ we have that $U \cap L$ contains $L_{I(L)}^0$ for some $I(L) \subset \Su$ open interval containing $\xi$, and $I(L)$ varying 
continuously with $L$.
In addition we say that an unbounded sequence $x_n \in \mt$ converges to $\xi \in \Su$ if for every $U$ neighborhood of $\xi$ there is $n_0$ such that if $n>n_0$ then $x_n \in U$. 
\end{defi}

There is a lot of freedom in the definition of these neighborhoods
of points $\xi$ in $\Su$. Notice in particular that we require
that the neighborhood is $\gh$ invariant (where $g_\infty$ is
the identity). This requirement is necessary for some technical
results later on to hold. Heuristically, what it allows to do is to think about neighborhoods of points at infinity in the quotient $\mt/_{<\gh>}$, if $\gh$ were a deck transformation (thus an isometry) one could define the neighborhoods by looking at the geodesic joining the points in the interval at infinity, but since this is not necessarily the case, we need to work some more to produce useful neighborhoods.

\begin{prop} \label{prop.superattracting}
Let $(\fh, \gh)$ be a good pair and $P  = \gh^m \circ \fh^n$
so that $\xi$ in $\Su$ is a super attracting fixed point
of $P_{\infty}$.
There is a neighborhood $U$ of $\xi$ in $\mt$ so that 
$P(\overline U) \subset U$ and for any $x$ in  $\overline U$ then
$P^i(x)$ converges to $\xi$ when $i \rightarrow \infty$.
\end{prop}

We will first construct a family of neighborhoods of $\xi$ depending on open intervals $I$  so
that $\xi \in I \subset \Su$,  and a given leaf $L \in \Ft$ it satisfies
the conditions on the next lemma:

\begin{lema}\label{l.ghandtau}
Fix a leaf $L \in \Ft$, then, for every open interval $I_0 \subset \Su$ and $\eps>0$
there is $I \subset I_0$ open interval whose closure is contained in $I_0$ and $I_0$ is contained in its $\eps$-neighborhood in $\Su$ satisfying
the following:  we can define an open set $U_I$ which is a neighborhood of any $\sigma \in I$ (cf. Definition \ref{not-neigh}) and such that $L \cap U_I=L_I^0$. Moreover, there exists $b_0>0$ such that for every $E \in [L, \gh(L)]$ we have that $E_I^{-b_0} \subset U_I \cap E \subset E_I^{+b_0}$ 
\end{lema}

\begin{proof} 
Since the result is in the universal cover, we can assume by taking
a double cover that $\cF$ is transversely orientable.
Then Theorem \ref{teo-bi} implies that $\cF$
is approximated by an actual foliation $\cF_{\eps}$.
The universal circles of $\cF$ and $\cF_{\eps}$ are
canonically homeomorphic, under an equivariant homeomorphism.
Given a leaf $E$ of $\widetilde{\cF}$ there is an associated
leaf $E'$ of $\widetilde{\cF}_{\eps}$.

The foliation $\cF_{\eps}$ has leaves with curvature
arbitrarily close to $-1$. A contracting direction in a leaf
of $E$ of $\cFe$ is an ideal point $y$ of $E$ so that a geodesic
ray $r$
from a basepoint $x_0$ in $E$ to the ideal point $y$ satisfies
that all nearby leaves in one side of $E$ contracts towards
$E$ along $r$.
Thurston proved (see \cite[\S 3]{Fen2002}) that either there is
a holonomy invariant transverse measure 
or for every leaf 
$E$ of $\cFe$ the set of contracting
points from $E$ is dense in $S^1(E)$.
In the first case for every $\delta > 0$ there is 
also a dense set of directions in every leaf $E$ so that
nearby leaves stay 
always less than $\delta$ from $E$  in these directions.
The contracting points or points which are $\delta$ close
to nearby leaves project down to similar points of the 
branching foliation $\cF$. These contracting directions will allow us to `interpolate' between curves in closeby leaves to produce the desired neighborhoods.

\vskip .05in
There is $k_1 > 0$ so that the image  under $\hat g$
of any geodesic $\ell$ in a leaf  $E$ of $\Ft$
is a $k_1$ quasigeodesic
in $\gh(E)$.
Let $b_0$ be a global constant so that if $\ell$ is
a $k_1$ quasigeodesic in a leaf $E$  of $\widetilde{\cF}$, then $\ell$
is at most $b_0/2$ distant from the geodesic in $E$ with
same ideal points as $\ell$.
Fix $\delta > 0$ which is much smaller than the local
product size of the foliation $\cF$.
Now let $\delta_0  > 0, \ \delta_0 << \delta$
so that if two leaves $E_1, E_2$ of
$\Ft$ 
are within $\delta_0$ of each other along a geodesic
$\beta$ of $E_1$, then $E_1, E_2$ are within $\delta$ of
each other in a neighborhood of size $b_0$ of $\beta$.
This is why we use the approximating foliation, rather
than the branching foliation.

So fix the leaf $L$ as in the statement of the
lemma.
Given $\gh(L)$ we find 
ideal points  $z_1, z_2$
in $S^1(\gh(L))$ arbitrarily close to the endpoints
$\Theta_{\gh(L)}(\partial I_0)$ so that 
rays in $\gh(L)$ in the direction of $z_i$ are
$\delta_0$ close to all nearby leaves $E$ of $\Ft$ in
$[L,\gh(L)]$.
Let $y_i = (\Theta_{\gh(L)})^{-1}(z_i)$.

Consider $\ell$ the geodesic in $L$ with ideal points
$\Theta_L(y_i)$. 
Consider in $\gh(L)$ the geodesic $\beta$ with same ideal
points as $\gh(\ell)$.
Then $\beta, \gh(\ell)$ are at most $b_0/2$ distant from 
each other in $\gh(L)$.
Choose $E$ in $[L,\gh(L)]$ which is at most 
$\delta_0$ from $\gh(L)$ along $\beta$. One can do this
for some rays of $\beta$ in either direction
by the choice of $I$. Then by
choosing $E$ closer to $\gh(L)$ if necessary, one can choose
this for the whole geodesic $\beta$.
Let $B$ be the $b_0/2$ neighborhood of $\beta$ in $\gh(L)$.
Then $B$ is $\delta$ near $E'$ for any $E'$ in $[E,\gh(L))$.

In $E$ let $\ell_E$ be the geodesic with ideal points
$\Theta_E(y_i)$. The foliation $\Ft$ is a product
in the $\delta$ neighborhood of $B$, hence one can continuously
chooose curves $\ell_{G}$ for $G$ in $[E,\gh(L)]$ so that:

\begin{enumerate}
\item $\ell_G$ is a quasigeodesic in $G$,
\item $\ell_G$ has ideal points $\Theta_G(y_i)$, 
\item $\ell_G$ is within $b_0$ of the geodesic $\beta_G$ 
in $G$ with ideal points $\Theta_G(y_i)$.
\item $\ell_{E}$ is the geodesic with ideal points
$\Theta_E(y_i)$.
\item $\ell_{\gh(L)} = \gh(\ell_L)$.
\end{enumerate}

Now for $G$ in $[L,E]$ let $\ell_G$ be the geodesic with
ideal points $\Theta_G(y_i)$.

This defines the neighborhood $U_I$ for $G$ in $[L,\gh(L)]$.
Then iterate by $\gh$ to construct all of $U_I$.
By construction $U_I$ satisfies the last property of
the lemma. In addition $L \cap U_I = L^0_I$.

Finally we check the first property of the lemma.
For any $G$ leaf of $\Ft$, there is a unique
$n$ so that $E = \gh^{-n}(G)$ is in $[L,\gh(L))$. 
Then $U_I \cap G$ is $\gh^n(E \cap U_I)$. 
The set $E \cap U_I$ is bounded by a uniform quasigeodesic in $E$,
with endpoints $\Theta_E(\partial I)$. 
Hence $\gh^n(E \cap U_I) = G \cap U_I$ are also bounded
by uniform quasigeodesics with ideal points $\Theta_G(\partial I)$,
because $\gh^n$ is a quasi-isometry between leaves.
Hence for any $\sigma$ in $I$, there is $J$ open subinterval 
of $I$ containing $\sigma$ so that 
the set $U_I \cap G$ contains $G^0_J$
for all $G$ in $[L,\hat g(L)]$. Now fix $\sigma$ in $I$, then the
interval $J$ above depends on $G$,  and one
can chooose $J(G)$ varying continuously with $G$, by decreasing
it if necessary.

This finishes the proof of Lemma \ref{l.ghandtau}.
\end{proof}

\begin{remark} Note that if $E \notin [L, \gh(L)]$ one cannot ensure the containment and inclusion with the same constant $b_0$:
this is because one applies iterates of the quasi-isometry
$\gh$, whose quasi-isometry constants get worse with iteration.

This means that given $\sigma \in I$ it is not a priori true that 
there is a fixed interval $J$ with $\sigma \in \mathring{J}$ and $J \subset I$
with $G^0_J \subset U_I \cap G$ for all $G$ in $\Ft$.

\end{remark}

\begin{proof}[Proof of Proposition \ref{prop.superattracting}] 
Given $I \subset \Su$, $U_I$ as constructed in the previous lemma, 
and $E \in \Ft$,
we denote $A_E^I=E \cap U_I$. 

We claim that if $I$ as above is a sufficiently small interval  around
$\xi$, then  there are smaller intervals $J$ around $\xi$ 
 such that 
$$\gh^k \circ P(\overline{A_E^I}) \ \ \subset \ \ A_{E'}^J $$

\noindent 
for all $E$ in $[L,\gh(L)]$. 
We explain what $E'$ and $k$ are in this formula.
They are uniquely defined so that 
$E' \in [L, \gh(L))$ is the image of $E$ by $\gh^k \circ P$, and $k \in \ZZ$ is defined uniquely so that $\gh^k \circ P (E) \in [L, \gh(L))$. This will complete the proof of the first statement of the Proposition because
the formula above shows that $P(\overline{A_E^I}) \subset U_J$
for all $E$ in $[L,\hat g(L)]$.
The fact that $P(\overline{U}) \subset U$ 
then follows from the facts below:

1) $[L,\gh(L)]$ is a fundamental for the action of $\gh$ on $\mt$,

2) $P$ commutes with $\gh$,

3) Both $U_I$ and $U_J$ are $\gh$ invariant.

To get the property above, first note that the value of $k$ is uniformly bounded in $[L,\gh(L)]$ and so one gets that the quasi-isometric constants of the map $\gh^k \circ P: E \to E'$ are uniformly bounded independently on $E \in [L, \gh(L)]$ (where the $k$ depends on the particular
leaf $E$). 
It follows that there exists $b_1>0$ such that for every $J \subset \Su$ if we denote by $Z= P_\infty(J)$ we have that
$$\gh^k \circ P(E_J^{+b_0}) \ \ \subset \ \ (E')_{Z}^{+b_1}$$

\noindent
for every $E$ in $[L,\gh(L)]$, where $b_0$ was defined in
the previous Lemma..

Now we use the property of $\xi$ being super-attracting for the map
$P_{\infty}$. Since $\gh$ acts as the identity on $\Su$, then
$\xi$ is super attracting for $(\gh^k \circ P)_\infty$.

For each fixed $k$  one has that $(\gh^k \circ P)_\infty = P_\infty$.
Therefore 
one gets that for small enough intervals $I$ around $\xi$ the image $Z=P_\infty(I)$ verifies that the distance between the geodesics $\ell_I^E$ and $\ell_Z^E$ are much larger than $2b_0+b_1$ for every $E \in [L, \gh(L)]$. 
Here again we use that $[L,\gh(L)]$ is a compact interval in the leaf space
of $\widetilde \cF$.
Hence we can choose $J$ interval in $\Su$ around $\xi$ and
$U_J$ as defined in the previous lemma, so that $\ell^L_J$ separates $\ell^L_I$ from
$\ell^L_Z$, and 

$$(E)^{+b_1}_Z \ \ \subset \ \ (E)^{+0}_J \ \ \subset \ \
U_I \cap E$$

\noindent  for
all $E$ in $[L,\gh(L)]$. Let $U = U_I$.
This proves 
that $P(\overline U) \subset U$.

In addition one can choose the starting $I$ small enough,
so that in the proof above the distance from any point in
$\ell^I_E$ to any point in $\ell^J_E$ is bigger than
a constant $b_2 >> 2 b_0 + b_1$ for all $E$ in $[L,\hat g(L)]$.

This holds
for smaller $I$ as well.
In particular it holds for $J$. Hence the distance in $E$
(for any $E$ in $[L,\hat g(L)]$)
from any point 
in $\ell^I_E$ to any point in $P^2(U_I) \cap E$ is
at least $2 b_2$, and similarly for any positive $P^i$ iterate
it is at least $i b_2$.
This implies that for any neighborhood
$V$ of $\xi$ there is $i > 0$ so 
$$P^i(U_I \cap [L,\gh(L)])  \ \ \subset \ \  V.$$

\noindent
The $\gh$ invariance of the sets $U_I$ and $V$ then
implies that $P^i(U_I) \subset V$.
Hence for any $x$ in $\overline{U}$ then $P^i(x)$
converges to $\xi$ when $i \mapsto \infty$. This completes
the proof of Proposition \ref{prop.superattracting}.
\end{proof}

\begin{defi}\label{defi-basin} 
Given a good pair $(\fh, \gh)$ so that $\xi \in \Su$ is super attracting (resp. super repelling) for $P=\gh^m \circ \fh^n$.
We define the \emph{basin of attraction} (resp. \emph{basin of repulsion}) of $\xi$ to be the set of points $x$ in $\mt$ such that $P^{k}(x) \to \xi$ as $k \to +\infty$ (resp. $k \to -\infty$) understood as in Definition \ref{not-neigh}. 
\end{defi}

Proposition \ref{prop.superattracting} says that a super attracting point (which is defined only by the action on $\Su$) has a non-trivial basin of attracting which is a neighborhood of the super attracting
point.

\begin{remark}\label{rem.forwardbackwarditerate}
Let  $(\fh,\gh)$ be a good pair and assume that $\gh$ act as the identity on $\Su$. It follows that if a point $\xi \in \Su$ is super-attracting  for a lift $P=\gh^m  \circ \fh^n$, then $n\neq 0$ and then $\xi$ will be super-attracting for every lift $\gh^{k} \circ \fh^{\ell n}$ if $\ell >0$ and super-repelling if $\ell  <0$.
\end{remark}

\subsection{Pseudo-Anosov pairs}\label{ss.pApairs}

We can now define a technical object that will be central in our proofs:

\begin{defi}\label{def-pApair}
A good pair $(\fh,\gh)$ is a \emph{pseudo-Anosov pair} (or \emph{pA-pair}) if there is $n,m \in \ZZ$ such that if $P= \gh^m \circ \fh^n$ then the homeomorphism $P_\infty$ in the universal circle $\Su$ (cf. Notation \ref{not-hn}) has exactly $2p$ fixed points,
all of which  are alternatingly super-attracting and super-repelling. 
Here $p$ is an integer $\geq 2$.
If $p=2$ the pair will be called a \emph{regular pA-pair} and if $p\geq 3$ it will be called a \emph{prong pA-pair}. We denote by $I(P) = 1-p$ the \emph{index} of the pseudo-Anosov pair.  
\end{defi}

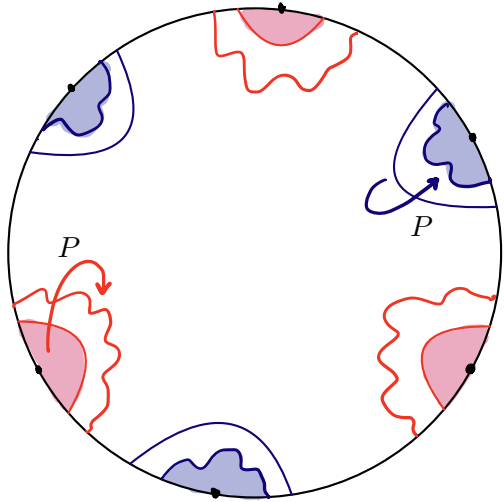
\begin{figure}[ht]
	\begin{center}
		\begin{overpic}[scale=0.83]{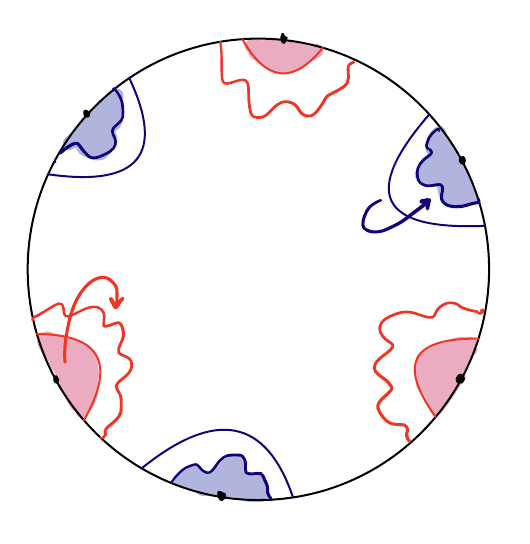}
		\end{overpic}
		\begin{picture}(0,0)

\put(-185,108){$P$}
\put(-53,116){$P$}

\end{picture}

	\end{center}
	\vspace{-0.5cm}
	\caption{{\small A pseudo-Anosov pair with $p = 3$.}}\label{f.papair}
\end{figure}

We next state a result which extends 
\cite[Proposition 8.1]{BFFP-3}  and \cite[\S 3]{BFFP-4}:

\begin{prop}\label{prop-core}
Let $(\fh,\gh)$ be a pseudo-Anosov pair and $P=\gh^m \circ \fh^n$ be a lift with $m,n \neq 0$, satisfying the conditions
of Definition \ref{def-pApair}.

Then there exists a closed set $T_{P} \subset \mt$ which is invariant under $\fh$ and $\gh$. The set $T_P$ 
intersects every leaf  $L$ of $\Ft$ in  a compact set $T_P \cap L$ which consists of the set of points 
which are not in the basin of attraction of any attracting
point of $P_{\infty}$ or the basin of repulsion of
any repelling point of $P_{\infty}$.

Moreover, if there is a leaf $L \in \Ft$ such that $P(L) = L$, then the total Lefschetz index of the compact invariant set $T_{P} \cap L$ is $I(P)$ the index of $P$. 
\end{prop}

The set $T_P$ is called the \emph{core} of the pair. It is the complement
in $\mt$ of what one detects
by looking at the action at infinity. 

We first define the basins of attraction.
Let $(\fh, \gh)$ be a pseudo-Anosov pair with
$\gh$ acting as the identity on $\Su$, and let $P = \gh^m \circ \fh^k$
as in Definition \ref{def-pApair}.
Let  $\{a_1, \ldots, a_p\}$ and $\{r_1, \ldots, r_p\}$ 
be the super attracting and super 
repelling points of $P_\infty$ on $\Su$. 
We define $T_P^+$ (resp. $T_P^-$) as the set of points
which is not in the 
basin of attraction of any of
the points $a_1, \ldots a_p \in \Su$ (resp.  not in the basis
of repulsion of any of the points $r_1, \ldots, r_p \in \Su$).
Let 
$$T_P \ \ = \ \ T^+_P \cap T^-_P.$$

Proposition \ref{prop-core} follows from applying the following consequence of Proposition \ref{prop.superattracting} that we state precisely for future use and prove below. 

\begin{prop}\label{prop-dynamicspA}
Let $(\fh, \gh)$ be a pseudo-Anosov pair with
$\gh$ acting as the identity on $\Su$, and let $P = \gh^m \circ \fh^k$
as in Definition \ref{def-pApair}. Then:

\begin{enumerate}
\item The set $L \setminus T_P^+$ (resp. $L \setminus T_P^-$) is non
empty and open.
\item  For each $L$ in $\widetilde \cF$,
$L \cap T_P \neq \emptyset$.
\item For any $L$ in $\Ft$, then $L \cap T_P$ is 
compact. In addition $T_P/_{<\hat g>}, \ T_P/_{<\hat f>}$  are compact.
\item  For every $\xi \in \Su \setminus \{r_1, \ldots r_p\}$ (resp. $\xi \in \Su \setminus \{a_1, \ldots, a_p\}$) if we denote $a_i$ (resp. $a_i$) to be the point such that $P^{n}(\xi)\to a_i$ (resp. $P^{-n}(\xi)\to r_i$) then, there exists a neighborhood $U_\xi$ of $\xi$ in $\mt$ contained in the basin of attraction of $a_i$ (resp. basin of repulsion of $r_i$).  
\end{enumerate}
\end{prop} 

\begin{proof}
Item (i) follows directly from Proposition \ref{prop.superattracting}
because it  proves that the basins of attraction and repulsion
of each point in $\{ a_1, ..., a_p, r_1, ..., r_p \}$
are open, non empty  sets in $\mt$.

\vskip .05in
Next we prove item (iv). 
Fix $L$ in $\widetilde \cF$.
Let $\xi$ not one of the $r_i$. 
Then $\xi$ is in the basin of attraction of some attracting
point under $P_\infty$, assume without loss of 
generality it is $a_1$. 
Let $U = U_I$ be a neighborhood of $a_1$ constructed as in
Proposition Lemma \ref{prop.superattracting}.
Let $i > 0$ so that $P^{n}_\infty(\xi)$ is in the
interior of $I$.
Then as in the proof of Proposition \ref{prop.superattracting},
there is $J$ open interval containing $P^i(\xi)$ so that 
$E^{2b_0+b_1}_J \subset U$ for all $E$ in $[L,\hat g(L)]$.
We can construct a set $U_J$ in $[L,\hat g(L)]$
as in Proposition \ref{prop.superattracting}
so that $U_J \cap E \subset U$ for all $E$ in $[L,\hat g(L)]$.
Then iterate by $\hat g$ to produce $U_J$. It is a neighborhood
of $P^i(\xi)$ which is contained in the basis of attraction
of $a_1$. Then $P^{-i}(U_J)$ is the desired neighborhood
of $\xi$ contained in the basis of attraction of $a_1$.
This proves (iv).

\vskip .05in
To obtain item (iii) we do the following.
Let $\xi$ in $\Su$ not one of $\{ r_i \}$.
There is $a_1$ an attracting point of $P_{\infty}$ so that
$\xi$ is in the basis of attraction of $a_1$ under $P_{\infty}$.
Fix a neighborhood $U_I$ of $a_1$ contained in the basis of
attraction of $a_1$ under $P$ as provided in item (iv). 
There is $i$ so that $P^i_{\infty}(\xi) \in I$.
Fix $L$ in $\Ft$,
let $\tau = \Theta_L(\xi)$. 
The above shows that $P^i(\tau)$ is an interior point
of $\Theta_{P^i(L)}(I)$. In particular there is a
neighborhood $V$ of $\tau$ in $L \cup S^1(L)$ so
that $P^i(V \cap L) \subset (U_I \cap P^i(L))$.
This is because of the definition of the neighborhoods
$U_I$.
Similarly there is $j$ so that $P^j(\xi)$ is 
in $J$ where $U_J$ is contained in 
a repelling neighborhood of a repelling
point $r_1$ of $P_{\infty}$.
Both of these facts together imply that $L \cap T_P$
is compact for any $L$ in $\Ft$.
In the argument above one can take the neighborhood
in $\bigcup_{E \in Z} (E \cup S^1(E))$, where $Z$ is any
compact interval in the leaf space with $L$ in the interior.
This shows that $T_P/_{<\hat g>}, T_P/_{<\hat f>}$ are compact.

\vskip .05in
Finally we prove item (ii).
Fix $L$ in $\widetilde \cF$.
Fix a union $V$ of  neighborhoods of the 
points $r_1, ..., r_p$, so that $P^{-1}(\overline V) \subset V$.
For any $n > 0$, the set
$$A_n \ = \ (T_P^+ - V) \cap P^{-n}(L)$$

\noindent  is non empty.
Otherwise the basins of attraction of different $a_i, a_j$
intersect, which is impossible.
Choose $x_n$ in $A_n$ and let $y_n = P^n(x_n)$ which is in
$L$. In addition $y_n$ is not in $V$ and $y_n$ is in $T_P^+$.
Therefore $y_n$ is in a compact set of $L$. Take a
subsequence $y_{n_i}$ converging to $y$ in $L$.
If $y$ is not in $T_P^-$ then there is $n_0 > 0$ so
that $P^{-n_0}(y)$ is in $V$, so a neighborhood
$W$ of $y$ so that $P^{-n}(W) \subset V$ for any $n > n_0$.
Assume all $y_{n_i}$ are in $W$.
But this contradicts that $P^{-n_i}(y_{n_i}) = x_{n_i}$ are
never in $V$.

This contradiction shows that $y$ is in $T_P^-$.
Since $x_n$ is in $T_P^+$ then $y_n$ is also in $T_P^+$
and $y$ is in $T_P^+$. It follows that $y$ is in $T_P$
so $T_P \cap L \not = \emptyset$.
This finishes the proof of (ii).

This finishes the proof of the proposition.
\end{proof}

\begin{figure}[ht]
	\begin{center}
		\begin{overpic}[scale=0.83]{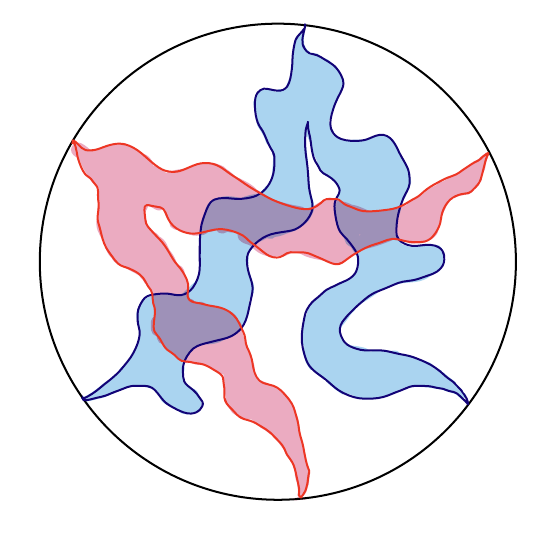}
		\end{overpic}
	\end{center}
	\begin{picture}(0,0)

\put(-16,205){\blue{$T_P^-$}}
\put(-57,180){\red{$T_P^+$}}
\put(-40,100){$T_P$}
\end{picture}

	\vspace{-0.5cm}
	\caption{{\small The core of the pA pair.}}\label{f.core}
\end{figure}
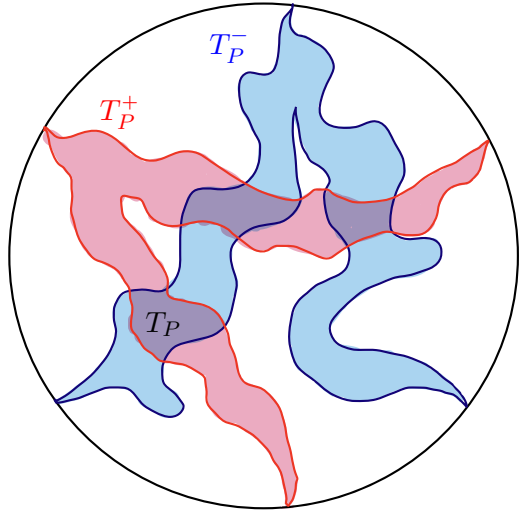

Since for pA pairs (and appropriate choices
of integers $n, m$ with $P = \hat g^m \circ \hat f^n$)
all fixed of $P_{\infty}$ in $\Su$ are super attracting and super repelling, we will dispose the use of the word super when it is clear that we are considering a pA pair and call the points attracting and repelling. 

\begin{addendum}\label{add-pA}
In the setting of Proposition \ref{prop-dynamicspA} we have the following:
for every family of attracting neighborhoods $U_{a_i}$ of the attracting points $a_i$ and $L$ in $\Ft$, 
there is some $R>0$ such that outside a ball of radius $R$ in $E$ we have that $T_P^- \cap E$ is contained in those neighborhoods, for
any $E$ in $[L, \hat g(L)]$. The symmetric statement holds for the repelling points and $T_P^+$. In particular, $T_P \cap L$ is contained in a ball of radius $R$ inside $L$. 
\end{addendum}

The second statement is obvious because the quotient $T_P/_{<\hat g>}$ is compact. 
To obtain the first statement: for each individual $E$ in $[L,\gt(L)]$ this is true for some
$R(L)$ by item (iv) of the previous proposition. Then since $[L,\gt(L)]$ 
is a compact interval of leaves, the result follows.

An argument very similar to the proof of item (ii) of the previous
proposition yields the following result which will be useful
in the future.
The map $\hat g$ acts freely and properly discontinuously in $\mt$,
hence $\mt /_{< \hat g>}$ is a manifold $N$ (cf. Remark \ref{rem.quotient}). The foliation $\Ft$
induces a foliation $\cF_N$ in $N$, whose leaves are homeomorphic
to planes and the leaf space of $\cF_N$ is the circle.
Let $\pi_N: \mt \rightarrow N$ be the projection map.
We say that a sequence $x_n$ in $\mt$ {\em converges to} $T_P$
if $\pi_N(x_n)$ converges to $T_P/_{<\hat g>}$, in the sense that for
any neighborhood $Z$ of $T_P/_{<\hat g>}$ in $N$ then $\pi_N(x_n)$ 
is eventually in $Z$.

\begin{lemma} \label{lem.auxiliary}
Under the hypothesis of Proposition \ref{prop.superattracting}
let $y$ in $T_P^+$. Then $P^n(y)$ converges to $T_P$ as $n \rightarrow 
\infty$.
\end{lemma}

\begin{proof}{}
We use the setup in the proof of item (ii) of the previous proposition.
In particular let $V$ be a union of neighborhoods of the repelling
points of $P_{\infty}$ so that $P^{-1}(\overline V) \subset V$.

Let $P_N$ be 
the induced map  by $P$ in $N$. 
Let $z = \pi_N(y)$.
Assume that $P^n(y)$ does not converge to $T_P$. 
Then there is a neighborhood $Z$ of $T_P/\hat g$ and
$n_i \rightarrow \infty$,  with $P_N^{n_i}(z)$ always not
in $Z$. 
There is $n_0 > 0$ so that if $n > n_0$ then $P^n(y)$ is
not in $V$, hence $P_N^n(z)$ is not in $V/_{<\hat g>}$.
By the addendum, $P_N^n(z)$ is in a compact set in $N$
for $n > n_0$. Hence up to another subsequence we may
assume that $P_N^{n_i}(z)$ converges to $z_0$. 
Notice that $z_0$ is not in $T_P/_{<\hat g>}$.

Lift $z_0$ to $x_0$ in $\mt$. Then $x_0$ is not in $T_P$, 
but since $y$ is in $T_P^+$, it follows that $x_0$ is in $T_P^+$,
so it follows that $x_0$ is not in $T_P^-$.

Hence as in the proof of item (ii) of the previous proposition
there is a  neighborhood $W$ 
of $x_0$ and $j_0$ integer so that if $j \leq  j_0$ then $P^{j}(w)$ is in
$V$  for any $w$ in $W$.
For any $i$ big $P_N^{n_i}(z)$ is in $W/_{<\hat g>}$, hence 
$P_N^{n_i + j_0}(z)$ is in $V/_{<\hat g>}$. This contradicts
the fact that $P_N^n(z)$ is not in $V/_{<\hat g>}$ for $n > n_0$.
This contradiction finishes the proof.
\end{proof}

\subsection{Abundance of pseudo-Anosov pairs}\label{ss.abundancepApairs}

In this section we specialize to the cases described in examples \ref{example1} and \ref{example2}: That is, we say that $(f,\cF)$ verifies the \emph{commuting property} if $f: M \to M$ is a diffeomorphism preserving an $\R$-covered uniform foliation $\cF$ by hyperbolic leaves and if one of the following conditions holds:

\begin{enumerate}
\item There is a lift $\fh$ to $\mt$ which commutes with all deck transformations\footnote{In particular, if $f$ is homotopic to the identity.} and does not fix any leaf of $\Ft$.
\item There is a deck transformation $\gamma$ which commutes with all deck transformations \footnote{In particular, $M$ is Seifert with hyperbolic base (because the leaves of $\cF$ are hyperbolic) and $\gamma$ corresponds to the center of $\pi_1(M)$ generated by the element corresponding to the fibers. This in particular implies that $\cF$ is horizontal, and so $\Su$ identifies with the boundary of the universal cover of the base.} and does not fix any leaf of $\Ft$.
\end{enumerate}

The assumption that $(f,\cF)$ has the commuting property 
implies on the one hand that it admits good pairs of the form $(\fh, \gamma)$ with $\fh$ a lift of $f$ and $\gamma\in \pi_1(M)$ a deck transformation (i.e. a lift of $\mathrm{id}:M \to M$) and on the other that one can construct new good pairs out of others. 

\begin{defi} \label{defi.admis}
A good pair $(\fh,\gamma)$ for a $(f,\cF)$ with the commuting property will be said to be \emph{admissible} if either $\fh$ commutes with all deck transformations, or $\gamma$ is in the center of $\pi_1(M)$.  
\end{defi}

In the first case, the good pair property is verified since $\fh$ acts as the identity on $\Su$, and in the latter, it is $\gamma$ that acts as the identity on $\Su$. 
In both cases this happens because if a map commutes with
all deck transformations then this map is 
a bounded distance from the identity in $\mt$
(see also \cite{BFFP-2}).

\begin{defi}\label{defi.conjugate}
Let $(f,\cF)$ have the commuting property and let $(\fh, \gamma)$ be an admissible good pair. We say that $(\fh', \gamma')$ is \emph{conjugate} to $(\fh,\gamma)$ by $\eta \in \pi_1(M)$ if we have that\footnote{Note that by the commuting property we have that either $\eta^{-1} \circ \fh \circ \eta = \fh$ or $\eta^{-1} \circ \gamma
\circ \eta = \gamma$.}

$$(\fh',\gamma')=(\eta^{-1} \circ \fh \circ \eta, \eta^{-1} \circ \gamma
\circ \eta).$$
\end{defi}

Note that if $(\fh,\gamma)$ is a pA pair (resp. regular pA pair) then every good pair conjugate to $(\fh,\gamma)$ also is a pA pair (resp. regular pA pair). 

The following result shows that there are plenty of pA-pairs with good properties. In our specific settings, we could obtain this directly, but here we give a unified proof.

\begin{prop}\label{prop-manypApairs}
Let $(f,\cF)$ with the commuting property, $(\fh, \gamma)$ be an admissible pA-pair and let $J \subset \Su$ be an open interval. Then, there exists an admissible pA-pair $(\fh',\gamma')$ conjugate to $(\fh,\gamma)$ such that it has all its fixed points in the interior of $J$.
\end{prop}

\begin{proof}
We will apply \cite[Lemma 5.4]{FPmin} stating that for every pair of disjoint open sets $U$ and $V$ in $\Su$ there is a deck transformation $\eta$ such that the action of $\eta$ on $\Su$ maps the complement of $U$ in the interior of $V$. Denote by $P= \gamma^m \circ \fh^n$ with $n,m$ not both
equal to $0$ and denote $P_\infty$ the action on $\Su$. 

Now, pick a open interval $U$ disjoint from all fixed points of $P_\infty$ and a deck transformation $\eta$ which maps the complement of $U$ inside $J$. 

Now, if one considers the map $\eta \circ P_\infty \circ \eta^{-1}$ it follows that it has all its fixed points inside $J$. Since the map is conjugated by a deck transformation, it follows that the points are super attracting/repelling. 
This is because deck translations acts in a H\"{o}lder way
on $\Su$.

Since $(f,\cF)$ has the commuting property, then either $(\fh, \eta \circ \gamma \circ \eta^{-1}$) or $(\eta \circ \fh \circ \eta^{-1}, \gamma)$ make an admissible pair for $(f, \cF)$. 
\end{proof}

\section{Pseudo-Anosov pairs and sub-foliations}\label{ss.pAandsubfoliations}

We will assume that there is a one-dimensional branching foliation $\cT$ subfoliating $\cF$ which is $f$-invariant (recall the definition at the end of \S \ref{ss.bran}). 
Denote by $\widetilde \cT$ the lift of $\cT$ to the universal cover. Recall from the previous section that whenever $(\fh,\gamma)$ is a pseudo-Anosov pair we will take $P$ to be some $P=\gamma^m \circ \fh^n$ which has a finite number of fixed points alternatingly super attracting and super repelling in $\Su$. If we do not choose explicitely the values of $n,m$ it will mean that any choice with this property will work. 

\subsection{Landing points}\label{ss.landing}

Given a leaf $c \in \widetilde \cT$ we say that $\hat c$ is a \emph{ray} of $c$ if it is a connected component of $c \setminus \{x\}$ for some $x \in c$. Since the leaves of
$\widetilde \cT$ are properly embedded in leaves of
$\widetilde \cF$ 
then every ray of $c \subset L \in \Ft$ accumulates only
in some connected subset of $S^1(L)$.  

Using the dynamics of pseudo-Anosov pairs one deduces the following simple proposition that we will use several times in the paper.  
In this result we 
we use the foliation
$\cF_N$ in $N = \mt /_{<\hat g>}$.

\begin{prop}\label{prop-nointerval}
Let $(\fh,\gh)$ be a pseudo-Anosov pair and let $\hat c$ be a 
ray in a leaf of $\widetilde \cT$
which accumulates in an interval $J \subset S^1(L)$. Then, $J$ contains
the $\Theta_L$ images of at most two fixed points of $P_\infty$ in $\Su$. 
\end{prop}

\begin{proof}
We assume that $\hat g$ acts as the identity on $\Su$.
If the interval $J$ intersects three such points we can assume without loss of generality that two of them (we call them $a_1, a_2$, points
in $\Su$) are attracting while one (called $r$) is repelling and between the two attracting ones there are no other fixed point of $P_{\infty}$.

Fix neighborhoods $U_{a_1}$ and $U_{a_2}$ of $a_1$ and $a_2$ in the respective basins of attraction given by Proposition \ref{prop.superattracting}. For those neighborhoods there is a sequence $\ell_1, \ldots, \ell_k, \ldots$ of arcs of $\hat c$ joining the neighborhoods $U_{a_1}$ and $U_{a_2}$. 
Fix a repelling neighborhood $V_r$ of the form $U_I$ of $r$
so that $I$ has endpoints $y_1, y_2$. 
We assume that $y_i$ is in the interval $(a_i,r)$ of $\Su$.
Then $P^i_\infty(y_i)$ converges to $a_i$ as $i \rightarrow \infty$.
Let $b_i = \Theta_L(y_i)$. 
For $i$ big $\ell_i$ has a subsegment $e_i$ in $V_r \cap L$ 
connecting a point very near $b_1$ in $L \cup S^1(L)$ 
to a point very near $b_2$ in $L \cup S^1(L)$.
These points are in the basins of attraction of $a_1, a_2$
respectively.
We claim that $e_i$ intersects $T^+_P$. If not then 
$e_i$ is contained in the union of basis of attraction 
of attracting points, but the endpoints are contained 
in distinct basis of attraction. Since the basis of
attraction are open sets this contradicts the connectedness
of $e_i$. 

We consider the manifold $N = \mt /_{<\hat g>}$ as in Lemma \ref{lem.auxiliary} (see also Remark \ref{rem.quotient}).
Let $\cT_N$ be the foliation induced by $\widetilde \cT$
in $N$.
As in Lemma \ref{lem.auxiliary} consider a fixed neighborhood
$Z$ of $T_P/_{<\hat g>}$ in $N$, but now with compact closure.
Cover the closure of 
$Z$ by finitely many foliated boxes of $\cF_N$ and $\cT_N$
all with compact support. Since the leaves of $\cF_N$ are planes,
and $\cF_N$ is a fibration over the circle the
following happens: we can choose the foliated boxes
small enough so that
a leaf of 
$\cT_N$ can only intersect each of these foliated boxes in a single
component.

Since $\ell_i$ intersects $T^+_P$,
Lemma \ref{lem.auxiliary} implies that there is some $k_i$ such that if $k>k_i$ the map $P^k$ will then map the arc $\ell_i$ to a curve intersecting 
$\pi_N^{-1}(Z)$. 

We can apply this several times 
to all the arcs, we find a sufficiently large number
of subarcs of a large iterate of $\hat c$ that when projected to $N$ 
they all intersect $Z$.
If there are sufficiently many, then more than two have to
intersect a product box of $\cT_N$ in $Z$.

This contradicts the fact that each curve of $\widetilde \cT_N$
can only 
intersect a local product box as above in a unique connected component. 

This produces a contradiction and proves the proposition. 
\end{proof}

\begin{figure}[ht]
	\begin{center}
		\begin{overpic}[scale=0.83]{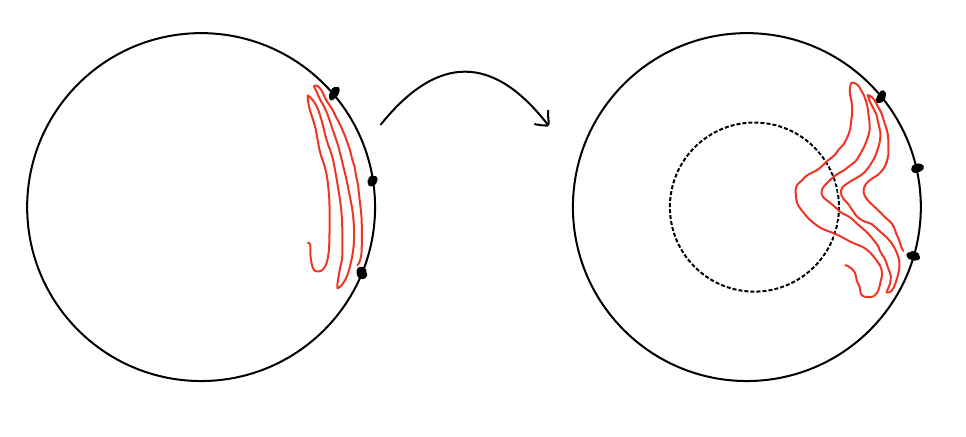}
		\end{overpic}
	\end{center}
	\vspace{-0.5cm}
	\caption{Proof of landing. The iterates $P^n$ push
the leaves away from the middle point and into a compact
part (when projected to $N$).}\label{f.land}
\end{figure}

{\bf {Throughout 
the remainder of this section we will consider $(f,\cF)$ with the commuting property (cf. \S \ref{ss.abundancepApairs}). }}

We will assume that there is a one-dimensional branching foliation $\cT$ subfoliating $\cF$ which is $f$-invariant (recall the definition at the end of \S \ref{ss.bran}). 
We say that a ray $\hat c$ of a curve $c \in \wT$ \emph{lands in a point} if there exists $\hat c_\infty \in S^1(L)$ so that the closure of $\hat c$ in $L \cup S^1(L)$ is $\hat c \cup \hat c_\infty$. 
In other words the ray $\hat c$ accumulates only on
$\hat c_\infty$.

Using the previous proposition, we deduce the following: 

\begin{prop}\label{prop-landing}
If $f: M \to M$ preserves an $\RR$-covered and uniform foliation by hyperbolic leaves $\cF$ which is subfoliated by a one-dimensional foliation $\cT$ and admits a pseudo-Anosov pair $(\hat f, \gamma)$ then every ray of $\wT$ lands in a point. 
\end{prop}
\begin{proof}
It is enough to show that given a pseudo-Anosov pair $(\hat f, \gamma)$ and an interval $J \subset \Su$ there exists another pseudo-Anosov pair so that $J$
contains three or more fixed points of it, in order to apply Proposition \ref{prop-nointerval}. But the fact we need follows from Proposition \ref{prop-manypApairs}. 
\end{proof}

\begin{notation}\label{not-orient} 
In $\mt$ we can orient leaves of $\wT$, and we will fix an orientation. 
Given a leaf $\ell \in \wT$ in a leaf $L \in \Ft$ both of whose rays land in a point we denote by $\ell^+$ and $\ell^-$ in $S^1(L)$ the landing points of the positive and negative ray (with respect to the orientation and a given point $x \in \ell$ which is not relevant for the definition of $\ell^\pm$). 
\end{notation}

\subsection{Pseudo-Anosov pairs with periodic leaves}\label{ss.pAperiodic}

We let $(\fh, \gamma)$ be a pA-pair and we will assume that: 

\begin{enumerate}
\item There is a leaf $L \in \Ft$ which is fixed by $P= \gamma^m \circ \fh^n$ for some $m \neq 0$ and $n>0$. 
\item The action of $P_\infty$ in $\Su$ has $2p$ fixed points which are alternatively super attracting and super repelling (with $p\geq 2$). 
\end{enumerate}

Let $c \in \wT \cap L$ be a leaf which is fixed by $P$ and $x \in c$. Write $c= c_1 \cup \{x\} \cup c_2$ where $c_1$ and $c_2$ are the
two connected rays of $c$ defined by $x$.  
Suppose that $c_1$ has ideal point $c^+$ in $S^1(L)$
and $c_2$ has ideal point $c^-$.
We say that $c_1$ is \emph{coarsely expanding} (resp. \emph{coarsely contracting}) if there is a compact interval $\cI$ of $c$ such that for every compact interval $\cJ$ of $c_1 \cup \{x\}$ there is $k>0$ such that $P^{-k}(\cJ) \subset \cI$ (resp. $P^k(\cJ) \subset \cI$).  These rays already played a prominent role in the arguments of \cite{BFFP-3}. The next result should be compared with the results in \cite[\S 11.2]{BFFP-3}. 

We can show: 

\begin{prop}\label{prop.landingpointsperiodic}
Given a center curve $c \in \wT \cap L$ which is fixed by $P$, then
$\Theta_L^{-1}(c^+), \ \Theta_L^{-1}(c^-)$ are
fixed by $P_\infty$ in $\Su$. Moreover, if $\Theta_L^{-1}(c^+)$
is an attracting (resp. repelling) point of $P_\infty$ in $\Su$ then the ray $c_1$ is coarsely expanding (resp. contracting). 
\end{prop}

\begin{proof}
This is direct from Proposition \ref{prop.superattracting}.
See also \cite[\S 11.2]{BFFP-3}. 
\end{proof}

We now  give a definition that we will use several times since we will be able to establish this strong property in the partially hyperbolic setting: 

\begin{defi}\label{defi.periodicproperty}
A pair $(f, \cF)$ has the \emph{periodic commuting property} if it has the commuting property (cf. \S \ref{ss.abundancepApairs})  and for every admissible pA-pair $(\fh, \gamma)$ for $(f,\cF)$ there exists $k>0$, $m \in \ZZ \setminus \{0\}$ and a leaf $L \in \Ft$ which is fixed by $P=\gamma^m \circ \fh^k$. 
\end{defi}

\begin{notation}\label{not-lift}
Whenever $(f, \cF)$ has the periodic commuting property and $(\fh,\gamma)$ is an admissible pA-pair, the lift $P$ will denote a lift $P=\gamma^m \circ \fh^k$ with $m \in \ZZ \setminus \{0\}$ and $k>0$ so that $P$ fixes some leaf $L \in \Ft$ and such that $P_\infty$ acting on $\Su$ has fixed points, all of
which are either super attracting or super repelling.
\end{notation}

\begin{prop}\label{prop.periodiccommuting} 
If $(f,\cF)$ has the periodic commuting property and $(\fh, \gamma)$ is any admissible pA-pair, then for every $L \in \Ft$ one has that $P^{m}(L)$ converges as $m \to \pm \infty$ to a leaf which is fixed by $P$. 
\end{prop}

\begin{proof}
The hypothesis implies that $P$ fixes a leaf $E$ of $\widetilde \cF$.
But then it also fixes $\gamma^i(E)$ for any $i \in \ZZ$.
For any $L$ leaf of $\widetilde \cF$ it is contained
in $[\gamma^i E, \gamma^{i+1} E]$ for some $i$ in $Z$ which implies
the result.
\end{proof}

The following property will be used several times:

\begin{figure}[ht] 
	\begin{center}
		\begin{overpic}[scale=0.88]{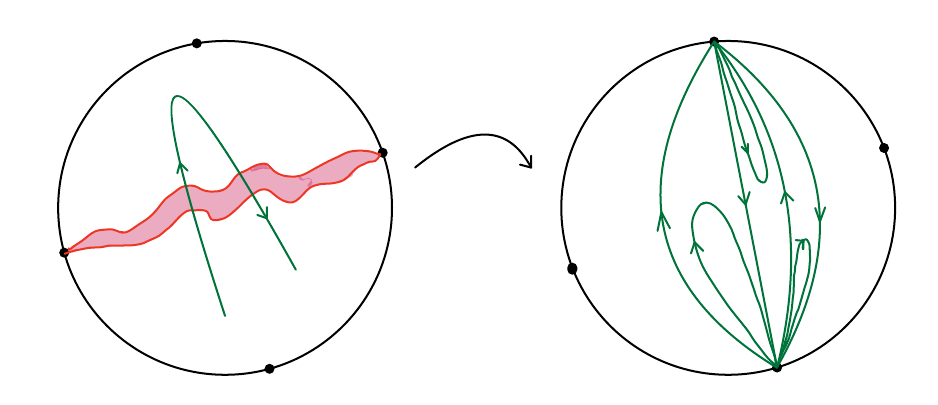}
		\end{overpic}
	\end{center}
	\vspace{-0.5cm}
	\caption{A configuration.}\label{f.config}
\end{figure}

\begin{prop}\label{p-getconfig} 
Let $(f, \cF)$ have the periodic commuting property preserving a one dimensional branching foliation $\cT$ that subfoliates $\cF$. Let $(\fh,\gamma)$ be an admissible regular pA-pair with attracting points $a_1, a_2$ for a lift $P=\gamma^m \circ \fh^n$ (cf. Notation \ref{not-lift}). Assume that there is a leaf $\ell$ of $\wT \cap L$ which has a segment $\cI \subset \ell$ with the property that both extreme points of $\cI$ belong to the basin of attraction of $a_1$ and so that $\cI$ intersects the basin of attraction of $a_2$. Then, there is  a leaf $E \in \Ft$ fixed by $P$ which has at least two disjoint leaves $\ell_1$ and $\ell_2$ from $\wT$ joining 
$\Theta_E(a_1)$ and $\Theta_E(a_2)$ and $\ell_1, \ell_2$
fixed by $P$. In particular, $\ell_1$ and $\ell_2$ are both coarsely expanding for $P$. 
\end{prop}

\begin{proof}
Let $E$ be the limit of $P^k(L)$ as $k \rightarrow \infty$
given by Proposition \ref{prop.periodiccommuting} which is fixed by $P$. 

Call $x_1$ and $x_2$ the endpoints of $\cI$ so that $\cI$ is oriented from $x_1$ to $x_2$. Take $y \in \cI$ which belongs to the basin of attraction of $a_2$. It follows that $\cI = \cI_1 \cup \cI_2$ where $\cI_1$ is the segment oriented from $x_1$ to $y$ and $\cI_2$ the segment oriented from $y$ to $x_2$. 

Let $r_1, r_2$ be the other fixed points of $P_\infty$, which
are both repelling, and $a_1, r_1, a_2, r_2$ circularly
ordered in $\Su$.
Let $V_1, V_2$ be neighborhoods of type $U_I$ of 
$r_1, r_2$ respectively so that $P^{-1}(\bar V_i) 
\subset V_i$, given by Proposition \ref{prop.superattracting}.
Let $I_i$ be the interval of $\Su$ defined by
$V_i$ and containing $r_i$.

For $k > 0$ big enough $P^k(\cI_i)$ cannot intersect $V_1$ or $V_2$,
therefore 
the sequence $(P^k(\cI_1))$ cannot escape compact sets in $\mt$ 
as $k \rightarrow \infty$.
This is because $P^k(L)$ converges to $E$ and $P^k(\cI_i)$
intersects the basis of attraction of both $a_1$ and $a_2$.
Hence the sequence $(P^k(\cI_1))$  converges to some family of leaves 
in $\wT$ in $E$. The leaves in the limit have must land. The set of
landing points of all such limit leaves is invariant
under $P$ (since $E$ is invariant under $P$).
In addition the set of landing points of these limit leaves cannot
intersect $\Theta_E(I_1)$ or $\Theta_E(I_2)$ because 
$I_1, I_2$ are expanding intervals
under the action of $P_\infty$.
Therefore the only possible limit points of the landing leaves
must be $\Theta_E(a_1)$ and $\Theta_E(a_2)$.
See Figure \ref{f.config} for a depiction of this situation.

Since $P^k(\cI_1)$ has endpoints in neighborhoods of $a_1$ and
$a_2$ there must be a limit leaf 
which has $a_1$ and $a_2$ as landing points (as opposed to both
landing points being $a_1$ or $a_2$).
In addition this  leaf is oriented going from $a_1$ to $a_2$. 
Similarly in the limit of $(P^k(\cI_2))$ there must be at least one leaf of $\wT$ oriented from $a_2$ to $a_1$. The family of such limit leaves  is closed, ordered, and avoids neighborhoods of the repellers of $P$; so we can consider the two outermost of them and these must be fixed by $P$. Moreover, since they are oriented in a different direction, these leaves cannot be close and therefore are disjoint.
\end{proof}

Indeed we get a further property: 

\begin{addendum}\label{ad.config} 
In the setting of Proposition \ref{p-getconfig} we further obtain that in $E$ there is at least one  leaf $\ell_3$ of
$\widetilde \cT$ in $E$ between $\ell_1$ and $\ell_2$ so that both endpoints coincide with either $a_1$ or $a_2$. 
\end{addendum} 

\begin{proof}
Consider the leaves $\ell_1, \ell_2$ in $E$ obtained
in Proposition \ref{p-getconfig}.
These leaves form the boundary of an infinite band
$B$ in $E$ which accumulates only in $a_1$ and $a_2$. Therefore
any leaf of $\widetilde \cT$ contained in $B$ can only
accumulate in these two points. Assume by way of contradiction
that no such leaf has both endpoints $a_1$ or both $a_2$.
Then every leaf has one ideal point in $a_1$ and one in $a_2$.
As a consequence
the set of leaves of $\wT \cap E$ between $\ell_1$ and $\ell_2$ 
has an order making it order isomorphic to an interval.
Moreover, each such leaf has an orientation either from $a_1$ to $a_2$ or from $a_2$ to $a_1$. Since the orientations of $\ell_1$ and $\ell_2$ differ, this is a contradiction.
Therefore it cannot happen that every leaf in between 
$\ell_1$ and $\ell_2$ has different endpoints. 
\end{proof}

\subsection{Shadows and visual measure} 
Let $(f,\cF)$ with the periodic commuting property (Definition \ref{defi.periodicproperty}) preserving a one dimensional branching foliation $\cT$ that subfoliates $\cF$.

In some cases it is possible to control the visual measure of a center arc from a point in a leaf $L' \in \Ft$ if one can exclude certain configurations. For a point $x \in L' \in \Ft$ and a subset $X \subset L'$ we call the \emph{shadow} of $X$ from $x$ to the set of points in $S^1(L') \cong T^1_xL'$ corresponding to geodesic rays from $x$ intersecting $X$.

\begin{figure}[ht]
	\begin{center}
		\begin{overpic}[scale=0.843]{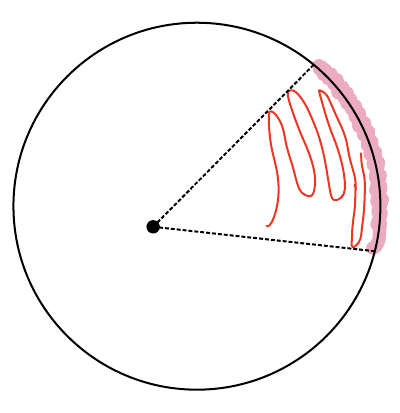}
		\end{overpic}
	\end{center}
	\vspace{-0.5cm}
	\caption{{\small The shadow.}}\label{f.shadow}
\end{figure}

\begin{defi}\label{defi.smallvisual}
We say that $\cT$ has \emph{small visual measure} in $\cF$ if for every $\eps>0$ there is $R>0$ such that if  $x \in L \in \Ft$ and $I$ is a segment of $\wT \cap L$ at distance larger than $R$ from $x$, then the shadow of $I$ from $x$ has visual measure smaller than $\eps$ in $S^1(L) \cong T^1_xL$ (cf. \S \ref{ss.boundaryvisual}). 
\end{defi}

The distance condition means that $d_L(x,y) > R$ for any $y$ in $I$.

\begin{prop}\label{prop.visualabstract}
Assume that there is a regular admissible pA-pair for $(f,\cF)$ and that $\cT$ does not have small visual measure in $\cF$. Then, there is a regular pA-pair $(\fh,\gamma)$ and a leaf $L \in \Ft$ fixed by $P$ (cf. Notation \ref{not-lift}) which has at least two disjoint leaves of $\wT$ each fixed by
$P$ and whose landing points are $\Theta_L$ images of distinct
attracting fixed points of $P_\infty$ in $\Su$.
In particular, these leaves of $\wT$ are both coarsely expanding for $P$.  
\end{prop}

\begin{proof}
By assumption there is $\epsilon > 0$, and there are
points $x_n$ in leaves $L_n \in \Ft$ such that there are segments $I_n$ of leaves of $\wT \cap L_n$ at distance bigger than $n$ from $x_n$ and whose shadow in $S^{1}(L_n)$ has visual measure larger than $\eps$. 
Deck transformations act as isometries on leaves of 
$\widetilde \cF$. Hence  up to applying deck transformations and
a subsequence we can assume that $x_n$ converges to a point $x_0$.
We can assume that $L_n$ converges to $L$ (notice $\cF$ is a 
branching foliation so a priori all $L_n$ could contain $x_0$).

Let $J_n$ be the shadow of $I_n$ on $S^1(L_n)$.
Up to another subsequence 
we can also assume that the intervals $J_n$ in $S^1(L_n)$ converge to an interval $J_\infty$ of visual measure larger than $\eps$ in $S^1(L)$.
We can assume without loss of generality that $J_\infty \neq S^1(L)$
by taking shorter segments $I_n$. 

Using Proposition \ref{prop-manypApairs} we can consider $(\fh,\gamma)$ an admissible regular pA-pair such that it has all of its fixed points in the interior of $J = \Theta^{-1}(J_{\infty})$.
Call the attracting points $a_1, a_2$ and the repelling ones $r_1, r_2$. Since $J \neq \Su$ we can order these points in $\Su$
up to renumbering so that $a_2$ is inside the segment $J' \subset J$ whose endpoints are $r_1$ and $r_2$. Consider neighborhoods $U_{a_i}$ and $V_{r_i}$ of each in $\mt$ as in Proposition \ref{prop.superattracting}. 

For large enough $n$ we have that the arcs $I_n$ contain subarcs $S_n \subset I_n$ joining $V_{r_1}$ with $V_{r_2}$ and intersecting $U_{a_2}$. We can also assume that $S_n$ are such that for some points $\xi_1, \xi_2$ in each connected component of $J \setminus J'$ the segment $S_n$ intersects the neighborhoods $U_{\xi_i}$ as in Proposition \ref{prop-dynamicspA} ($iv$). 
In particular these points in $U_{\xi_i}$ are in the basis
of attraction of $a_1$ (the other attracting point).
Denote as $S_n^1$ and $S_n^2$ two segments of $S_n$ joining respectively $U_{a_2}$ with $U_{\xi_1}$ and $U_{\xi_2}$. 

Now the result follows from Proposition \ref{p-getconfig}. 
\end{proof}

\begin{remark}
Note that the proposition admits a symmetric statement
since it can be applied to $(\fh^{-1}, \gamma)$ which is a regular pA-pair for $f^{-1}$ which also preserves $\cF$ and $\cT$. So, under those assumptions there also exist a fixed leaf of $\Ft$ with two distinct leaves of $\cT$ being fixed and coarsely contracting. Disjointness of the curves is important since we do not assume that $\cT$ is a true foliation. This will allow us to rule out such behavior for centers in the partially hyperbolic setting. 
\end{remark}

\section{Pseudo-Anosov pairs and partially hyperbolic foliations}\label{s.pApairsandPH}
In this section $f: M \to M$ will be a partially hyperbolic diffeomorphism preserving two transverse branching foliations $\cs$ and $\cu$. We denote by $\cW^{s}$ and $\cW^{u}$ the strong stable and strong unstable foliations respectively, and by $\cW^{c}$ the center (branching) foliation. 
We will assume that at least one of $\cs$ or $\cu$ is 
$\R$-covered and uniform and that some lift $\hat f$ acts
as a translation on this leaf space.
Many results will be stated for $\cs$ but obviously work
equally well for $\cu$.

\subsection{Periodic leaves for pseudo-Anosov pairs}\label{ss.pApairsperiodic}

Here we restate a result from \cite{BFFP-3, BFFP-4} in the context of pseudo Anosov pairs. 

\begin{prop}\label{prop.periodic}
Assume that $(\fh,\gamma)$ is a pA pair for the foliation $\cs$. Then, there exists $n>0$, $m \in \ZZ_{\neq 0}$ and a leaf $L \in \wcs$ such that $\gamma^m \circ \fh^n(L) = L$. 
\end{prop}

\begin{proof}
Under these conditions we proved in Proposition \ref{prop-dynamicspA}
(ii) that the set $T_P$ is non empty. 
The quotient of $T_P$ in $\mt/\gamma$ $-$ that is, $T_P/\gamma$
is compact.
Since $\gamma$ is a deck transformation,
the map $\hat f$ projects to a map, which we
denote by $f_0$,  in $\mt/\gamma$ which is 
partially hyperbolic and preserves the compact set $T_P/\gamma$.

Let $z \in T_P/\gamma$ and let $y \in T_P/\gamma$ be an accumulation
point of $(f^n_0(z))$. Take $i, j$ big enough, with $j$ much 
bigger than $i$, such that $f^i_0(z)$ and $f^j_0(z)$ are both
very close to $y$.

Consider a small closed unstable segment $\tau$
containing $f^i_0(z)$ in
its interior. Since $f_0$ increases unstable
lengths uniformly, then if $j$ is big enough,
every leaf of $\wcs/\gamma$  intersecting
$\tau$ intersects the interior of $f^{j-i}_0(\tau)$.
This set of leaves of $\wcs/\gamma$ is an interval. This
produces a fixed $\wcs/\gamma$ leaf under
$f^{j-i}_0$. Lifting to $\mt$ proves the Proposition.
\end{proof}

This is the same proof as in \cite[Proposition 10.3]{BFFP-3} (which itself uses 
\cite[Proposition 9.1]{BFFP-2}) or \cite[Proposition 4.1]{BFFP-4}.

\begin{remark}\label{rem.incoherence}
Notice that once one has this, one immediately deduces that  $\cs$ cannot be a true foliation (cf. \cite[Theorem B]{BFFP-2} and \cite[\S 5]{BFFP-4}). 
This is related with the fact that partially hyperbolic diffeomorphisms having pA pairs with respect to the $\cs$ or $\cu$ foliation cannot be \emph{dynamically coherent} and will force that the map $h$ in the definition of \emph{collapsed Anosov} flow is not a homeomorphism. 
\end{remark}

As a consequence of Proposition \ref{prop.periodic} we deduce immediately that:

\begin{cor}\label{cor.periodic}
If $(f,\cs)$ has the commuting property and has an admissible pA-pair, then it has the periodic commuting property.  
\end{cor}

\begin{remark}\label{remark-phinleaf}
Note that if $(\fh, \gamma)$ is a pA pair and $P=\gamma^m \circ \fh^n$ with $n>0$, the map $P$ is a lift of a positive iterate of $f$  therefore is partially hyperbolic and the invariant bundles are exactly the lifts of those of $f$ in $M$ to $\mt$ (the stable switches with the unstable if we take $n<0$). 
\end{remark}

\begin{remark}\label{rem.novikov}
Note that both $\wc$ and $\ws$ are one dimensional (branching) subfoliations of $\cs$. By construction, it holds that $\wc$ is also a subfoliation of $\cu$ (which is also a branching foliation) and therefore we know that in $\mt$ we have that a curve of $\wws$ cannot intersect the same leaf of $\wwc$ twice. 
\end{remark}

\subsection{Visual measure and distance of curves to geodesics} 

Here we show the following result which has validity beyond the context we are working in this paper as it does not require a full set of pA pairs
(defined later). 

\begin{teo}\label{teo.curvesgeneral}
Let $f: M \to M$ be a partially hyperbolic diffeomorphism preserving branching foliations $\cs$ and $\cu$ so that $(f,\cs)$ has the commuting property (see subsection \ref{ss.abundancepApairs}). Assume moreover that there is an admissible regular pA pair for  $(f,\cs)$.  Then it follows that both $\wc$ and $\ws$ have small visual measure in $\cs$ (cf. Definition \ref{defi.smallvisual}). Moreover, there is $R>0$ such that given a center leaf $\ell \in \wwc$ (resp. a stable leaf $\ell \in \wws$) in $L \in \wcs$ if we denote by $\hat \ell$ a segment a or ray of $\ell$ whose landing is either $\ell^-$ or 
$\ell^+ \in L \cup S^1(L)$ then the geodesic segment or geodesic ray $\hat r$ of $L$ 
joining either the endpoints of $\hat \ell$ or the starting 
point of $\hat \ell$ with its landing point 
is contained in the $R$-neighborhood in $L$ of $\hat \ell$. 
\end{teo}

\begin{remark}\label{rem-horocycles}
It is important to mention what this Theorem does not say. In particular, it does not ensure that the ray $\hat \ell$ is contained in a bounded neigborhood of the geodesic ray (in particular, it does not say that $\hat \ell$ is a quasigeodesic). Later, we will use this result to show that under some more assumptions, all center curves are quasigeodesics. This cannot hold for stable curves as there may be some stable curves which have both endpoints being the same (see eg. \cite{BGHP}). However, the fact that the strong stables have small visual measure is something quite remarkable as they can be made to have tangent vectors arbitrarily close to horocycles (see \cite{BGHP}).  
\end{remark}

Note first that the fact that curves from $\wwc$ and $\wws$ land in leaves of $\wcs$ is direct from Proposition \ref{prop-landing}. To show that the visual measure of the arcs, rays or shadows is small we will use the following result about center curves that will also be useful later: 

\begin{lema}\label{lem.imposibleconf} 
Let $f$ be a partially hyperbolic diffeomorphism preserving branching foliations $\cs$ and $\cu$ so that $(f, \cs)$ has the commuting property, and there is an admissible regular pA pair $(\fh,\gamma)$.  Let $P$ as in Notation \ref{not-lift} with $P(L) = L$ for some $L \in \wcs$. Then, there cannot be two disjoint center curves $c_1$ and $c_2$ of $\wwc$ in $L$ which are fixed by $P$ and join the $\Theta_L$ images of
distinct attracting fixed points of $P_{\infty}$ in $\Su$.
\end{lema}

\begin{proof}
Such center curves should be coarsely expanding by $P$ by Proposition \ref{prop.landingpointsperiodic}. This forces $P$ to have at least one fixed point $x$ in $c_1$. We look at $s(x)$ the stable manifold of $x$. It cannot intersect $c_2$ since both $s(x)$ and $c_2$ are invariant by $P$
and so is their intersection, which is a single point $y$. See Figure \ref{f.impos}.
Since $c_2, c_1$ are disjoint $y$ would be a fixed point of $P$ in $s(x)$ different from $x$ $-$ impossible, since $s(x)$ is a stable leaf
and $P$ is contracting in stables. Then, the ray of $s(x)$ in the connected component of $L \setminus c_1$ containing $c_2$ must land in an attracting point of $P$ in $\Su$. This again is impossible since $s(x)$ is coarsely contracting (cf. Remark \ref{remark-phinleaf}),  and this contradicts Proposition \ref{prop.landingpointsperiodic}. 
\end{proof}

\begin{figure}[ht]
	\begin{center}
		\begin{overpic}[scale=0.873]{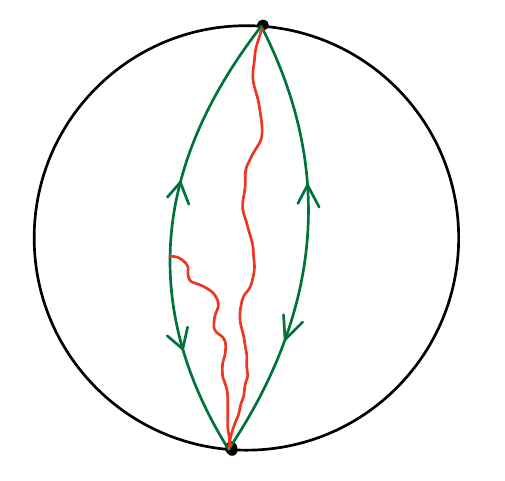}
		\end{overpic}
	\end{center}
	\begin{picture}(0,0)
	\put(-49,127){\color{ForestGreen}{$c_1$}}
	\put(-42,111){$x$}
	\put(-10,110){\red{$s$}}
	\put(30, 100){\color{ForestGreen}{$c_2$}}
	\end{picture}
	\vspace{-0.5cm}
	\caption{{\small Proof of Lemma \ref{lem.imposibleconf}.}}\label{f.impos}
\end{figure}

We complete the proof of Theorem \ref{teo.curvesgeneral}
by showing that if geodesics joining points of centers or stables do not remain boundedly close to the respective curves in leaves, then one can construct arcs with shadows with large visual measure: 

\begin{lema}\label{lem.geodesic}
Let $\cF$ be an $\RR$-covered uniform foliation with hyperbolic leaves of a closed 3-manifold $M$ and let $\cT$ be a one dimensional branching subfoliation of $\cF$. Assume that for every $n>0$ there is a segment $\ell_n$ of a leaf of $\cT$  such that the geodesic segment joining the endpoints of $\ell_n$ is not contained in the ball of radius $n$ of the segment $\ell_n$. Then, $\cT$ does not have small visual measure in $\cF$. 
\end{lema}

\begin{proof}
Just consider the segments $\ell_n \subset L_n$ and the corresponding geodesic segment $r_n \subset L_n$ joining the endpoints. By assumption, we know that there is a point $x_n \in r_n$ at distance larger than $n$ from $\ell_n$, or equivalently, that $B_{L_n}(x_n, n) \cap \ell_n = \emptyset$. 

Since the shadow of $\ell_n$ from $x_n$ is connected and $\ell_n$ intersects both sides of $r_n$ we know that the shadow of $\ell_n$ through $x_n$ has at least half of the visual measure from the point $x_n$ while it is completely outside the ball of radius $n$ around $x_n$. This implies that $\cT$ cannot have small visual measure in $\cF$. 
\end{proof}

Now we can complete the proof of Theorem \ref{teo.curvesgeneral}. 

\begin{proof}[Proof of Theorem  \ref{teo.curvesgeneral}]
The statement about visual measure in the case of $\ws$ follows
by appliyng Proposition \ref{prop.visualabstract} using
$\cT$ as the stable foliation..
The statement follows because
strong stable leaves cannot be coarsely expanding
under $P$, if $P = \gamma^m \circ \fh^n$ with $n > 0$.

To show that this is also the case for $\wc$ we again
apply Proposition \ref{prop.visualabstract} using
$\cT$ as $\wc$.  If centers did not have small visual measure in $\cs$, it follows that there is a regular pA-pair $(\fh,\gamma)$ associated to $(f,\cs)$ and we can find two disjoint leaves $c_1, c_2 \in \wwc$ contained in a leaf $L$ which are fixed by $P$ as well as $c_1, c_2$. Now, Lemma \ref{lem.imposibleconf} gives a contradiction. 

 The statement about rays or segments of the leaves in the foliations follows from Lemma \ref{lem.geodesic}. The statement about segments is strictly contained in that Lemma. To get the result for rays it is enough to approximate the ray by longer and longer segments which all have the same property. 
\end{proof}

\subsection{Impossible configurations}\label{ss.pApairsconfig}
We show that some configurations of the foliations
in leaves of $\wcs$ (or $\wcu$)
are impossible and this will be used to show that the leaf space of $\wwc$ is Hausdorff inside leaves of $\wcs$.  The next proposition will combine well with  Lemma \ref{lem.imposibleconf} (which together with Proposition \ref{p-getconfig} gives other impossible configurations).  We note that the next result works for a single pA pair with certain properties and does not need to have the full set of pA pairs that will be used in next section. In fact, we will need to deal with a case slightly more general than a pA pair which is when there are only two fixed points in $\Su$ one super attracting and one super repelling. 

\begin{prop}\label{prop.centerssamepoint} 
Let $f: M \to M$ be a partially hyperbolic diffeomorphism preserving a branching foliation $\cs$ which is $\RR$-covered and uniform with hyperbolic leaves. 
Suppose that $(f, \cs)$ admits an  admissible regular
pseudo-Anosov pair.
Let $(\fh,\gamma)$ be a good pair
pair and $P=\gamma^m \circ \fh^n$ with $n>0$ so that $P_\infty$ has fixed points in $\Su$ and such that all fixed points are either super attracting or super repelling. Let $L \in \wcs$ be a leaf fixed by $P$. If there is a center curve $c$ in $\wwc \cap L$ with endpoints $c^+$ and $c^-$ in $S^1(L)$ such that $c^+=c^-$ then one must have that 
$\Theta_L^{-1}(c^+)$ cannot be an attracting fixed 
point of $P_{\infty}$.  
\end{prop}

\begin{proof} 
We stress that we do not assume that $(\hat f, \gamma)$ is
a pA-pair. In particular $P_{\infty}$ may have only two fixed points
in $\Su$.

Since $P(L) = L$ we can reduce the proof to $\hat L = L \cup S^1(L)$.
The map $P$ induces a homeomorphism of $\hat L$. A point $\xi$ in $\Su$
is a fixed point, attracting or repelling point of $P_{\infty}$ 
if and only if
$\Theta_L(\xi)$ is a fixed point, attracting or repelling point
of $P$ in $\hat L$.
We will prove the result by contradiction assuming that $c^+=c^- = a$
is an attracting fixed point of $P$ in $\hat L$.
We denote by $D(c)$ to the connected component of $L \setminus c$ whose closure in $\hat L$ intersects $S^1(L)$ only in 
$a$.

Note that such a center cannot be fixed by $P$. If it were the case, then it would be coarsely expanding by Proposition \ref{prop.superattracting} and therefore there would be a fixed point $x \in c$ by $P$. Let $s(x)$ be
the stable leaf through $x$. The ray of $s(x)$ intersecting $D(c)$ must land in $a=c^+=c^-$ since a strong stable cannot intersect a center curve twice (cf. Remark \ref{rem.novikov}) and therefore the ray is completely contained in $D(c)$ and lands in $c^+$. That forces that stable curve to be coarsely expanding by Proposition \ref{prop.superattracting} which is impossible. Compare with Lemma \ref{lem.imposibleconf}.  
In fact the same argument shows that this center cannot be
periodic under $P$ as well.

Up to taking the square of $P$ we assume that $P$ 
preserves orientation when acting on $L$, and hence  also 
on $S^1(L)$.

Consider now the iterates $c_k:=P^{k}(c)$ with $k\in \ZZ$. Denote by $D(c_k) = P^{k}(D(c))$ which is the connected component of $L \setminus c_k$ whose closure in $\hat L$ intersects $S^1(L)$ only in $a=c^+=c^-$.

Consider $\cD =\overline{\bigcup_{k} D(c_k)}$. Note that $\cD$ is 
a $P$ invariant, closed set.
Let $\cC$ be the set of center leaves which make
up the boundary of $\cD$.

In order to prove the proposition we establish  
some general claims. The first one is the place where we use that $a$ is attracting
for $P$. If it were repelling there would be no
a priori contradiction\footnote{Indeed, this behavior can happen for the strong stable/unstable foliations of some partially hyperbolic diffeomorphisms such as the ones constructed in \cite{BGHP}.}.

\begin{af}\label{af.1}
There cannot be a fixed point of $P$ in $\cD$. 
\end{af}
\begin{proof}
Let $x \in \cD$ be fixed by $P$. If $x \in D(c_k)$ for some $k$ it follows that one of the rays of $s(x)$ has to land in $c^+$ which is a contradiction. Otherwise, $x$ is accumulated by the curves $c_k$, therefore, for large enough $k$ we have that one ray of $s(x)$ intersects $c_k$ and therefore enters in $D(c_k)$ and must land in $c^+$, a contradiction. This completes the proof. 
\end{proof}

To continue the proof of
Proposition \ref{prop.centerssamepoint}
we distinguish two options (see Figure~\ref{f.options}):

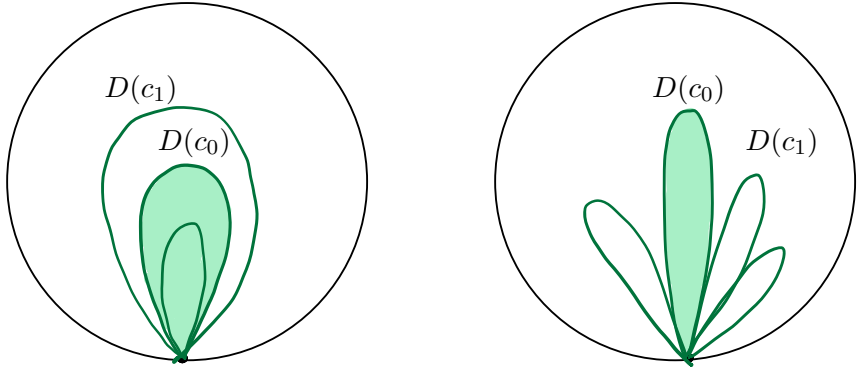
\begin{figure}[ht]
	\begin{center}
		\begin{overpic}[scale=0.73]{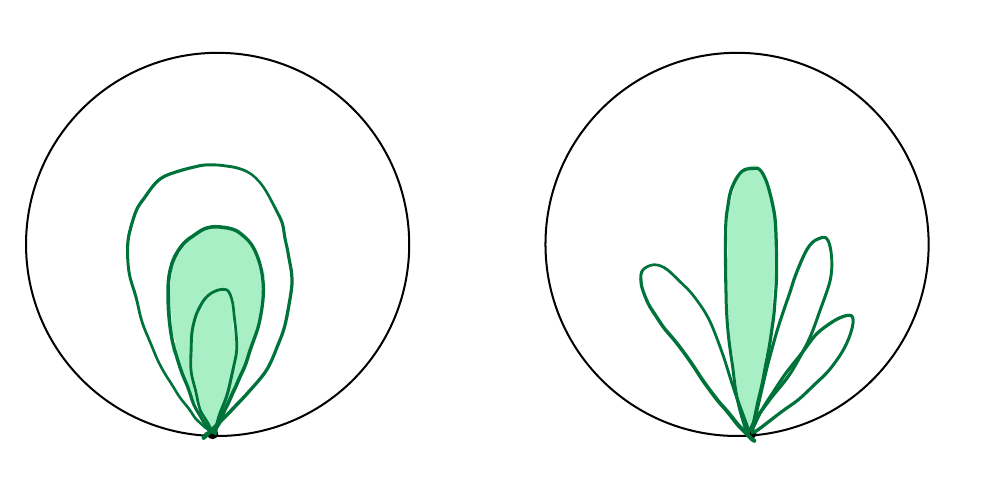}
		\end{overpic}
	\end{center}
	\begin{picture}(0,0)
	\put(80,135){$D(c_0)$}
	\put(-105,115){$D(c_0)$}
	\put(115,115){$D(c_1)$}
	\put(-125,135){$D(c_1)$}
	\end{picture}
	\vspace{-0.5cm}
	\caption{{\small Option (1) to the left and option (2) to the right.}}\label{f.options}
\end{figure}

\begin{itemize}
\item[(1)] The sets $\{ D(c_k)) \}$ are not pairwise
disjoint.
\item[(2)] The sets $\{ D(c_k)) \}$ are pairwise disjoint.
\end{itemize} 

In option (1) there is $i > 0$ so that $D(c_i)$ intersects
$D(c)$, hence either $D(c_i) \subset D(c)$ or $D(c) \subset D(c_i)$.
Hence up to taking a further positive iterate  of $P$
we assume that $D(c) \subset D(c_1)$ or $D(c_1) \subset D(c)$. 
Since it is a positive iterate, the point $a$ is still
super attracting for $P$.

\vskip .1in
\noindent
{\bf {Option (1)}}

We assume first that we are in option (1). 
This situation is by far the harder to deal with.
The overall strategy in this case is the following:
we find a stable leaf $s$ intersecting $\cD$ which
is fixed by $P$. This leaf  $s$ has one ideal point
in $a$ and this will contradict that $a$ is attracting
for $P$. 
To find such $s$ we essentially consider the set of all stables
intersecting $\cC$ plus the stable leaves ``in between".
We show that this set has a linear order, is invariant by
$P$ and $P$ fixes a leaf in this  set. The last step
is the hardest and depends on understanding the structure
of the boundary of $\cD$, how it interacts with the fixed
points of $P$ in $S^1(L)$. Notice that a priori the center
foliation in $L$ can be very complicated, so there are many
theoretical configurations. 
\footnote{The eventual goal, done in Section \ref{s.Hausdorff}
 is to prove
that the center foliation in $L$ is actually fairly simple,
that is, its leaf space is Hausdorff and homeomorphic
to  $\R$.}

Therefore, the first goal is to
obtain some useful properties about the boundary of $\cD$. 
There are similar properties in option (2), but not the same,
and option (2) is much easier to deal with.

\begin{af}\label{af.2}
The boundary $\partial \cD$ of $\cD$ in $L$ is a non empty set saturated by center curves. Every point in $\partial \cD$ belongs to a center curve which is a limit of subintervals of the curves $c_k$. Moreover, the collection $\cC$ of center leaves are pairwise non separated in the
center leaf space in $L$.
\end{af}

\begin{proof}
Assume first that $\cD=L$, then we have that there are compact arcs converging uniformly to some interval in $S^1(L)$ in the topology of $\hat L = L \cup S^1(L)$, these arcs would have large visual measure and escape to infinity contradicting Theorem \ref{teo.curvesgeneral}. Therefore $\partial \cD \neq \emptyset$. 

Since $\cD$ is saturated by center curves, then so is $\partial \cD$. Moreover, if $x \in \partial \cD$ belongs to a certain center leaf  $e \subset
\partial \cD$
then we have that every compact subinterval $I$ of $e$ must be accumulated by the sets $D(c_j)$ with $j \to +\infty$ or $j \to -\infty.$ This implies that there are arcs $I_j$ of $c_{k_j}$ converging uniformly to $I$. 

Finally let $e_1, e_2$  be two  distinct center leaves in $\cC$.
Non separated means that in the center leaf space they do
not have disjoint neighborhoods.
As above the leaves $e_1, e_2$ are contained
in the limit of $c_j$ with $j \to \infty$
or $j \to -\infty$.  Therefore $e_1, e_2$ are not separated
from each other.
\end{proof}

Since $P$ preserves the orientation in $L$ the following happens:,
if $e$ is in $\cC$ then either $P(e) = e$ or all iterates $P^n(e)$
are pairwise disjoint.
We can also show: 

\begin{af}\label{af.6}
Let $e \in \cC$ a center curve in $\partial \cD$ such that $P(e) \neq e$. Then $\{ P^n(e) \}$ cannot accumulate in a point in $L$ when $n \to \infty$
or $n \to -\infty$. 
\end{af}
\begin{proof}
Consider $G_e$ to be the connected component of $L \setminus e$ which is disjoint from $P(e)$. Since the curves in $\cC$ are 
pairwise non-separated, and $P$  preserves
orientation in $L$, we know that $P^n(G_e)$ are all disjoint. Assuming that $\{ P^n(e) \}$ accumulates in
some point $x \in L$ with $n \to \pm \infty$,
we can fix a local product structure neighborhood around $x$ for the center foliation and we can see $\{ P^n(G_e) \}$ accumulating in this point. Since these sets are all disjoint, then this means that the leaves
$\{ P^n(e) \}$ accumulate on a local product structure box in more than one connected component, which is impossible. 
\end{proof}

\begin{af}\label{af.5}
Let $e$ be a leaf in $\cC$. If $e$ 
has an ideal point $\xi$ which is a fixed point of $P$ then the other
ideal point $\nu$ of $e$ is different from $\xi$ and one of them
is attracting and one repelling.
In addition the ideal points of $e$ cannot be in
distinct complementary components of the set of fixed points
of $P$ in $S^1(L)$.
\end{af}

\begin{proof}
Suppose $e$ is a leaf in $\cC$ which has an
ideal point $\xi$ fixed by $P_\infty$. Let $\nu$ be the
the other ideal
point of $e$. 
 Suppose first that $\nu$ is distinct from $\xi$. Then since $P(e)$ is
non separated from $e$ and $c_k$ converges to both $e$ and $P(e)$ 
it follows that $P(e) = e$. 
If both endpoints of $e$ are either attracting or repelling for
$P$ then there is a fixed point of $P$ in $e$, hence a fixed
point of $P$ in $\cD$, disallowed in Claim \ref{af.1}.
So one of the ideal points of $e$ is
attracting and the other one is repelling.

Suppose now that $\nu =\xi$. 
Let $G_e$ be the component of $L - e$ which accumulates
only in $\xi$ in $S^1(L)$. Since $\xi$ must be either super attracting or super repelling, then using
Claim \ref{af.1}, we deduce that $P(e) \neq e$ and that all the iterates
$P^i(G_e)$ are all distinct. Moreover, by Claim \ref{af.6} we know that $P^i(e)$ cannot accumulate on a point $x \in L$ which implies that the sets $P^i(G_e)$ converge as $i \to \pm \infty$ to $\xi$. However, since $\xi$ is super attracting or super repelling it follows that $P^i(e)$ cannot converge to $\xi$ as $i \to -\infty$ or $i \to +\infty$. 
This proves the first assertion of the claim.

Finally suppose that $e$ has ideal points in two distinct
complementary components of fixed points of $P$ in $S^1(L)$.
Up to an iterate these complementary components are fixed by
$P$. Then since $e$ is a boundary leaf of $\cD$ this implies
that $P(e) = e$. In particular the ideal points of $e$ are
fixed by $P$ and are not in complementary of the set of fixed
points of $P$. 
This finishes the proof of the claim.
\end{proof}

We now define a set $\cS$
of stable leaves which will produce a $P$ invariant
stable leaf intersecting $\cD$.
The construction of $\cS$ is geometric and not dynamical.
For simplicity assume that $c_j$ converges to $\cC$ when
$j \to \infty$. The case when $c_j$ converges to $\cC$
when $j \to -\infty$ is entirely analogous and we address that later.

Let $\cS$ be the set of stable leaves $s$ in $L$ so that
there is $j_0 \in \Z$ so that $s$ intersects $c_j$ for any
$j \geq j_0$.
Each such stable leaf $s$  intersects some $c_j$. Since
$c_j$ has both ideal points equal to $a$, it follow that $s$
has a ray limiting on $a=c^+=c^-$. 
Each stable leaf $s$  intersecting $\cC$ intersects $c_j$ for
all $j \geq j_0$ (the $j_0$ depends on $s$), so $s$ is in $\cS$. 
In addition if $s_0, s_1$ are in $\cS$ then they intersect
$c_j$ for all $j \geq j_0$ for some $j_0$ 
(take a $j_0$ that works for both).
For any stable $s$ intersecting
$c_{j_0}$ between $s_0 \cap c_{j_0}$ and $s_1 \cap c_{j_0}$ 
then $s$ intersects
$c_j$ for any $j \geq  j_0$, so $s$ is also in $\cS$.
This is because $c_{j_0}, c_j$ and $s_0, s_1$ form a 
``quadrilateral" in $L$ and $s$ intersects $c_{j_0}$, hence
intersects $c_j$ also.
It follows that $s$ is in $\cS$.
Therefore
the set $\cS$ is linearly ordered. Since the subset of $\cS$  between
$s_0$ and $s_1$ is order isomorphic to an interval then
$\cS$ is order isomorphic to the reals. With the quotient topology
it is homeomorphic to the reals.
Since  we took a square of $P$ if necessary, then
the map $P$ preserves this order.

If on the other hand $c_j$ converges to $\cC$ when $j \to -\infty$,
then  in the definition of $\cS$ we require a $j_0$ so that
$s$ intersects all $c_j$ for $j \leq j_0$.

\vskip .1in

Let $I, J$ be the connected components of $S^1(L)$ minus 
the set of fixed
points of $P$ in $S^1(L)$, so that $I, J$ have
one endpoint equal to $a$.
We now define a subset $\cS_I$ of $\cS$ associated with $I$. 
The definition will depend on whether there is a leaf
of $\cC$ with an ideal point in $I$ or not.
Consider first the case that there is a leaf $e$ in $\cC$ with
an ideal point in $I$. By Claim \ref{af.5}
the other ideal point of $e$ is also in $I$.
Let $A_i$ be the set of stable leaves in $L$ intersecting 
$P^i(e)$. Let $\cS_I$ be the smallest connected set of $\cS$
containing all $A_i$.
In this case for any $s$ in $\cS_I$ then $s$ has an ideal point
in $I$.
In fact for any $s$ in $\cS$ it has an ideal point in $a$.
If the other ideal point $x$  of $s$ is in $I$ then $x$ is
in the closed segment contained in $I$ with endpoints
$P^i(z)$ and $P^{i+1}(z)$ where $z$ is an ideal point
of a stable intersecting $e$ and $i$ is some integer. Hence
$s$ separates two elements in $\cS$ and intersects $c_j$ for
all $j \geq j_0$ (for some $j_0$) hence $s$ is in $\cS$.
It follows that in this case $\cS_I$ is exactly the set of
stables intersecting $\cD$ which have one ideal point in $I$.
In particular the definition of $\cS_I$ is independent of
the particular leaf $e$ that we start with.

The other possibility is that 
there is no leaf of $\cC$ with an ideal point in 
$I$. We deal  with this case now.
Let $x$ in $I$. We claim that there is a neighborhood
$V$ of $x$ in $L \cup S^1(L)$ so that $V \cap L$ is disjoint
from $\cD$. Suppose not. If for  some such $V$ we have that
$V \cap L
\subset \cD$, then we get a sequence of arcs in $c_k$ (with 
$c_k$ limiting to $\cC$ as $k \to \infty$)
 so that they escape compact sets in $L$
and limit to $V \cap S^1(L)$. These arcs do not have small
visual measure, violating Theorem \ref{teo.curvesgeneral}.
So this cannot happen.
Choose ${ V_i }$ with $ i \in {\bf N}$ a basis neighborhood of
$x$ in $L \cup S^1(L)$, with $V_i \cap S^1(L)$
always contained in $I$. By assumption,
for each $i$ there is a point 
$y_i$ in $V_i \cap c_{k_i}$ for a suitable choice
of $k_i$.
If the $k_i$ can be chosen
constant equal to $k$ then $c_k$ has one ideal point
in $V_k \cap S^1(L)$. But this is impossible by hypothesis
as $V_k \cap S^1(L) \subset I$. So up to subsequence we can assume
that $k_i$ are pairwise distinct. Since the $c_{k_i}$ have both ideal
points outside of $I$ then either they escape compact sets,
contradicting Theorem \ref{teo.curvesgeneral}, or up to
subsequence they keep intersecting a fixed compact set.
This is impossible since the elements in $\cC$ are pairwise non
separated from each other, cf. Claim \ref{af.2}.
This shows that there is such $V$ as above, so
that $(V \cap L) \cap \cD  = \emptyset$. 
In this case let $e$ be the unique leaf of $\cC$ which separates
$V \cap L$ from the interior of
$\cD$. By $P$ invariance of $\cD$ and the fact that
no ideal point of $e$ is in $I$, it follows that $P(e) = e$.
In addition one ideal point of $e$ is $a$. This is because
$e$ separates $V \cap L$ from the interior of $\cD$ and the
interior of $\cD$ has points limiting to $a$.
Since $e$ is fixed by $P$ the other ideal point of $e$ is
a repelling fixed point of $P$ by Claim \ref{af.5},
and  hence it is not $a$. 
Let now $I_{ex}$ be the open  interval of $S^1(L)$ determined
by the ideal points of $e$ and which contains $I$.
We remark that $I_{ex}$ is disjoint from $J$.
In this case let $\cS_I$ be the set of stable leaves
intersecting $e$. 

Notice that $\cS_I$ is again a connected subset of $\cS$.
In either case we remark that if $s$ is a leaf in $\cS_I$
then no ideal point of $s$ is in $J$.

Notice that in either case $\cS_I$ is $P$ invariant.
In the same way we define a set $\cS_J$.

\begin{af}\label{af.7}
The sets $\cS_I$, $\cS_J$ are disjoint.
\end{af}

\begin{proof}
Suppose that there is a leaf $e$ in $\cC$ with an ideal
point in either $I$ or $J$. For simplicity assume an ideal
point in $I$. Then by construction for every leaf $s$ in $\cS_I$
it has an ideal point in $I$. Since no leaf in $\cS_J$ has
an ideal point in $I$, the claim is proved in this 
case.

The remaining case is that we have the intervals $I_{ex}$
and $J_{ex}$, which are defined by leaves $e, e_1$ in $\cC$.
In this case $\cS_I$ is the set of stable leaves intersecting
$e$ and $\cS_J$ is the set of stable leaves intersecting
$e_1$. Since $e, e_1$ are distinct but non  separated from each
other, no stable leaf can intersect both of them. Hence
again $\cS_I \cap \cS_J = \emptyset$.
This proves the claim.
\end{proof}

Since $\cS_I, \cS_J$ are disjoint, let $s$ be the stable
leaf in $\cS$ corresponding to the endpoint of $\cS_I$ separating
it from $\cS_J$ in $\cS$. Then since both $\cS_I, \cS_J$ are fixed
by $P$, so is $s$. 
Since $s$ is in $\cS$ then it intersects $c_j$ for some
$j$ and hence has an ideal point $a$.
This contradicts the
fact that $a$ is an attracting fixed point of $P$.

This finally finishes the proof of Proposition \ref{prop.centerssamepoint} assuming option (1).

\vskip .1in
\noindent
{\bf {Option (2)}}

We now assume option (2).

\begin{af}\label{af.4}
In option (2) we have that as $k \to \pm \infty$ then $D(c_k)$ can only accumulate in $a=c^+=c^-$. 
\end{af} 

\begin{proof}
In option (2) the sets $D(c_k)$ are pairwise disjoint.
An argument entirely analogous to that 
of Claim \ref{af.6} shows that $D(c_k)$ cannot accumulate
anywhere in $L$ as $k \to \infty$ or $k \to  -\infty$. 

If the collection $D(c_k)$ accumulates in another point of $S^1(L)$ besides $a$, then since it does not accumulate in $L$ it will have
subsegments which limit uniformly on non empty intervals 
of $S^1(L)$. In particular these segments escape compact sets in $L$.
This violates that the center foliation has small visual 
measure, Theorem \ref{teo.curvesgeneral}. This finishes the proof of the claim.
\end{proof}

Now we can complete the proof in option (2). As $k \rightarrow \pm \infty$,
the $D(c_k)$ cannot accumulate in $L$ or in any other
point of $S^1(L)$ besides $a$. Choose a neighborhood $U$
of $a$ which is contracting under $P$ as in 
Proposition \ref{prop.superattracting}. Choose $U$ sufficiently
small so that $D_0$ is not contained in $U$.
Then for $k$ big negative $D(c_k)$ is contained in $U$, and
applying $P^{-k}$ sends $D(c_k)$ inside of $U$,
but also to $D(c_0)$ not contained in $U$, contradiction. This completes the proof  of Proposition \ref{prop.centerssamepoint}. 
\end{proof}

\section{Hausdorff center leaf space}\label{s.Hausdorff}
In this section we show that under some assumptions the center leaf space of $\wwc$ has to be Hausdorff in leaves of $\wcs$. This is an important step in the proof of our main theorems and will use all the results on pseudo-Anosov
pairs we have been developing so far. 
We use the abbreviation pA pairs for pseudo-Anosov pairs.

To be able to exclude non-Hausdorff leaf space we will need enough pA pairs to be able to force certain configurations (see Remark \ref{rem.configurations} below). This will be defined precisely in \S \ref{ss.fullsetofpA}. 

After we rule out a certain configuration in \S \ref{ss.uniquelandingpt}, we will show in \S \ref{ss.proofHsdff} the following: 

\begin{teo}\label{teo.hsdff} 
Let $f: M \to M$ be a partially hyperbolic diffeomorphism preserving branching foliations $\cs$ and $\cu$. Assume that $(f, \cs)$  has full pseudo-Anosov behavior (cf. Definition \ref{defi.fullpA}). 
Assume also that $\cu$ is $\R$ covered.
Then, inside each leaf of $\wcs$, the foliation $\wwc$ by center curves has leaf space which is Hausdorff. 
\end{teo}

Recall that a one dimensional
 (branching) foliation $\wT$ in a complete simply connected surface $L$ has \emph{Hausdorff leaf space} if for every pair of curves of $\wT$ in $L$ it follows that the positive (closed) half space in $L$ determined by one of the curves is contained in the positive (closed) half space determined by the other. 

\begin{remark}
Definition \ref{defi.fullpA} is quite restrictive and asks for the existence of several pseudo-Anosov pairs for $f$. We suspect that the only examples which verify these assumptions are the ones we treat in Theorem \ref{teoB} and Theorem \ref{teoC}. We note however that we do not ask that the actions of the good pairs on the universal circle of $\cs$ and $\cu$ coincide (this is immediate in the context of Theorem \ref{teoC}, but not a priori obvious for Theorem \ref{teoB}). 
\end{remark}

\begin{remark}\label{rem.configurations}
Until now, all arguments used a given pA pair and then found sequences of curves that approached the universal circle in certain configurations that would ensure that some of their points belong to the basins of attracting/repulsion of the fixed points in $\Su$ of the pA pairs.
In this section the strategy will be different. We will fix a curve and find sequences of pA pairs whose configurations will force that the curve has some points in basins of attraction of different fixed points of the pA pairs (as the configuration required in Proposition \ref{p-getconfig}). Two delicate issues with this approach will appear:
\begin{itemize}
\item The curves we will consider already have limit points and approach the boundary very fast (cf. Theorem \ref{teo.curvesgeneral}). Therefore we need that the configuration of attracting/repelling points of the pA pairs are very special; 
\item Also, the core $T_P$ of a pA pair depends somewhat on the particular pA pair we choose, and that is why it will be important to consider pA pairs which are related to the same object in $M$ (i.e. different lifts of the same `orbit') so that we get some uniform estimates. 
\end{itemize}
For these reasons, we will need to restrict to a class of diffeomorphisms that we will later show contains the classes we are studying in this paper
to show Theorems \ref{teoB} and \ref{teoC}. 
\end{remark}

\subsection{Diffeomorphisms with a full set of pseudo-Anosov pairs}\label{ss.fullsetofpA} 

In some arguments we will need not only one pseudo-Anosov pair, but also that its conjugates (cf. Definition \ref{defi.conjugate}) fill the universal circle in a particular way. 

\begin{remark}
In what follows one should have in mind the difference between a pseudo-Anosov diffeomorphism of a surface and a reducible diffeomorphism of 
a surface with a pseudo-Anosov piece. One can also think about regulating flows for uniform foliations in atoroidal manifold versus manifolds with atoroidal pieces but non-trivial JSJ decomposition. (Recall Examples \ref{example1} and \ref{example2}.) When there is a unique pseudo-Anosov piece, the laminations are minimal, so every (regular) periodic orbit verifies that its stable/unstable manifold is dense in the stable/unstable lamination of the pseudo-Anosov map which forms a ``full lamination'' (see \cite{Calegari} and references therein). 
\end{remark}

We consider a pair $(f,\cF)$ with the periodic commuting property (cf. Definition \ref{defi.periodicproperty}) and let $(\fh, \gamma)$ be an admissible regular pA-pair with attracting points $a_1,a_2$ and repelling points $r_1,r_2$ in $\Su$ (here the action is with respect to a lift $P=\gamma^m \circ \fh^k$ as in Notation \ref{not-lift}). 

\begin{defi}\label{defi.fullpair} 
The regular pA-pair $(\hat f, \gamma)$ is a \emph{full pair} (for 
$P = \hat f^k \circ \gamma^m$), if there are $\alpha_0 > 0$, $d_0 > 0$
satisfying the following: for every geodesic ray $\eta$ 
in a leaf $L$ of $\Ft$, with starting point $x_0$
there is:
\begin{itemize}
\item a compact non degenerate interval $\cI$ in the leaf space of $\widetilde \cF$, 
\item and for every $n>0$, a deck transformation $\beta_n \in \pi_1(M)$ such that $\beta(L) \in \cI$ so that:
\end{itemize} 
if we denote by $g_n^a$ the geodesic in $L$  joining 
$\Theta_L(\beta_n a_1)$ with $\Theta_L(\beta_n a_2)$ and 
$g_n^r$ the geodesic  in $L$ joining $\Theta_L(\beta_n r_1)$ with $\Theta_L(\beta_n r_2)$ then:
\begin{itemize}
\item either $g_n^a$ or $g_n^r$ intersects $\eta$ in  a point $x_n$ making  angle larger than $\alpha_0$,
\item $d_L(x_0,x_n) > n$ and $d_L(x_n, g^a_n \cap g^r_n) < d_0$.
\end{itemize}
\end{defi}

We will use this property to obtain the following important result.

\begin{prop}\label{prop-fullpair}
Let $(f,\cF)$ have the periodic commuting property and admitting a full pair $(\fh,\gamma)$. Then, given a geodesic ray $\eta$ in a leaf  $L \in \Ft$ from a point $x_0 \in L$ with ideal point
$\Theta_L(\xi)$ ($\xi  \in \Su$), the following happens: There exists a conjugate pair $(\fh', \gamma')$ of $(\fh, \gamma)$ $-$ with $P'$ the corresponding conjugate of $P$
(cf. Definition \ref{defi.conjugate}) such that either $x_0$ and $\xi$ belong to different basins of attraction of the fixed points of $P'_\infty$ in $\Su$ or they belong to different basins of repulsion of the fixed points of $P'_\infty$. 
\end{prop}

\begin{proof}
Consider the geodesic ray $\eta$ in $L$ with starting $x_0 \in L$ and 
ideal point $\Theta_L(\xi)$. Let $a_1, a_2$ be the attracting points of $P_\infty$. Since $(\fh,\gamma)$ is a full pair, without loss of generality we can assume that there is a sequence $\gamma_n \in \pi_1(M)$ such that the geodesic $g_n$ in $L$ with ideal points $\Theta_L(\gamma_n a_1)$ 
and $\Theta_L(\gamma_n a_2)$ makes angle larger than $\alpha_0$ 
with $\eta$ and intersects $\eta$ in a point 
$x_n$ at distance larger than $n$ from $x_0$. 
In addition if $h_n$ is the geodesic in $L$ with ideal points
$\Theta_L(\gamma_n r_1)$ and $\Theta_L(\gamma_n r_2)$ then
$d_L(g_n \cap h_n,x_n) < d_0$.
Finally $\gamma_n(L)$ is in a fixed compact interval in the
leaf space of $\Ft$ for every $n$.

Let $P'_n = \gamma_n \circ P \circ \gamma_n^{-1}$.

Extend $\eta$ in $L$ beyond $x_0$ to a full geodesic
still denoted by $\eta$ and with other ideal point $\Theta_L(\nu)$.
Now we map back by $\gamma^{-1}_n$. The fact that $\gamma_n (L)$
is in a compact interval of the leaf space means that
the slithering distance of $\gamma_n$ is bounded \cite{Thurston}
and so is the slithering distance of $\gamma_n^{-1}$. In
other words $\gamma^{-1}_n(L)$ is in a compact interval,
which we denote by  $\cJ$.

Suppose that up to subsequence that one of
$\gamma^{-1}_n(\xi)$ or 
$\gamma^{-1}_n(\nu)$ converges to $a_1$ or $a_2$. Without loss of
generality assume that $\gamma_n^{-1}(\xi)$ converges to
$a_1$ or $a_2$.
Up to another subsequence assume that $\gamma^{-1}_n(L)$ converges
to $L_0$. 
Notice that $g_n \cap h_n$ is a globally bounded distance
from $T_{P'_n} \cap L$ $-$ because of the 
following: 
\begin{enumerate} 
\item $T_P/_{<\gamma>}$ is compact, and 
\item Since deck transformations are isometries of $\mt$ they
induce metrics in the quotients, and also 
$$T_P/_{<\gamma>}, \ \ \ T_{P'_n}/_{<\gamma_n \circ \gamma \circ 
\gamma_n^{-1}>}$$
\noindent
are isometric.
\end{enumerate} 

Since $d_L(g_n \cap h_n, x_n)$ is bounded by $d_0$, it follows
that 

$$d_{\gamma_n^{-1}(L)}(\gamma_n^{-1}(x_n), T_P \cap \gamma^{-1}_n(L))$$

\noindent
is bounded above. This 
is because $\gamma^{-1}_n(g_n), \gamma^{-1}_n(h_n)$ are geodesics
in $\gamma_n^{-1}(L)$ with ideal points in $S^1(\gamma^{-1}_n(L))$
which are $\Theta_{\gamma^{-1}_n(L)}(a_1),
\Theta_{\gamma^{-1}_n(L)}(a_2),
\Theta_{\gamma^{-1}_n(L)}(r_1)$ and 
$\Theta_{\gamma^{-1}_n(L)}(r_2)$, respectively.
Hence the intersection $\gamma^{-1}_n(g_n) \cap \gamma^{-1}_n(h_n)$ is
a bounded distance from $T_p$.

So we can assume $\gamma^{-1}_L(x_n)$ also converges
to $y_0$. 
The sequence of geodesics $\gamma^{-1}_n(g_n)$ in $\gamma_n^{-1}(L)$
have ideal points $\Theta_{\gamma^{-1}_n(L)}(a_1),
\Theta_{\gamma^{-1}_n(L)}(a_2)$,
so this sequence of geodesics
converges to the geodesic $g$ in $L_0$ with ideal
points $\Theta_{L_0}(a_1),
\Theta_{L_0}(a_2)$.
In addition the sequence $\gamma^{-1}_n(\eta)$ converges
to a geodesic in $L_0$ through $y_0$ and making an angle
with $g$ of at least $\alpha_0$ as specified
in the beginning of this proof. This is impossible 
since $\gamma^{-1}_n(\xi)$ converges to either $a_1$ or $a_2$.
Therefore none of $\gamma_n^{-1}(\xi), \gamma_n^{-1}(\nu)$
converges to either $a_1$ or $a_2$.

The previous paragraphs show that
there are fixed interval neighborhoods $I, J$
of the repellers $r_1, r_2$ not containing either $a_1, a_2$ in
their closures, so that $\gamma^{-1}_n(\xi),
\gamma^{-1}_n(\nu)$ are always in $I \cup J$ for
$n$ sufficiently big.
Let $U_I, U_J$ be neighbhorhoods of $r_1, r_2$ respectively
as constructed in Proposition \ref{prop.superattracting}.
Up to subsequence and without loss of generality
assume that $\gamma_n^{-1}(\xi)$ is in $U_I$ for $n$ big.
Since $d_L(x_0,x_n) > n$, and $\gamma_n^{-1}(g_n \cap h_n)$ is
in a fixed compact set in $\mt$,  then eventually $\gamma^{-1}_n(x_0)$
is in $U_J$.

This shows that $\xi, x_0$ belong to the basins of
repulsion of $\gamma^n(r_1), \gamma^n(r_2)$ respectively
under $P'$.

This proves the proposition.
\end{proof}

Now we are ready to give the definition we will use to get Theorem \ref{teo.hsdff}. 

\begin{defi}\label{defi.fullpA}
We say that $(f, \cF)$ has \emph{full pseudo-Anosov behavior} if $(f,\cF)$ has the periodic commuting property (cf. Definition \ref{defi.periodicproperty}) and:
\begin{enumerate}
\item every admissible good pair (cf Definition
\ref{defi.admis}) of $(f,\cF)$, up to iterate, has only super attracting and super repelling fixed points. and,
\item it contains a regular pA-pair which is a full pair (cf. Definition \ref{defi.fullpair}). 
\end{enumerate}
\end{defi}

\subsection{Distinct landing points}\label{ss.uniquelandingpt}
We first need the following auxiliary result. This is the place where we will use the full pseudo-Anosov behavior on one of the foliations.

\begin{prop}\label{prop.uniquelandcenter}
Let $f: M \to M$ be a partially hyperbolic diffeomorphism such that it preserves branching foliations $\cs$ and $\cu$ and such that $(f,\cs)$ 
has full pseudo-Anosov behavior. 
Assume also that $\cu$ is $\R$-covered.
Then, for every $c \in \wwc \cap L$ with $L \in \wcs$ we have that the endpoints $c^+$ and $c^-$ of $c$ in $S^1(L)$ are different.  
\end{prop}

The proof of this statement will require 
to first iterate in $\cs$ until we get a center curve both of whose endpoints land in a single fixed point of a pA pair for $(f,\cs)$. If the fixed point is super attracting one can apply Proposition \ref{prop.centerssamepoint} to get a contradiction. If the fixed point is repelling, we need to 
use Theorem \ref{teo.curvesgeneral} 
and an analysis of the center unstable foliation $\cu$ to 
derive a contradiction.

\begin{lema}\label{l.lema1}
Let $f$ be as in Proposition \ref{prop.uniquelandcenter} and assume there is a center curve $c \in \wwc \cap L$ for $L \in \wcs$ so that $c^-=c^+$. 
Then, there exists a regular pA-pair $(\fh, \gamma)$ such that $P$ as in Notation \ref{not-lift} fixes a leaf $L' \in \wcs$ which has a center curve $c'$ so that both endpoints coincide and so that $c'$ is fixed by $P$.
\end{lema}

\begin{proof}
Pick a point $x_0 \in c$ and consider the geodesic ray $\eta$ from $x_0$ to $\Theta_L(\xi) = c^-=c^+$. Since $(f, \cs)$ has full pseudo-Anosov behavior it has an admissible full pair $(\fh,\gamma)$ which is a regular pA pair. 

We will use Proposition \ref{prop-fullpair}.
By Proposition \ref{prop-fullpair} 
 we deduce that, up to conjugating $(\fh,\gamma)$, there is a $P$
so that $x_0$, $\xi$ are in different basis of attraction of either
$P_{\infty}$ or $P_{\infty}^{-1}$.

We will apply Proposition \ref{p-getconfig} with $\cT = \wc$
the center foliation.

We consider the case where the point $x_0 \in c$ is in the basin of repulsion of a repeller of $P$ and $\xi$ in the basin of repulsion of the other
repeller of $P$. Applying Proposition \ref{p-getconfig} to $P^{-1}$ we obtain that iterating by $P^{-n}$ the leaf $L$ converges to a leaf $L'$ fixed by $P$ and $P$ fixing disjoint center curves $c_1$ and $c_2$ which join the repelling points. (The other case is symmetric and obtains curves that join the attracting points of $P$ by iterating forward.) 

By Addendum \ref{ad.config} one can also see that there is a center curve $c_3$ in $L'$ between $c_1$ and $c_2$ and so that both
endpoints of $c_3$ in $L'$ are equal.
Since the curve is between $c_1$ and $c_2$ it follows that its forward iterates remain in a compact region, and so the curves converge to at least
one curve $c'$ whose endpoints coincide with the endpoints of $c_3$ and
which is fixed by $P$.

The symmetric case is dealt with using $P^{-1}$ instead of $P$.
\end{proof}

\begin{proof}[Proof of Proposition \ref{prop.uniquelandcenter}]
Let $c'$ given by the previous lemma.
By Proposition \ref{prop.centerssamepoint} we get that 
the endpoints of $c'$ must correspond to a super repelling point of $P_{\infty}$. 

Since $c'$ is fixed by $P$, there is a  leaf $F \in \wcu$ containing $c'$ and 
fixed by $P$.  Since the action of $P$ on $c'$ is coarsely 
contracting, there is a fixed point $x \in c$ by $P$. 
Let $c_1,c_2$ be the rays of $c'  \setminus \{x\}$.

We first consider the situation in $L$ and the foliation $\cs$.
Both $c_1, c_2$  limit to the same point $z = \Theta_L(\xi)$ in $S^1(L)$ with
$\xi$ super repelling for $P$. Theorem \ref{teo.curvesgeneral} implies that the geodesic ray in $L$ starting in $x$ and with ideal
point $z$ is contained in uniform neighborhoods in $L$ of $c_1$ and $c_2$. In particular, there are sequences $p_n \in c_1$ and $q^n \in c_2$ converging to $z$ in $L \cup S^1(L)$ so that $d_L(p_n,q_n)$ is bounded. 
It follows that $d_{\mt}(p_n, q_n)$ is bounded.

We now look at the center unstable foliation $\cu$, and use
that it is $\RR$-covered.
Since $\cu$ is $\RR$-covered, this implies that in $F$ the points $p_n$ and $q_n$ are a bounded distance apart. 
This is because $F$ is uniformly properly embedded in $\mt$.
Let $\ell_n$ be the geodesic segment in $F$ from $p_n$ to $q_n$.
By trimming $\ell_n$ or 
replacing $p_n, q_n$ if necessary, we assume that $\ell_n$
intersects $c'$ only in $p_n, q_n$, still keeping the length
of $\ell_n$ globally bounded. 

Since $\xi$ is super repelling for $P_{\infty}$ (acting on
the universal circle of $\cs$) it follows that $d_{\mt}(p_n,P(p_n))$
converges to infinity and likewise for $q_n$. The length of
$P(\ell_n)$ is uniformly bounded. In particular for $n$ big
enough $P(\ell_n)$ is disjoint from $\ell_n$. Fix one such $n$.
Let $D$ be disk in $F$ bounded by $\ell_n$ and the segment
in $c'$ from $p_n$ to $q_n$. By the above $P(D)$ is strictly contained
in $D$. There is a ray of the unstable leaf
of $x$ entering $D$. This ray intersects $\partial D$.
This ray is expanded by $P$.
This contradicts that $P(D)$ is a subset of $D$.

This finishes the proof of Proposition \ref{prop.uniquelandcenter}.
\end{proof}

\subsection{Proof of Theorem \ref{teo.hsdff}}\label{ss.proofHsdff}

The proof of Theorem \ref{teo.hsdff} is very similar to previous arguments. 

\begin{proof}
Assume by contradiction that there is a leaf $L \in \wcs$ on which 
the leaf space of $\wwc$ is non Hausdorff. Consider two center leaves $c,c' \in L$ which are non-separated in the sense that there is a sequence $c_n$ of center leaves such that $c_n$ converges both to $c$ and $c'$ (it may converge to other center leaves too). Up to changing orientation of the 
center foliation, we can assume that there are arcs $I_n$ of $c_n$ which approximate\footnote{By this we mean that the Haudsorff limit of the arcs $I_n$ in $\hat L$ contains both $c^+$ and $(c')^-$.} the points $c^+$ and $(c')^-$ in $\hat L = L \cup S^1(L)$ which may or may not coincide. 

By Proposition \ref{prop.uniquelandcenter}, we have
$c^- \neq c^+$ and also $c^+ \neq (c')^+$. Hence as done in the proof
of Proposition 
\ref{prop-fullpair} we can choose
an admissible regular pA pair $(\fh,\gamma)$ which verifies that in $L$ either the geodesic joining the attracting points or the repelling points separates $c^+$ from the ideal points of 
the center curves $c_n$. As before, we assume that it is the geodesic joining the attracting ideal points that makes the separation (as the other case is symmetric). 

We can assume by further conjugating $(\fh,\gamma)$ that it verifies that both ideal points of the curves $c_n$ belong to the same basin of repulsion of the repeller $r_1$ of $P_\infty$. Since $c^+$ belongs to the basin of repulsion of the other repeller $r_2$ it follows that for $n$ large enough, the segment $I_n$ intersects the basin of repulsion of $r_2$. We can then apply Proposition \ref{p-getconfig} to $P^{k}, k < 0$ to find a fixed center stable leaf $L' \in \Ft$ such that it contains two center curves which are disjoint and join the repelling points of $P_\infty$. Moreover, between these center curves there is a fixed center curve both of whose points coincide by Addendum \ref{ad.config}. This contradicts Proposition \ref{prop.uniquelandcenter} and concludes the proof of Theorem \ref{teo.hsdff}. 
\end{proof}

\section{Quasigeodesic behavior}\label{s.QG}
In this section we want to show that under our assumptions the centers behave as uniform quasigeodesics in leaves of center stable and center unstable branching foliations. 

Recall that for $k>1$, an embedded rectifiable 
curve $\ell \subset L$ in a complete Riemannian manifold is called a $k$-\emph{quasigeodesic} if one has that for all $x, y\in \ell$

$$ d_{\ell}(x,y) \leq k d_L(x,y)  + k. $$

\noindent where $d_L$ denotes the Riemannian distance in $L$ and $d_\ell$
is the length along $\ell$.
Note that $d_\ell (x,y) \geq d_L(x,y)$ always. 

\begin{defi} Let $\cF$ be a (branching) foliation in a closed 3-manifold $M$. We say that a  one dimensional branching subfoliation $\cT$  of $\cF$ 
by rectifiable curves is by \emph{uniform quasigeodesics} if there exists $k$ such that every curve $\ell$ of $\wT$ in $L \in \Ft$ with the induced path metric is a $k$-quasigeodesic. 
\end{defi}

\begin{remark}
The fact that a subfoliation is by uniform quasigeodesics is independent of the metric in $M$ since $M$ is compact. Only the constant may change. In our setting we typically work with (branching) foliations whose induced metric is negatively curved, where quasigeodesics have very meaningful geometric properties thanks to the classical Morse lemma (see \cite[\S III.H.1]{BH}). 
\end{remark}

Unless otherwise stated we always assume that the one dimensional
(branching) subfoliations are by rectifiable curves.

Now we can state the main result of this section: 

\begin{teo}\label{teo.QG}
Let $f: M \to M$ be a partially hyperbolic diffeomorphism preserving branching foliations $\cs$ and $\cu$ such that both $(f,\cs)$ and $(f,\cu)$ have full pseudo Anosov behavior (Definition \ref{defi.fullpA}). 
Suppose that there is $(\hat f, \gamma)$ a regular full pair 
for $(f, \cs)$ which is also a good pair for $(f, \cu)$.
Then $\wc$ is by uniform quasigeodesics in both $\cs$ and $\cu$. 
\end{teo}

Note that uniform quasigeodesics in leaves whose metric vary continuously can be followed in nearby leaves, so we deduce that: 

\begin{cor}\label{cor.continuousendpoints}
The endpoint maps $c \mapsto c^{\pm}$ from the leaf space $\cL^c$ of the center foliation $\wwc$ to $\Su$ is continuous and $\pi_1(M)$-equivariant. It is also $\fh$-equivariant for $\fh$ a lift of $f$ to $\mt$ (see Proposition \ref{prop.extuniv}). 
\end{cor}

We explain what we mean by $c^{\pm}$. Suppose that $c$ is contained
in a leaf $L$ of $\wcs$. Let $b$ be the ideal  point of $c$ in
the positive center direction. Then  $c^+ = \Theta_L^{-1}(b)$.
Similarly for $c^-$.

\subsection{Tracking geodesics} 
Consider $\cT$ a one-dimensional sub-branching foliation of a branching foliation $\cF$ of $M$. We assume that $\cF$ is $\RR$-covered and uniform and by hyperbolic leaves so that we can apply all what was developed in \S \ref{ss.Candel}-\S \ref{ss.universalS1}. When lifting to the universal cover we get a (branching) foliation $\Ft$ of $\mt$ which is subfoliated by a (branching) foliation $\wT$ and we choose an orientation for both. 

We will say that $\cT$ has \emph{efficient behavior} in $\cF$ if the following conditions hold:

\begin{enumerate}
\item the leaf space of $\wT$ in  each leaf $L \in \Ft$ is Hausdorff, 
\item each curve $\ell \in \wT$ in a leaf $L \in \Ft$ has well defined limit points $\ell^-$ and $\ell^+$ in $S^1(L)$ which are different
\item there is $R>0$ such that for each $\ell \in \wT$ and $x \in \ell \subset L \in \Ft$  if we denote by $r^\pm_x$ the geodesic ray in $L$ joining $x$ with $\ell^\pm$ then we have that $r^{\pm}_x$ is contained in the $R$-neighborhood in $L$ of the ray of $\ell$ from $x$ to $\ell^{\pm}$. 
\item $\cT$ has small visual measure (cf. Definition
\ref{defi.smallvisual}).
\end{enumerate}

\begin{remark}\label{rem.efficient}
In the previous sections we have established that if $f: M \to M$ is a partially hyperbolic diffeomorphism in the hypothesis of Theorem \ref{teo.QG} then the center (branching) foliation $\wc$ has efficient behavior in both $\cs$ and $\cu$: Point ($i$) is done in \S \ref{s.Hausdorff}, point ($ii$) in  \S \ref{ss.landing} and Proposition \ref{prop.uniquelandcenter} and 
points ($iii$), ($iv$) in Theorem \ref{teo.curvesgeneral}. 
\end{remark}

The following will be an auxiliary result to show the quasigeodesic behavior.

\begin{lema}\label{lem.cont}
Let $\cT$ having efficient behavior in $\cF$ and let $c_n \in \wT$ be a sequence of leaves. Assume that $c_n \subset L_n \in \Ft$ so that $L_n \to L$, $c_n \to c \in L$ and there are points $x_n \in c_n$ so that $x_n \to x \in c$. Assume there is a point $y_n$ in the ray of $c_n \setminus \{x_n\}$ with positive orientation so that $y_n \to \xi \in \Su$
(as in Definition \ref{not-neigh}). Then, the endpoint $c^+$ of the positively oriented ray of $c\setminus \{x\}$ is $\xi$ is $\Theta_L(\xi)$. 
\end{lema}

\begin{proof}
Suppose this is not true, let $c_n, x_n, y_n$ failing
this condition. Up to subsequence assume that 
$y_n$ converges to $\xi$, with $c_+ \not = \Theta_L(\xi)$.
Let $\nu \in \Su$ with $\Theta_L(\nu) = c^+$.

Since $c^+ = \Theta_L(\nu)$ and $c_n$ converges to $c$ in
the center leaf space, then $c_n$ also has points $z_n$
between $x_n$ and $y_n$ so that $z_n$ converges to $\Theta_L(\nu)$.
Consider the segments $J_n$ in $c_n$ from $z_n$ to $y_n$.
These segments do not have visual measure converging to zero
as $n \rightarrow \infty$, because $\Theta_L(\nu) \not = 
\Theta_L(\xi)$. Since $\cT$ has small visual measure it
follows that these segments cannot escape compact sets $\mt$
and converge to a collection of center leaves in $L$.
Let  $\ell$ be such a center leaf.
In particular there are $w_n$ in $J_n$ converging to $w$ in $\ell$.
If $c = \ell$, then the the local product
structure of foliations (in the center leaf space)
shows that the length of segments in $c_n$ from $x_n$
to $w_n$ is bounded, so the length in $c_n$ from $x_n$
to $z_n$ would also be bounded contradiction.

We conclude that $\ell, c$ are distinct center leaves in $L$.
By Theorem \ref{teo.hsdff} the center leaf space restricted
to $L$ is Hausdorff. Hence there is a transversal to the
center foliation in $L$ from $x$ to $w$. This transversal
produces in nearby leaves $L_n$, transversals to the 
center foliation in $L_n$ from $x_n$ to $w_n$. This
is a contradiction, because $x_n$ and $w_n$ are
in the same center leaf in $L_n$.
This finishes the proof.
\end{proof} 

Clearly the same statement holds for the negatively oriented ray. 

\begin{remark}
We say that \emph{the endpoints of curves in $\cT$ vary continuously} if given a sequence of leaves $\ell_n \in \wT$, $\ell_n$  in $L_n$ leaves
of $\wcs$, with endpoints $\ell^+_n$ and $\ell^-_n$ in 
$S^1(L_n)$ and so that $\ell_n$ converges to $\ell$,
in the leaf space of $\wT$ 
and $\ell \subset  L$,
where $L$ is the limit of $L_n$, 
then the sequences 
$$\Theta_{L_n}^{-1}(\ell_n^+), \ \ \ \Theta_{L_n}^{-1}(\ell_n^-)
\ \ \ {\rm in} \ \ \ \Su$$ 

\noindent
converge to $\Theta_L^{-1}(\ell^+)$ and
 $\Theta_L^{-1}(\ell^-)$ respectively. It is worth making the remark that if the sequence of leaves $\ell_n$ converges to $\ell$ in the leaf space of $\wT$ it means that given a compact part of $\ell$ it will be well approached\footnote{If we consider a sequence of points $x_n \to x$ one may choose leaves of $\wT$ through $x_n$ which are far from a given curve passing through $x$ due to potential non unique integrability of the (branching) foliation. Convergence in the leaf space is needed to make this work.} by the curves $\ell_n$. The proof of the previous Lemma can be adapted to show that the endpoints of the curves in $\cT$ vary continuously. Since we will only use the statement above, and continuity also follows a posteriori from the fact that $\cT$ is uniformly quasigeodesic (cf. Corollary \ref{cor.continuousendpoints}), we do not prove this here. 
\end{remark}

Now we can use the previous lemma to show the following property which implies one of the main consequences of being uniform quasigeodesic. This will allow us to show the quasigeodesic property in the next subsection.  

\begin{lema}\label{lem.track}
If $\cT$ has efficient behavior in $\cF$ then there is $R>0$ such that for every center leaf $\ell \in \wT$ contained in $L$ leaf of $\wcs$, 
then the geodesic $g$ in $L$ with ideal points $\ell^+$ and $\ell^-$
is at Hausdorff distance less than $R$ in $L$ from $\ell$.
The same results holds for segments or rays in leaves of $\wT$.
\end{lema}

\begin{proof}
Notice that condition (ii) of a efficient lamination says that $\ell^-
\not = \ell^+$, so the geodesic $g$ is defined.

We first prove this for finite segments: there is a uniform constant $R>0$ such that for every $L \in \Ft$, the Hausdorff distance in $L$ between a geodesic segment in $L$ joining the endpoints of an arc $I \subset \ell \in \wT \cap L$ and $I$ is less than $R$. Assume that this is not the case. Then, we can find a sequence $I_n$ of segments of leaves $\ell_n \in \wT$, 
$\ell_n \subset L_n \in \widetilde \cF$,
  so that there is a point $x_n \in I_n$ at distance larger than $n$ 
in $L_n$ from the geodesic segment $g_n$ in $L_n$ joining the endpoints of $I_n$. 

Up to composing with deck transformations we can assume that the points $x_n$ belong to a fixed compact fundamental domain of $M$ in $\mt$ and therefore, up to subsequence, that $x_n \to x \in \mt$, that $\ell_n \to \ell$ through $x$, and that $L_n \to L$ containing $\ell$. 
Up to another subsequence assume that one of the endpoints
of $g_n$ converges to a point $\xi$ in $\Su$.

Since the distance in $L_n$
from $x_n$ to $g_n$ converges to infinity, and $g_n$ are
geodesic segments in $L_n$ it follows that visual measure of $g_n$
in $L_n$ from $x_n$ converges to $0$. In other words both
endpoints of $g_n$ in $L_n$ converge to the same point $\xi$ of $\Su$.
Applying Lemma \ref{lem.cont} it follows that both
endpoints of $\ell$ are $\Theta_L(\xi)$.
This contradicts condition (ii) of an efficient lamination.
This proves the result for segments.

To get the result for full center leaves, take $x \in \ell$ and consider a sequence $I_n$ of intervals of $\ell$ from $z_n$ to a point $y_n$ so that $y_n \to \ell^+$ and $z \to \ell^-$.  Since $y_n \to \ell^+$, $z_n \to \ell^-$  in $\hat L = L \cup S^1(L)$ it follows that the geodesic segments from $z_n$ to $y_n$ converge uniformly on compact sets
to the geodesic with ideal points $\ell^+$ and $\ell^-$. Therefore the result holds maybe by taking $R$ slightly larger.   
A similar proof holds for rays.
\end{proof}

\subsection{The quasigeodesic behavior} 
Here we show the following result which is standard. A similar result in a slightly different setting can be found in \cite{FMosh}. 

\begin{prop}\label{prop.QG}
Let $\cT$ be a one dimensional (branching) foliation of $M$ which subfoliates $\cF$. Assume that there exists $R>0$ such that for every $L \in \Ft$ and every finite segment $I$ in a leaf $\ell \in \wT \cap L$ there is a geodesic segment in $L$ with same endpoints as $I$ which is at 
Hausdorff distance in $L$ less than $R-1$ from $\ell$. Then $\cT$ is by uniform quasigeodesics in $\cF$. 
\end{prop}

\begin{proof}
As seen by the proof of the last Lemma the condition implies
that full leaves of $\wT$ have distinct ideal points 
and are $R$ distant from the corresponding geodesics
in their respective leaves. This also immediately implies
that the leaf space of $\wT$ in any leaf of $\widetilde \cF$ 
is Hausdorff.

Let $a_0 > 3R$. We first claim that there is a global length
$a_1$ so that if a segment $I$ in a leaf $\ell$ of $\wT$
contained in $L$ leaf of $\widetilde \cF$ has length
more than $a_1$ then the distance in $L$ between the endpoints
of $I$ is more than $a_0$. 
Otherwise find segments $I_n$ of length $> n$ with endpoints
$x_n, y_n$ less than $a_0$ in their respective leaves. Up
to deck transformations and subsequences assume that $x_n \rightarrow
x_0, y_n \rightarrow y_0$ both in $L$, which is the limit
of leaves $L_n$ containing $I_n$.
The leaves $\ell_n$ through $x_n, y_n$ converge to a leaf
$b_0$ through $x_0$ and a leaf $b_1$ through $y_0$. This is in
the leaf space of $\wT$. If $b_0 = b_1$ then the lengths of
$\ell_n$ between $x_n$ and $y_n$ are bounded contradiction.
Hence $b_0$, $b_1$ are distinct leaves. There is a trasversal
from $x_0$ to $y_0$ and this leads to transversals in respective
leaves of $\widetilde \cF$ from $x_n$ to $y_n$, contradiction.
This proves the claim.

Given $\ell$ leaf of $\wT$ in leaf $L$ of $\widetilde \cF$ consider the geodesic $g$
in $L$ with same ideal points as $\ell$ and the orthogonal
projection from $\ell$ to $g$. Since $\ell$ is in a neighborhood
of size $R$ in $L$ from $g$, then the claim above shows that 
every time we follow along $\ell$ a length $\geq a_1$ the projection
to $g$ moves forward at least $R$. 
This proves the uniform quasigeodesic behavior of leaves of $\wT$.
\end{proof}

\begin{proof}[Proof of Theorem \ref{teo.QG}]
As observed in Remark \ref{rem.efficient} we know that under our assumptions the center foliation $\cT$ has efficient behavior in $\cs, \cu$. 
Properties ($ii$) and ($iii$) of efficient behavior
plus Lemma \ref{lem.track} imply that $\wc$ is in the hypothesis of Proposition \ref{prop.QG} with respect to both $\cs$ and $\cu$. The result then follows. 
\end{proof}

\section{The collapsed Anosov flow property}\label{s.CAF} 

In view of the previous section we can deduce: 

\begin{teo}\label{teo.CAFabstract} 
Let $f: M \to M$ be a partially hyperbolic diffeomorphism preserving branching foliations $\cs$ and $\cu$ which are uniform, $\RR$-covered, 
and such that both $(f,\cs)$ and $(f,\cu)$ have full pseudo-Anosov behavior. 
Suppose that $\cs, \cu$ are transversely orientable.
Then, $f$ is a collapsed Anosov flow. 
\end{teo}

\begin{proof}
The orientation hypothesis are equivalent to $E^u, E^s$ being
orientable respectively.
Theorem \ref{teo.QG} shows that centers are quasigeodesics
in the respective leaves of $\wcs, \wcu$.
This is what is called a {\em quasigeodesic partially
hyperbolic diffeomorphism}, see \cite[Definition 2.15]{BFP}.
Under the orientation hypothesis of $E^s, E^u$,
\cite[Theorem D]{BFP} implies that $f$ is {\em leaf space collapsed
Anosov flow}, which we do not define here. Again
using the orientation hypothesis, \cite[Theorem B]{BFP} 
then 
implies that $f$ is a strong collapsed Anosov flow.
The Definition of a strong collapsed Anosov flow
\cite[Definition 2.9]{BFP} immediately implies that 
$f$ is a collapsed Anosov flow as we have defined here.
Items (i) and (iv) of \cite[Definition 2.9]{BFP} are
the exact conditions defining a collapsed Anosov flow.
\end{proof}

The situations we analyze in this article are simpler
than the general situation analyzed in \cite{BFP}. In particular
the branching foliations $\cs, \cu$ here are $\R$-covered.
We give here a detailed sketch of a proof of 
Theorem \ref{teo.CAFabstract}
in our simpler setting. The reason for this sketch is twofold:
on the one hand some arguments can be simplified and we will point to some of these simplifications that make the paper more self contained. On the other hand we also will use some of the notions to get uniqueness of branching foliations, in particular we will use the following notion:

\begin{defi}\label{def.qGfan} 
A one dimensional branching foliation $\cT$ in an $\R$-covered uniform foliation $\cF$ by hyperbolic leaves is said to be a \emph{quasigeodesic fan foliation} if the following happens: For every $L \in \Ft$ there is a point $p=p(L) \in S^1(L)$ called the \emph{funnel point}, such that there is a bijection from the leaf space of $\wT \cap L$ and points in $S^1(L) \setminus p$, and so that the leaf of $\wT \cap L$ corresponding to the point $q \in S^1(L)$ is a quasigeodesic joining $q$ and $p$. 
\end{defi}

The key point of the proof of Theorem \ref{teo.CAFabstract} is to show that the center branching foliation $\wc$ is a quasigeodesic fan foliation in both $\cs$ and $\cu$ since this allows to produce a (topological) Anosov flow rather easily (in fact, an expansive flow preserving transverse foliations, which is equivalent to being a topological Anosov flow). We refer the reader to \cite{BFP} for details on this, we will concentrate here in explaining how to obtain that centers form a \emph{quasigeodesic fan foliation} just by knowing that the leaves of $\wwc$ are uniform quasigeodesics in $\wcs$ and $\wcu$ (cf. Theorem \ref{teo.QG}). 

To prove this, we follow a path which is somewhat more direct than the one taken in \cite{BFP} since the $\R$-covered property simplifies the arguments. 

First, we notice that the branching foliations $\cs$ and $\cu$ must be minimal. Recall that being minimal means that there is no closed $\pi_1(M)$ invariant set in the leaf space of the (branching) foliation in the universal cover. See \cite[Appendix F]{BFFP-2} for more discussion. 

\begin{prop}\label{prop.minimality}
Let $f: M \to M$ be a partially hyperbolic diffeomorphism preserving branching foliations $\cs$ and $\cu$ which are uniform and $\RR$-covered and such that the center foliation $\wc$ is a quasigeodesic foliation both in $\cs$ and $\cu$. Then $\cs$ and $\cu$ are minimal. 
\end{prop}

\begin{proof}
We argue for $\cs$ since $\cu$ is symmetric. To see this, we use the fact that since they are foliated by quasigeodesics then every leaf of $\wcs$ has cyclic stabilizer in $\pi_1(M)$, and is either a plane, annulus or M\"{o}bius 
band. 
Note that up to finite cover we can assume that
all foliations are orientable and transversally orientable.
In addition, minimality in a finite cover implies minimality in $M$, so it is no loss of generality to assume these orientability assumptions. 
Given the orientation hypothesis the leaves can only
be planes and annuli. 
If there was a proper minimal set one arrives at a contradiction
using a volume versus length argument, exactly
as is done in the proof of \cite[Proposition 6.1]{BFFP-3}.
We explain a bit more:
suppose there is a a proper minimal set
of say $\wc$, then one can construct a region contained in
the complement of the minimal set
which is either a ball or solid torus,  and such that
the region is mapped
inside itself by an iterate of $f$.
The ball or solid torus is obtained by looking at a connected
components of the set of points that are $\geq \eps$ away from the
minimal set for a suitably small $\eps$. Using that leaves
of $\cs$ are planes or annuli, and $\cs$ is $\R$-covered,
one shows this complementary region has to be contained
in either a ball or a solid torus. 
This contradicts \cite[Proposition 5.2]{HaPS}.
\end{proof}

Now, we close this section with a sketch of the proof of Theorem \ref{teo.CAFabstract} pointing to some results from \cite{BFP} when the arguments cannot be simplified. 

\begin{proof}[Sketch of proof of Theorem \ref{teo.CAFabstract}.] One can first argue similar to what is done
in \cite[\S 5]{CalegariPromoting} to see that the set of leaves of $\wcs$ on which the foliation $\wwc$ is a (weak)-quasigeodesic fan is a $\pi_1(M)$ and $\fh$-invariant closed set of the leaf space of $\wcs$ which is non-empty. 
Therefore this set is
everything because of minimality. By weak quasigeodesic fan we mean that 
all center leaves share a common ideal point, but
we allow several curves of the foliation in a leaf of $\wcs$ to joint the same pair of points. 

One gets the less powerful weak quasigeodesic fan property
because to apply arguments similar to those
in \cite[\S 5]{CalegariPromoting} one needs to tighten up the foliation to an equivariant geodesic foliation on leaves. The arguments in \cite{CalegariPromoting} are for geodesic and not
quasigeodesic leaves.

After this is done, 
it is however possible to show that the subset of the leaf space of $\wcu$ corresponding to the interval of centers in a $\wcs$ leaf
that join the same pair of points produces
an open and $\pi_1(M)$-invariant sublamination that cannot be the whole leaf space. So again using minimality we exclude this possibility. 
We refer to detailed proofs, which work with much more
generality, in \cite[Proposition 6.19]{BFP}. 

Once the center (branching) foliation $\wc$ is a quasigeodesic fan foliation in $\cs$ and $\cu$ we apply the approximation foliation result (Theorem \ref{teo-bi}) to obtain a true foliation $\wc_\eps$ 
which subfoliates the approximating foliations $\cs_\eps$ and $\cu_\eps$ with the same quasigeodesic fan property (this is where we use the transverse orientability assumption). One can then show that this gives an expansive flow in $M$ which, by virtue of preserving a pair of foliations, is topologically Anosov \cite[Theorem 5.9]{BFP}. Since the foliation is $\RR$-covered, the flow is transitive and therefore orbit equivalent to a true Anosov flow thanks to Shannon's result (cf. \S~\ref{ss.discretisedandcollapsed}) . 
This already proves the existence of an Anosov flow in $M$ and the notion called \emph{leaf space collapsed Anosov flow} in \cite{BFP}. 
In addition the maps given by the approximating foliation allow one to construct the collapsing map $h$, which must then be intertwining the action of $f$ with a self orbit equivalence associated to how it permutes the orbits of the flow.  
The construction of the semiconjugacy $h$ in \cite{BFP} is fairly complex due to the possibility of branching in the foliations and it is done in detail in \cite[\S 9]{BFP}. 
\end{proof}


\section{Funnel directions}\label{s.funnel}

In this section we obtain a couple of technical properties.

Let $f$ be a partially hyperbolic diffeomorphism satisfying
the hypothesis of Theorem \ref{teo.CAFabstract}.
In particular in every leaf $L$ of either $\wcs$ or $\wcu$,
the center foliation is a fan. 
The {\em stable funnel direction} of a center $c$ 
in a leaf $L$ of $\wcs$ is given by the orientation
in $c$ towards the funnel point in $L$. Similarly
one defines the unstable funnel direction.
By Corollary \ref{cor.continuousendpoints} the stable funnel
direction varies continuously and clearly it is invariant
by deck transformations. 
Notice that the stable funnel direction is defined a priori
for points in center leaves contained in center stable leaves and
not just on points. However any two center leaves through
a point $x$ in $\mt$ are connected by a continuous path
of center leaves through $x$. Since the stable directions
on these centers $-$ verified at $x$ vary continuously,
they all define the same direction at $x$. Therefore
the stable funnel direction depends only on the point.

Let $\cV$ be the universal circle of the center unstable 
foliation $\cu$. For each $U$ leaf of $\wcu$, let $\tau_U:
\cV \rightarrow S^1(U)$ be the canonical identification.

\begin{lema}\label{lem.notsamedirection}
The stable and unstable funnel directions disagree everywhere.
\end{lema}

\begin{proof}{}
Since the stable and unstable funnel directions vary
continuously they either coincide everywhere or disagree
everywhere. Let us assume they coincide everywhere.
Let $\cJ$ be the leaf space of $\wcu$.

We consider a map $\eta: \cJ \rightarrow \cV$ defined as follows.
Given $U$ in $\wcu$, let $q_U \in S^1(U)$ be the unstable
funnel point of $U$. Let $\eta(U) = (\tau_U)^{-1}(q_U)$.

Let $L$ be a leaf of $\wcs$ and $e_1, e_2$ distinct
centers in $L$. Let $I$ be the interval of $\cJ$ of $\wcu$ leaves
intersecting $L$ in a center between $e_1, e_2$ including
the boundary leaves.
Let $U, U'$ be leaves in $I$ intersecting $L$ in centers
$c, c'$. Rays of $c, c'$ in the stable funnel direction
are a bounded distance from each other in $L$, hence in $\mt$.
By the definition of the universal circle of the center unstable
foliation, these rays define the same point in $\cV$.
By hypothesis in this proof the stable funnel direction
is also the unstable funnel direction in the center leaves.
This implies that $\eta$ is constant in $I$. 

By Proposition \ref{prop.minimality} for every $U$ leaf of
$\wcu$ there is a deck translate $\gamma(U)$ contained in
the interior of $I$. Hence the union of deck translates
of $I$ is all of the leaf space $\cJ$. This shows that $\eta$ is
constant. But then $\eta$ would be $\pi_1(M)$ invariant.
This contradicts \cite[Proposition 5.2]{FPmin}.
This finishes the proof.
\end{proof}

\begin{lema}\label{lem.centerconv}
Let $f$ be a partially hyperbolic diffeomorphism satisfying
the hypothesis of Theorem \ref{teo.CAFabstract}. Let $L$
be a leaf of $\wcs$. Then any two centers $c, c'$ in $L$
are asymptotic in $L$ in the stable funnel direction.
In addition if two distinct center leaves $c, c'$ in $L$
intersect in a point $x$, the following happens:
if $c_1, c_2$ are the rays of $c, c'$ respectively starting
in $x$ and in the stable funnel direction then $c_1 = c_2$.
\end{lema}

\begin{proof}{}
Suppose that the first statement
is not true. Then there are $L, c, c'$, and
$\eps > 0$ so that there are points $x_n$ (say in $c$) 
converging to the funnel point of $L$ so that $d_L(x_n, c') > \eps$.
Up to subsequence there are $\gamma_n$ in $\pi_1(M)$
so that $\gamma_n(x_n)$ converges to $x$. Then up to
subsequence $\gamma_n(L)$ converges to $E$,
 $\gamma_n(c)$ converges to a center $e$ in $E$, and 
$\gamma_n(c')$ converges to a center $e'$ in $E$. Since
$d_L(p_n,c') > \eps$ then $e, e'$ are distinct centers in $E$.
By construction and the uniform quasigeodesic property,
the centers
$e, e'$ have the same pair of ideal points in $S^1(E)$.
This contradicts Theorem \ref{teo.CAFabstract} that 
the center foliation in $E$ is a quasigeodesic fan.
This proves the second statement.

To prove the second statement, suppose that there are $c, c'$ 
center leaves in some leaf $L$ which intersect in $x$ but
so that the rays $c_1, c_2$ in the stable funnel direction
in $L$ are not
the same. We already know that the rays $c_1, c_2$ are asymptotic
in $L$.
Let $V$ be a component of $L -  (c_1 \cup c_2)$ which is between
$c_1$  and $c_2$. Then it contains a stable segment $s_0$
through a point $y_0$ in $V$.
As usual let $\hat f$ be a lift of $f$.
Take deck translates $\gamma_i$ of a subsequence $\hat f^{n_i}(y_0)$
converging to $y$ with $n_i \to -\infty$, so the
stable lengths increase. Up to another subsequence suppose that 
$\gamma_i \hat f^{n_i}(c_i)$ converges to curves $d_1, d_2$
which are contained in center leaves $e_1, e_2$ in the limit
center stable leaf $E$. 
Let $W$ be the limit of $\gamma_i \hat f^{n_i}(V)$ which is a region
between $d_1, d_2$.
The limit of $\gamma_i \hat f^{n_i}(s_0)$ is  
at least the full stable leaf $s$ through $y$ which is contained
in $W$.

It could be that $d_1, d_2$ have an
endpoint, which then would be the limit $z$ of $\gamma_i \hat f^{n_i}(x)$.
In this case $d_1, d_2$ are rays in $e_1, e_2$.
Otherwise $d_1, d_2$ are  the full leaves $e_1, e_2$.
In the first case the two rays of $s$ limit to the same point
in $S^1(E)$
which is the common ideal point of $d_1, d_2$ in $S^1(E)$.
But the two rays
of $s$ have to be at least some distance apart from each other 
or else they would intersect the same foliated box of
the center foliation, a contradiction.
The rays are in the region between $e_1$ and $e_2$. This would
imply that $e_1, e_2$ are not asymptotic in the stable funnel
direction. This contradicts the first statement that has already
been proved.
In the second case $d_1, d_2$ are the full leaves $e_1, e_2$.
But then the two distinct center leaves $e_1, e_2$ in $E$ have
both endpoints which are the same. This is impossible because
the center foliation is a quasigeodesic funnel in $E$.
This finishes the proof of the second statement.
\end{proof}

\section{Uniqueness of the branching foliations}\label{s.uniqueness}

In this section we show: 

\begin{teo}\label{teo-uniqueness}
Let $f: M \to M$ be a partially hyperbolic diffeomorphism preserving branching foliations $\cs$ and $\cu$ such that both $(f,\cs)$ and $(f,\cu)$ have full pseudo Anosov behavior (Definition \ref{defi.fullpA}). Then, if $\cs_2$ is another branching foliation such that $(f, \cs_2)$ has full pseudo-Anosov behavior then $\cs=\cs_2$.
\end{teo}

Some parts will require less assumptions, 
but whenever shorter we will choose to give a direct proof in our specific setting. Of course there is a symmetric statement 
to show uniqueness of $\cu$.
One should compare this result to \cite[\S 12]{BFFP-3} were we get some uniqueness results for branching foliations in a different setting. Later, we will put these results together to get very strong unique integrability properties for partially hyperbolic diffeomorphisms of hyperbolic 3-manifolds. 

We start by showing that the induced center foliations $\wc_1$ (by intersection between $\cs$ and $\cu$) and $\wc_2$ (by intersection between $\cs_2$ and $\cu$) coincide. 

\subsection{Limit behavior} 
In this section $f: M \to M$ will be a partially hyperbolic diffeomorphism preserving a branching foliation $\cu$ tangent to $E^{cu}$ so that $\cu$ is subfoliated by two $f$-invariant one dimensional branching foliations $\wc_1$ and $\wc_2$ tangent to $E^c$ which are quasigeodesic fan foliations (cf. Definition \ref{def.qGfan}) obtained by intersecting with $f$-invariant branching foliations $\cs_1$ and $\cs_2$.  
In the proof of Theorem \ref{teo-uniqueness}, \  $\cs_1 = \cs$.

Notice that a priori we have four choices for funnel directions
on center leaves: two stable funnel directions (the pairs
$\cu, \cs_1$ and  $\cu, \cs_2$) and likewise two unstable funnel
directions (for the same pairs). Lemma \ref{lem.notsamedirection} shows
that for the same pair, the stable and unstable funnel directions
are opposite.

Here we show that a particular configuration holds if the foliations 
$\wc_1, \wc_2$ do not coincide. In the next section we will show that this is impossible.

For the next few results we only consider unstable funnel
directions or points. So for simplicity, unless otherwise
stated we refer to them as funnel directions or funnel points.
In addition the universal circle, still denoted by 
$\Su$, is the universal circle of $\cu$. Similarly, we use the
previous notation for the maps
$\Theta_L: \Su \rightarrow S^1(L)$ for $L \in \wcu$.

We need to show the following:

\begin{lema}\label{lem.fanpoint} 
The funnel points of $\wc_1$ and $\wc_2$ coincide. 
\end{lema}

\begin{proof} 
By Corollary \ref{cor.continuousendpoints} we have that the set of leaves $L \in \wcu$ where the funnel points of $\wwca$ and $\wwcb$ coincide is closed and $\pi_1(M)$ invariant. Therefore, by minimality (Proposition \ref{prop.minimality}) we just need to show that there exist some leaf where they coincide.  

To do this, take a leaf $L$ where the funnel points differ (if there is no such leaf, there is nothing to prove). Denote by $p_i$ to the funnel point of $\tild{\wc_i}$ in $L$ (with $i=1,2$) and consider a point 
$\xi \in \Su$ with $\Theta_L(\xi)$ not in $\{p_1,p_2\}$. 

Choose a sequence $x_n \in L$ such that $x_n \to \Theta_L(\xi)$
in $L \cup S^1(L)$.
Now, composing with deck transformations $\gamma_n \in \pi_1(M)$ sending $x_n$ to a given bounded set, and up to extracting a subsequence  we have that $\gamma_n x_n \to x \in \mt$. Let $L_\infty \in \wcu$ be the limit of the leaves $\gamma_n L$ which is a leaf through $x$. 
The funnel points of $\gamma_n L$ are given by $\gamma_n p_1$ and
$\gamma_n p_2$. These converge to the funnel points in $L_\infty$.
Since the visual measure from $x_n$ of the interval between $p_1$ and $p_2$ in $S^1(L)$ that does not contain $\Theta_L(\xi)$ goes to zero with $n$ we deduce that the endpoint of the quasigeodesic fans in $L_\infty$ must coincide. This completes the proof.
\end{proof}

Finally we show the following result which is important to get a contradiction in our case, but we note that the proof may work in more generality. 

\begin{lema}\label{lem.configunique}
Assume that $(f,\cu)$ has full pseudo-Anosov behavior. If $\wc_1 \neq \wc_2$ then there exists a regular pA pair $(\fh,\gamma)$ and a leaf $L \in \wcu$ which is fixed by a conjugate $P$ of $P_0=\gamma^m \circ \fh^k$ (cf. Notation \ref{not-lift}) such that it contains two disjoint curves $c_1 \in \wwca$ and $c_2 \in \wwcb$ whose endpoints in $S^1(L)$ are $\Theta_L$ images
of super-attracting points of $P_\infty$ in $\Su$. 
In addition $P(c_i) = c_i$.
\end{lema}

\begin{proof}
If $\wc_1 \neq \wc_2$, using Lemma \ref{lem.fanpoint} we know that there is a leaf $E' \in \Ft$ such that there are center curves $e_1 \in \wc_1$ and $e_2 \in \wc_2$ in $E'$ which share both endpoints and so that $e_1 \neq e_2$.  Denote by $p,q \in  S^1(E')$ the ideal points
of the curves $e_i$ ($i =1,2$). One can then define the region $Z$ between $e_1$ and $e_2$ as the union of connected components of $E' \setminus (e_1 \cup e_2)$ whose closure in $\hat E' = E' \cup S^1(E')$ is contained in $E' \cup \{p,q\}$. This is an open and non-empty set and we can then consider an unstable interval $I \subset E'$ (i.e. tangent to $E^u$) which is contained in $Z$. 

Fix $x$ in the interior of $I$.
Consider a lift $f_1$ of $f$ to $\mt$. Up to subsequence there 
are $\gamma_j$ in $\pi_1(M)$ so that $\gamma_j f^{n_j}_1(x)$
converges to $y$ in a leaf $E_0$, where $E_0$ is the 
limit of $\gamma_j f^{n_j}_1(E')$. We can assume that
$\gamma_j f^{n_j}_1(e_i), i = 1, 2$ also converge to centers $e'_i, 
\ i = 1,2$ 
in $E_0$. 
This is because the curves

$$\gamma_j f_1^{n_j}(e_1), \ \ \gamma_j f_1^{n_j}(e_2)$$
\noindent
 have
the same pair of ideal points in $S^1(\gamma_j f_1^{n_j}(E'))$ and
hence are a bounded Hausdorff distance from each other in
$\gamma_j f_1^{n_j}(E')$. Finally $\gamma_j f_1^{n_j}(x)$ is
between them in $\gamma_j f_1^{n_j}(E')$.
The limit of $\gamma^{n_j}_1(I)$ contains
the full unstable leaf $u'$ of $q$, which is then betweeen
$e'_1, e'_2$ in $E_0$. 

If the ideal points of $u'$ are distinct in $S^1(E_0)$,
let $u_1 = u'$, $e^1_i = e'_i$, $E = E_0$.

If the ideal points of $u'$ are the same point $z$ in $S^1(E_0)$
consider $y_n$ in $u'$ converging to $z$ in $E_0 \cup S^1(E_0)$ 
so that $\pi(y_n)$
converges in $M$. There are $\beta_n$ in $\pi_1(M)$ so that 
$\beta_n(y_n)$ converges to $q_0$, $\beta_n(E_0)$ converges,
and we let the limit of $\beta_n(E_0)$ be $E$.
We can also assume that $\beta_n(e'_i), i = 1,2$ converge,
and we let the limits be $e^1_i$.
Then $\beta_n(u')$ converges to at least one unstable leaf
$u_1$ in $E$ which separates $e^1_1$ from $e^1_2$ in $E$.
So in any case we obtain a leaf $E$ of $\wcu$ with 
two centers $e^1_1, e^1_2$ of $\widetilde \wc_i$ respectively
so that $e^1_1, e^1_2$ have the same ideal points in $S^1(E)$
and there is an unstable leaf $u_1$ in $L$ separating 
$e^1_1$ from $e^1_2$.

We assumed that $(f,\cu)$ has full pseudo Anosov behavior (cf. Definition \ref{defi.fullpA}). 
So there is a conjugate $P$ of the full regular pair so that
points in disjoint rays of $e^1_i$ are either in distinct
basins of attraction of $P$ or distinct basins of repulsion
of $P$. 
To get this use the last two bullet points of 
Definition \ref{defi.fullpair}.
In the case of attraction (repulsion) use iterates
$P^n$ as $n \rightarrow \infty$ (as $n \rightarrow -\infty$).
In either case we get $L$ is the limit of $P^n(E)$, and
up to subsequence $P^n(e^1_i)$ converges to $c_i$ center
leaves in $L$ which are invariant by $P$ and there is an
unstable leaf $u$ in $E$ separating $c_1$ from $c_2$ and $u$
invariant by $P$.
Since $P(u) = u$ the ideal points can only be super attracting.

This finishes the proof of the Lemma.
\end{proof}

\subsection{Orientability of the center foliation} 

Even if the funnel points coincide (cf. Lemma \ref{lem.fanpoint}), the orientation may be different in both. This will play a crucial role in the proof, so we introduce the following definition: 

\begin{defi}\label{def.orientfan}
Let $\cT$ be a quasigeodesic fan foliation of $\cF$ (cf. Definition \ref{def.qGfan}) and consider an orientation of the tangent space to $\cT$. We say that $\cT$ is oriented \emph{towards} the funnel point if every curve of $\wT \cap L$ is oriented in the direction of the funnel point. Otherwise, we say that $\cT$ is oriented \emph{against} the funnel direction. 
\end{defi}

\begin{remark}\label{rem.orientcenter}
Notice that the definition makes sense. First of all, the fact that $\cT$ is a quasigeodesic fan foliation implies that $\cT$ is orientable. Secondly, the funnel point varies continuously (cf. Corollary \ref{cor.continuousendpoints}), therefore, either $\wT \cap L$ is oriented in the direction of the funnel point everywhere or nowhere. 
\end{remark}

We will choose from now on an orientation in 
$E^c$ making that $\wc_1$ is oriented towards the funnel point. 
As remarked before we are always considering the unstable
funnel point and the direction in the center unstable leaves.

\begin{lema}\label{lema-orientinverse} 
Assume that $(f,\cu)$ has full pseudo-Anosov behavior. If $\wc_1 \neq \wc_2$ then $\wc_2$ is oriented against the funnel point. 
\end{lema} 

\begin{proof} 
We work in the leaf $L$ given by Lemma \ref{lem.configunique} where we have disjoint centers $c_i \in \tild{\wc_i}$ fixed by $P$ which join the attracting points $a_1, a_2$ of $P$ in $\Su$ and which are separated by a fixed unstable leaf $u$ which also joins those points. Let $x \in u$ be the unique fixed point. Note that since both $\wc_1$ and $\wc_2$ share their funnel points, we can assume that $a_1$ is the funnel point for both. 

Let $e_i$ be a curve of $\tild{\wc_i}$ through the point $x \in u$ and fixed by $P$. Note that $x$ may belong to many curves in $\tild{\wc_i}$ but at least one must be fixed by $P$, we choose any such fixed curve. It is important to remark that $e_i \neq c_i$ since $c_i$ does not intersect $u$. 

Consider the ray of $e_i$ from $x$ pointing to the region between $u$ and $c_i$. This region has limit points only $a_1$ and $a_2$.
We claim that the endpoint of the ray must be $a_1$: if it were $a_2$ 
this is different than $a_1$ and so the other endpoint of $e_i$ 
would be $a_1$. Hence $e_i$ is another curve in $\tild{\wc_i}$ from $a_1$ to $a_2$ (besides $c_i$) 
 and this is inconsistent with being a quasigeodesic fan foliation. 

Since the regions between $u$ and $c_1$ and $u$ and $c_2$ are oriented differently from $x$, we deduce that the orientation of $e_2$ has to be against the funnel point. Since orientations coincide or disagree everywhere, this concludes. 
\end{proof}

\begin{figure}[ht]
	\begin{center}
		\begin{overpic}[scale=0.73]{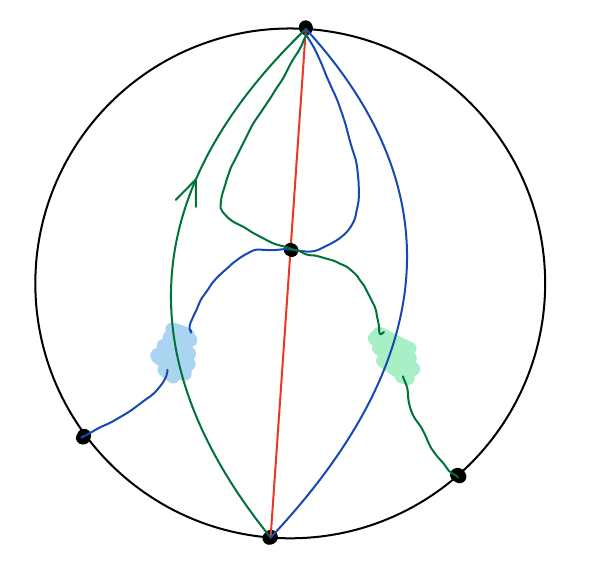}
		\end{overpic}
	\end{center}
	\begin{picture}(0,0)
	\put(0,120){$x$}
	\put(-1,95){\red{$u$}}
	\put(3,212){$a_1$}
	\put(-7,20){$a_2$}
	\put(-85,58){$r_1$}
	\put(68,44){$r_2$}
	\put(-53,120){\color{ForestGreen}{$c_1$}}
	\put(48,75){\color{ForestGreen}{$e_1$}}
	\put(-58,64){\blue{$e_2$}}
	\put(43,128){\blue{$c_2$}}
	\end{picture}
	\vspace{-0.5cm}
	\caption{{\small Proof of Lemma \ref{lema-orientinverse}.}}\label{f.orient}
\end{figure}

\subsection{Proof of Theorem \ref{teo-uniqueness}}

We first show: 

\begin{lema}
If $\cs_1 = \cs \neq \cs_2$ then there is a leaf $U \in \wcu$ such that the foliations $\wwca$ and $\wwcb$ are different. 
\end{lema}

\begin{proof} 
We show that if $\wwca = \wwcb$ in every leaf $U$ of $\wcu$, then
$\cs = \cs_2$.
Let $L$ in $\cs_1$ and $U$ in $\cu$  intersecting $L$ in a center
$c$ leaf of $\wwca$. Since $\wwca = \wwcb$, then there is $E$ leaf
of $\cs_2$ intersecting $U$ also in $c$.
We will show that $E = L$, hence every leaf of $\cs_1$ is also
a leaf of $\cs_2$ and vice versa, proving the result. 

Let $W$ be the union of the stable leaves intersecting $c$.
The foliations
$\cs_1, \cs_2$ have leaves which are stable saturated, hence 
$W$ is contained in both $L$ and $E$.
Let $p$ be the (stable) funnel point of $L$.
Lemma \ref{lem.centerconv} shows that for any other center
$c_1$ in $L$
then $c, c_1$ are asymptotic in the direction of $p$.
Hence $c_1$ has a ray towards $p$ contained in $W$ (so contained
in $E$). This ray defines direction 1 in $c_1$.
Let $V$ be a leaf of $\wcu$ so that $V \cap L = c_1$.
Let $c_2 = V \cap E$. Then $c_1, c_2$ share a ray in direction 1.
The unstable funnel direction in $V$ induces the opposite
direction (direction 2) in $c_1$ by Lemma \ref{lem.notsamedirection}. 
It follows that in $V$ the rays of $c_1, c_2$ corresponding
to direction 2 have the same ideal point $q$ in $S^1(V)$.

Notice that $c_1$ is a leaf of $\wwca$ and $c_2$ is a leaf
of $\wwcb$. Since these foliations are the same in $V$,
then $c_2$ is also a leaf of $\wwca$ in $V$. But then $c_1, c_2$
are leaves of $\wwca$ with same pair of ideal points in $V$
(in direction 1 they share a ray, in direction 2 they
both limit to $q$). Since $\wwca$ is a quasigeodesic fan
in $V$ it follows that $c_1 = c_2$. In other words $c_1$
is contained in $E$. Since this is true for any center in $L$
then $L \subset E$. Since $L$ is properly embedded this
implies that $L = E$.
This finishes the proof.
\end{proof}

Now we can apply what we showed before to prove uniqueness:

\begin{proof}[Proof of Theorem \ref{teo-uniqueness}]
By the previous Lemma it
is enough to show that $\wc_1 = \wc_2$, so assume by way
of contradiction that $\wc_1 \neq \wc_2$. 
Now use Lemma  \ref{lem.configunique}, which provides a $P$
and a leaf $L$ of $\wcu$ fixed by $P$, containing two
leaves $c_i$ of $\wc_i$ invariant by $P$ and an unstable
leaf $u$ in $L$ fixed by $P$ separating $c_1$ from $c_2$ in $L$. 
This is the 
setup of Lemma \ref{lema-orientinverse} 
and we use the same center curves $e_i$ through a fixed point $x \in u$
as in that lemma.

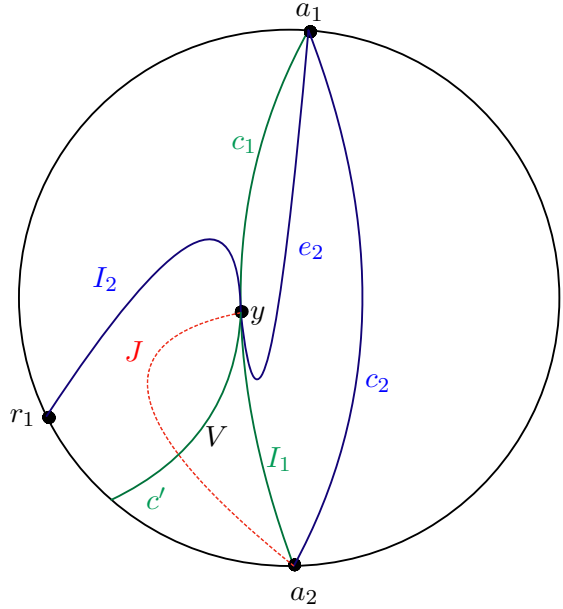
\begin{figure}[ht]
	\begin{center}
		\begin{overpic}[scale=0.73]{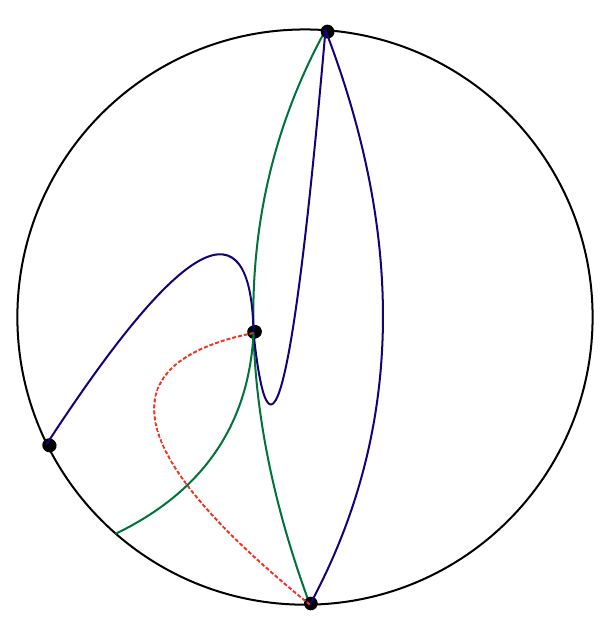}
		\end{overpic}
	\end{center}
	\begin{picture}(0,0)
	\put(4,240){$a_1$}
	\put(2,20){$a_2$}
	\put(-103,87){$r_1$}
	\put(-13,126){$y$}
	\put(-20,190){\color{ForestGreen}{$c_1$}}
	\put(-7,70){\color{ForestGreen}{$I_1$}}
	\put(-60,110){\red{$J$}}
	\put(-30,78){$V$}
	\put(-52,54){\color{ForestGreen}{$c'$}}
	\put(-72,138){\color{blue}{$I_2$}}
	\put(30, 100){\color{blue}{$c_2$}}
	\put(5,150){\color{blue}{$e_2$}}
	\end{picture}
	\vspace{-0.5cm}
	\caption{{\small Proof of uniqueness.}}\label{f.unique}
\end{figure}

Recall the setup of Lemma \ref{lema-orientinverse}: there are 
$4$ fixed points of $P$ on $S^1(L)$, which are $a_1, a_2$ (attracting)
and $r_1, r_2$ (repelling).
Since $e_2$ is fixed by $P$ its ideal points are 
fixed points of $P$ in $S^1(L)$. One of them is $a_1$.
The other ideal point $z$ of $e_2$ cannot be $r_2$ as $c_2$ 
separates $r_2$ from $e_2$. The point $z$ cannot be $a_2$ either since
$c_2$ already has ideal point $r_2$.
It follows that $z = r_1$, see figure \ref{f.unique}.
Here $r_1$ is the repelling fixed point of $P$ acting on 
$L \cup S^1(L)$ which is not separated by $u$ from $c_2$.
In particular, $e_2$ must intersect $c_1$. 

Let $y \in c_1 \cap e_2$ be the last point of intersection (when following $e_2$ towards the point $r_1$ or $c_1$ towards $a_1$ which are the positive orientations). The point $y$ must be fixed by $P$. Let $I_1$ be the ray
of $c_1$ from $y$ to the ideal point $a_2$,
and let $I_2$ be the ray of $e_2$ from $y$ to the ideal point
$r_1$. 
It follows that $I_1 \cup I_2$ separates $L$ in two components, one of which, that we call $Z$ has its closure in $\hat L$ containing the segment in $S^1(L)$ from $r_1$ to $a_2$ (and not intersecting any other of the fixed points of $P$ in $S^1(L)$). Notice that $I_1$ is in a leaf $c_1$ of
$\wwca$ and at $y$ the orientation in $c_1$ is pointing
away from $I_1$.
Conversely $I_2$ is contained in a leaf $E_2$ of $\wwcb$ 
and at $y$ its orientation is pointing into $I_2$ $-$ in other
words pointing away from $I_1$. This is
because Lemma \ref{lema-orientinverse}.
Since both $I_1$ and $I_2$ are oriented coherently at $y$ it follows that the unstable manifold $u(y)$ of $y$ has one ray $J$ inside $Z$. 
But $J$ is invariant under $P$ and it is expanding under $P$, hence
$J$ must have ideal point $a_2$. Consider the region $V$ of $L$
 bounded by the union of $I_1$ and the ray $J$.

If $c' \in \wwca$ is a curve intersecting $V$ it follows that it must intersect $u(y)$ twice, contradicting the fact that an unstable manifold cannot intersect the same leaf of $\wcs$ twice (cf. \S \ref{ss.bran}). This completes the proof. 
\end{proof} 

\section{Hyperbolic manifolds: Proof of Theorems A and B}\label{s.hyperbolic}
In this section $f: M \to M$ will be a partially hyperbolic diffeomorphism of a hyperbolic 3-manifold. Recall that a hyperbolic manifold is one obtained as a compact quotient of $\mathbb{H}^3$ by isometries. By Perelman's proof of Thurston's geometrization conjecture this is equivalent to being aespherical (i.e. $\pi_2(M)=\{0\}$) and homotopically atoroidal (i.e. no $\pi_1$-injective torus) with infinite fundamental group (see \cite[Appendix A]{BFFP-2}). Note that we will only use the atoroidal condition plus generalities about foliations.

\subsection{Dichotomy: Discretized Anosov or double translation}\label{ss.dichotomy}
Here we explain how the main results of \cite{BFFP-3} allow us to reduce the proof of Theorem \ref{teoB} to what we did so far.  

Up to finite cover and iterate, we have that $f$ must preserve branching foliations and be homotopic to the identity (this is because of Mostow rigidity, see e.g. \cite[Proposition A.3]{BFFP-2}). We will lift these assumptions in \S \ref{ss.teoAB} to prove the full theorem.

This is the main statement of \cite{BFFP-3} we will use. See \cite[Theorem 2.4]{BFFP-3}. 

\begin{teo}\label{BFFP}
Let $f : M \to M$ be a partially hyperbolic diffeomorphism homotopic to the identity of a hyperbolic 3-manifold $M$ which preserves branching foliations $\cs$ and $\cu$. Then, 
\begin{enumerate}
\item\label{it1hyp} either $f$ is a discretized Anosov flow, or,
\item\label{it2hyp} the pairs $(f,\cs)$ and $(f,\cu)$ have the commuting property (cf. \S \ref{ss.abundancepApairs}). 
\end{enumerate}
\end{teo}

Notice that in case (ii) of this theorem we are in option
(i) of \S \ref{ss.abundancepApairs}. 
In particular there is
a lift $\ft$ of $f$ to $\mt$ which commutes with all deck
transformations and $\ft$ acts freely on the leaf space of $\wcs$.
This global commuting property of $\ft$ immediately implies
that for any $\gamma$ deck transformation, then $\gamma$
preserves transversal orientations to $\wcs$ and $\wcu$. 
Hence the orientation conditions of Theorem \ref{teo.CAFabstract}
are satisfied.

We need to make some comments to explain how this follows directly from \cite{BFFP-3}. When $f$ is homotopic to the identity we call a lift $\ft: \mt \to \mt$ a \emph{good lift} if it commutes with all deck transformations. Such a lift can be obtained by lifting a homotopy to the identity. See \cite[Definition 2.3]{BFFP-3}. The good lift has the property required for $(f,\cs)$, $(f,\cu)$ to have the commuting property (the fact that $\cs$ and $\cu$ are $\RR$-covered and uniform are direct consequences of \cite[Theorem 2.4]{BFFP-3}). 

Notice that if $f$ is a discretized Anosov flow then it is a collapsed Anosov flow, so we need to analyse only the second situation which we call \emph{double translation}.  
 
\subsection{Regulating pseudo-Anosov flows}\label{ss.anosovPA} 
We state here the results that follow from \cite{Thurston,CalegariPA,Fen2002}. We remark that these results depend only on the fact that $M$ is atoroidal and not on its geometry (said otherwise, they depend on the coarse geometry and not on the precise hyperbolic metric). 

See \cite[Proposition 10.1]{BFFP-3} 
for the adaptation to branching foliations: 

\begin{teo}\label{teo-regpA} 
Let $\cF$ be a $\RR$-covered, transversely oriented
 and uniform branching foliation of a hyperbolic 3-manifold $M$. Then, there exists a pseudo-Anosov flow $\Phi_t: M \to M$ transverse and regulating to $\cF$. 
\end{teo}

Recall that the condition of being regulating means that in the universal cover $\mt$, given two leaves $L, L' \in \Ft$ there is a uniform time $t_0:= t_0(L,L')$ such that for every $x \in L$ it holds that $\Phi_t(x) \in L'$ for some $|t|< t_0$. We state the following relevant properties about pseudo-Anosov flows in hyperbolic 3-manifolds that follow from previous work by several authors (we give a very short sketch of the proof pointing to some references for more details): 

\begin{teo}\label{teo.pAdynamics}
Let $\Phi_t: M \to M$ be a pseudo-Anosov flow in a hyperbolic 3-manifold. Then, $\Phi_t$ is transitive and therefore both the weak stable and weak unstable (singular) foliations are minimal. Moreover, if $\Phi_t$ is regulating to an $\RR$-covered foliation it cannot be an Anosov flow. 
\end{teo}
\begin{proof}
If a pseudo-Anosov flow in a 3-manifold is not transitive, then it has an incompressible torus (or Klein bottle) transverse to the flow (see \cite{Mos}).
Since $M$ is hyperbolic then this is impossible. 

Once that a pseudo-Anosov flow is transitive, the minimality of the singular foliations follows, since a closed set saturated by unstable (resp. stable) leaves is an attractor (resp. repeller). Finally, in \cite[Proposition D.4]{BFFP-2} we explain how the fact that pseudo-Anosov flows transverse and regulating to $\RR$-covered foliations in hyperbolic 3-manifolds cannot be Anosov follows from previous work by Barbot and the first author. 
\end{proof}

We note that this will provide pA-pairs for $(f,\cs)$ and $(f,\cu)$ in case (\ref{it2hyp}) of Theorem \ref{BFFP} since for every periodic orbit of the transverse and regulating pseudo-Anosov flow one can construct a pA-pair associated to it using the deck transformation associated to the orbit. Note that $\ft$ acts trivially in the universal circle, so one needs only to care about the action of the deck transformation (see \cite[Proposition 10.2]{BFFP-3}). Using Corollary \ref{cor.periodic} one deduces that both $(f,\cs)$ and $(f,\cu)$ have the periodic commuting property. 
In Lemma \ref{lema.superattracting} we show how this produces the announced pA pairs.

\begin{remark} Let $\gamma$ be a deck transformation.
A point $p \in \Su$ is superattracting for $\gamma_{\infty}$
(the induced action on $\Su$),
if and only if for some (and hence for any) $L$ in $\Ft$ if $\gamma_L$ 
is an expression in $L$ of the action of $\gamma$, then the
following happens: there is a neighborhood basis of $\Theta(p)$ in
$L \cup S^1(L)$ defined by geodesics $\ell_i$ in $L$ so that
the minimum distance in $L$ between points in $\ell_i$ and 
$\gamma_L(\ell_i)$ goes to infinity with $i$.
See also proof of Proposition \ref{prop.superattracting}.
\end{remark}

\begin{lema}\label{lema.superattracting}
Let $\Phi_t$ be a pseudo-Anosov flow transverse to $\cF$ as
in Theorem \ref{teo.pAdynamics}. Let $\gamma$ be a deck
transformation associated with a periodic orbit of $\Phi_t$.
Then some power of $\gamma$ has fixed points in $\Su$ and so that all
fixed points are either super attracting or super repelling.
\end{lema}

\begin{proof}{}
We follow the setup in \cite{BFFP-2}.
Fix $L$ in $\widetilde \cF$. Let $\cG^s_L, \cG^u_L$ 
be the singular one dimensional foliations in $L$ induced by
intersecting the stable and unstable $2$-dimensional singular
foliations of $\Phi_t$ lifted to $\mt$ with $L$. 
The non singular leaves are uniform quasigeodesics 
\cite[Fact 8.3]{BFFP-2}.
Let $\cB^s, \cB^u$ be the geodesic laminations in $L$ 
obtained by pulling tight the leaves of $\cG^s_L, \cG^u_L$
respectively. Each non singular leaf of $\cG^s_L$ is a uniformly
bounded Hausdorff distance in $L$ from a unique leaf of 
$\cB^s$. A $p$-prong leaf of $\cG^s_L$ generates $p$ leaves
of $\cB^s$.

The deck transformation
$\gamma$ is associated with a periodic orbit $\alpha$ of $\Phi_t$
and fixes a lift $\widetilde \alpha$ to $\mt$. Up to taking a power
assume that $\gamma$ fixes all prongs of $\widetilde \alpha$.
Assume that $\gamma$ is associated with the negative direction
of $\alpha$. As in \cite[Section 8]{BFFP-2} let $\tau_{12}: L 
\rightarrow \gamma^{-1}(L)$ be the map obtained by flowing
$x$ in $L$ along its $\widetilde \Phi_t$ flow line until it
hits $\gamma^{-1}(L)$. Notice that 
$$d_{\mt}(x,\tau_{12}(x)), \ \ x \in L$$ 
\noindent is bounded. Then $\gamma \circ \tau_{12}$ is a 
representative in $L$ of the action of $\gamma$.
Let $h = \gamma \circ \tau_{12}$.

Let $x = \widetilde \alpha \cap L$, which is the only
fixed point of $h = \gamma \circ 
\tau_{12}$. 
Fix an unstable prong $\eta$ of $x$ with ideal point $p$
in $S^1(L)$. We will prove that $p$ is a super attracting
fixed point of $\gamma$. For a stable prong we get a super repelling fixed
point.
Up to applying a power of $\gamma$ we can assume that 
the Hausdorff distance between $L$ and $\gamma(L)$ is very
big. 
This is okay since the lemma claims the result for a power of 
$\gamma$.
Then by \cite[Fact 8.4]{BFFP-2} the map $h$ expands
length along $\cG^u_L$ exponentially and contracts length
along $\cG^s_L$ exponentially (see also \cite{Fen2002}).
This means that length along $\eta$ from $y$ to $h(y)$ 
goes to infinity as $y$ escapes in $\eta$. 
We consider a basis neighborhood of $p$ defined by leaves of
$\cG^s_L$ intersecting $\eta$: given $y$ in $\eta$ let $\ell_y$
the leaf of $\cG^s_L$ through $y$.

Given $y$ in $\eta$ let $g_y$ be the geodesic associated with 
$\ell_y$: it is a bounded Hausdorff distance in $L$ from $\ell_y$.
Let $\nu$ be the geodesic in $L$ associated with 
$\cG^u_L(x)$ (for simplicity assume $x$ is non singular,
otherwise there are 2 such geodesics associated with the
ray $\eta$). Then the angle between $\nu$ and any
$g_y$ is bounded below by $a_0 > 0$. 
Also the point $y$ is a 
bounded distance from the intersection between $\nu$ and
$g_y$. 

These facts imply that the minimum distance between
points in $\ell_y$ and $h(\ell_y)$ goes to infinity
as $y$ escapes in $\eta$.

This proves that $p$ is a superattracting point.
This finishes the proof.
\end{proof}

\begin{remark}
Note that the pseudo-Anosov flows associated to $\cs$ and $\cu$ given by Theorem \ref{teo-regpA} may be different and not even share the same homotopy classes of periodic orbits. This will not be an issue, and we will obtain a posteriori, that both pseudo-Anosov flows are orbit equivalent since this is the case always for the weak stable and unstable foliations of an $\RR$-covered Anosov flow in a hyperbolic 3-manifold. 
\end{remark}

\subsection{Existence of full pseudo-Anosov pairs}
Here we show: 

\begin{prop}\label{p.fullpAhyperbolic}
Let $f: M \to M$ be a partially hyperbolic diffeomorphism of a hyperbolic 3-manifold with $f$ homotopic to the identity and preserving
transversely oriented branching foliations $\cs$ and $\cu$.
Suppose that both $(f,\cs)$ and $(f,\cu)$ have the periodic commuting property. Then, both pairs have full pseudo-Anosov behavior (cf. Definition \ref{defi.fullpA}). In particular, $f$ is a collapsed Anosov flow. 
\end{prop}

\begin{proof}
This follows from the existence of a regulating pseudo-Anosov flow. We discuss the arguments to get the statements in our current framework. The fact that $(f,\cs)$ and $(f,\cu)$ have the periodic commuting property follows from Corollary \ref{cor.periodic} and \cite[Proposition 10.2]{BFFP-3} 
as explained in the previous section. 

Let $\Phi^{cs}_t$ be the pseudo-Anosov flow given by Theorem \ref{teo-regpA} for the 
branching foliation $\cs$ (the same arguments apply for $\cu$). 
To obtain the existence of a full pA pair (cf. Definition \ref{defi.fullpair}) we use the fact that the singular foliations of the pseudo-Anosov flow are minimal.
The good pairs we will be using are $(\ft, \gamma)$ where $\gamma$
is a deck transformation associated with a regular periodic orbit of 
$\Phi^{cs}_t$ and $\ft$ is the good lift of $f$ to $\mt$.
Since $\Phi^{cs}_t$ is regulating for $\cs$, then
$\gamma$ acts freely on the leaf space of $\wcs$.
Hence $(\ft, \gamma)$ is a good pair.
Up to a power assume that $\gamma$ preserves
all the prongs of the periodic orbit when lifted to the
universal cover. Lemma \ref{lema.superattracting} 
any $P = \ft^m \gamma^n$ ($n$ non zero) has periodic points
when acting on the universal circle of $\cs$. If there are
fixed points 
then they are all either super attracting or super repelling
if $| n |$ is sufficiently big.
This is achievable, because any power of $\gamma$ satisfies
this, and $\ft$ moves
points a bounded distance.
Hence $(\ft, \gamma)$ is a regular pA-pair for $(f,\cs)$.

Now we explain why this provides a full pair.
For each leaf $L$ of $\wcs$ let 
$\cB^s_L, \cB^u_L$ be the geodesic laminations in $L$ obtained by
pulling tight in $L$ the leaves of the stable and unstable foliations of 
$\widetilde \Phi^{cs}_t$ intersected with $L$.
The complementary regions of each of 
these geodesic laminations in $L$ are finite sided ideal polygons,
and the complementary regions of the union are relatively
compact polygons with bounded diameter. The union of these 
over $L$ projects to transverse laminations in $M$ $-$ for
details on these laminations see \cite{Fen2002}\footnote{In \cite{Fen2002} the leafwise geodesic laminations are constructed first, before the
pseudo-Anosov flow, via an analysis of the action of $\pi_1(M)$ on the universal circle of the foliation $\cs$. Then these laminations 
blow down to singular foliations producing a pseudo-Anosov flow. In \cite{Fen2002} this is worked out for (non branching) 
foliations. The case of $\cs$ a branching foliation is worked out in \cite{BFFP-3}.}.
These laminations are minimal. For each $\eps > 0$ there is
a diameter $d_0 > 0$ so that disks or annuli of size $d_0$ in
any of these laminations are $\eps$ dense in $M$.
Choose $\eps$ much smaller than the product foliation
size of all the foliations or laminations involved.
Given the deck transformation $\gamma$ associated to a regular
periodic orbit $\mu$, then the stable and unstable leaves of $\mu$
are annuli or M\"{o}bius bands producing like sets in the
leafwise geodesic laminations. A fixed compact annulus or 
M\"{o}bius band (denoted by
$A^s, A^u$) band near the blow up of the periodic orbit
is $\eps$ dense in $M$.
The $A^s, A^u$ intersect in a core closed curve corresponding
to the blow up (or pre-image) of the periodic orbit $\mu$.

Let now $\eta$ be a geodesic ray in $L$. By the above there is
a length $d_1 > 0$ so that any segment of length $\geq d_1$ 
in $\eta$ intersects one of the laminations $\cB^s_L$ or $\cB^u_L$.
There is $\alpha_1 > 0$ so that the intersection
with at least one of $\cB^s_L$ or $\cB^u_L$ has angle $> \alpha_1$.
This implies that $\eta$ intersects either a lift of $A^s$ or $A^u$
making an angle $> \alpha_1$. This lift is given by a deck
translate $\beta^{-1}$ of a fixed lift of either $A^s$ or $A^u$.
This implies that the conditions of Definition 
\ref{defi.fullpair} are satisfied.

After we showed that both pairs have full pseudo-Anosov behavior, the fact that $f$ is a collapsed Anosov flow follows from Theorem \ref{teo.CAFabstract}. 
\end{proof}

\subsection{Proof of Theorems A and B}\label{ss.teoAB}
Theorem \ref{teoA} follows immediately from Theorem \ref{teoB} since the existence of a collapsed Anosov flow in $M$ explicitely asks for the existence of a (topological) Anosov flow in $M$. Notice that in hyperbolic manifolds every topological Anosov flow is transitive, and therefore the existence of a topological Anosov flow implies the existence of an Anosov flow (cf. \S \ref{ss.discretisedandcollapsed}). 

To show Theorem \ref{teoB} we need to be careful since the existence of branching foliations is ensured by Theorem \ref{teo-bi} only after some iterate
and finite lift. 

\begin{proof}[Proof of Theorem \ref{teoB}]
As explained we can assume that if $f: M \to M$ is a partially hyperbolic diffeomorphism in a hyperbolic 3-manifold, then Theorem \ref{teoB} holds for the lift of some iterate of $f$ to a finite cover (see Theorem \ref{BFFP} and Proposition \ref{p.fullpAhyperbolic}). We denote the finite cover of $M$ as $M_0$ and $f_0$ to the lift of the finite iterate of $f$ to $M_0$. 
The lift $f_0$ is chosen so that it is a lift of an iterate of $f$
which is homotopic to the identity in $M$.
We emphasize that the finite cover is considered so that all bundles are orientable. In the \emph{double translation case} we will show a posteriori that this finite cover is indeed not necessary as the bundles were orientable in the first place. 
Up to taking a further cover and lift of further of iterate
we may assume that $M_0$ is a regular cover of $M$.

We want to show that $f$ preserves branching foliations so that Theorem \ref{BFFP} applies and this completes the proof together with Proposition \ref{p.fullpAhyperbolic}. 

For this, we lift the branching foliations $\cs$, $\cu$ preserved by $f_0$ to $\mt$ which is the common universal cover of $M$ and $M_0$ and denote the lifts as $\wcs$ and $\wcu$.  Let $\tild{f_0}$ the good lift of $f_0$ to $\mt$. We need to show first that deck transformations $\pi_1(M)$ preserve $\wcs$ and $\wcu$ (we know that the subgroup $\pi_1(M_0) < \pi_1(M)$ does preserve them).

\vskip .1in
We first assume that we are in the situation of Theorem \ref{BFFP} (\ref{it2hyp}).

We consider then the pair of foliations $\cs_2$ and $\cu_2$ in $M_0$ obtained by projecting to $M_0$ the foliations
$\gamma \wcs$ and $\gamma \wcu$ for some $\gamma \in \pi_1(M)$. 
The reason why these project to $M_0$ is because $\pi_1(M_0)$
is a normal subgroup of $\pi_1(M)$ so $\pi_1(M_0)$ preserves
$\gamma \wcs, \gamma \wcu$.
Since $\ft_0$ commutes with all deck transformations,
then $f_0$ preserves $\cs_2, \cu_2$.
By Theorem \ref{teo-uniqueness} it is enough to show that the pairs $(f_0,\cs_2)$ and $(f_0,\cu_2)$ have full pseudo-Anosov behavior. But this follows as in Proposition \ref{p.fullpAhyperbolic} once we show that $\tild{f_0}$ acts as a translation on $\gamma \wcsb$ and $\gamma \wcub$ which is direct since $\tild{f_0}$ commutes with $\gamma$. 

Since the foliations are invariant by deck transformations of $M$, and $\tild{f_0}$ acts as a translation and commutes with deck transformations,
it follows that deck transformations of $\pi_1(M)$ must preserve the orientation transverse to both $\wcs$ and $\wcu$. Since the center direction is orientable because of the existence of a funnel point (that must also be preserved by deck transformations) we deduce that all bundles were orientable in $M$ and therefore the finite lift was not necessary to make the bundles orientable. 

Finally, in this case, taking the iterate is not necessary. For this it is enough to show that the foliations $f(\cs)$ and $f(\cu)$ are equal to $\cs$ and $\cu$, but this follows by the same argument applying Theorem \ref{teo-uniqueness}. (See also \cite[Theorem B]{BFP}.)

This finishes the analysis of the case when $\widetilde f_0$ 
acts as a translation in the leaf spaces of $\wcs, \wcu$.

\vskip .1in
We now deal with the case that $\widetilde f_0$ fixes every
leaf of $\wcs$ and of $\wcu$.
Here we use
\cite[Theorem 12.1]{BFFP-3}. It shows that $f$ is dynamically
coherent preserving actual foliations, center stable and center unstable.
The center foliation is the intersection of these, and
hence it is preserved by $f$ as well.
In addition \cite[Theorem 12.1]{BFFP-3} shows that
a finite iterate of $f$ is a discretized Anosov
flow preserving each leaf of the center foliation.
In this case let $h$ be the identity.
The self orbit equivalence $\beta$ is $f$ itself since it
preserves the center foliation.
Orbits of the flow are tangent to the center direction,
showing that $f$ is a collapsed Anosov flow.

This completes the proof of Theorem \ref{teoB}. 
\end{proof}

\subsection{Unique integrability properties}\label{ss.uniqueintegrable} 
We state here a strong geometric consequence of our study: 

\begin{teo}\label{t.unihyp}
Let $f: M \to M$ be a partially hyperbolic diffeomorphism in a hyperbolic 3-manifold. Then, $f$ admits a unique pair $\cs$, $\cu$ of $f$-invariant branching foliations tangent respectively to $E^{cs}$ and $E^{cu}$. Moreover, every curve $c$ tangent to $E^c$ in $\mt$ is contained in the intersection of a leaf $L \in \wcs$ and a leaf $F \in \wcu$ (which is connected). 
\end{teo}

\begin{proof}

Suppose that $f^k$ is a positive iterate homotopic to the identity
and let $g = f^k$. 
Let $\widetilde g$ be the good lift to $\mt$.
We start by proving uniqueness of the branching foliations.

Suppose first $g, \cs, \cu$ is a double translation and suppose that
$f$ preserves another branching foliation $\cs_2$.
Then $g$ also preserves $\cs_2$.
Mixed behavior in general means that $\gt$ fixes leaves of one
foliation (of the pair $\wcu, \widetilde {\cW^{cs}_2}$), 
 but not the other.
But 
mixed behavior is impossible in 
hyperbolic 3-manifolds \cite[\S 12]{BFFP-3}. 
Since $\gt$ acts as a translation on $\wcu$,  then
$\gt$ also acts as a translation on $\widetilde {\cW^{cs}_2}$ so $\cs_2, \cu$
is a double translation pair. 
Then $(g,\cs_2)$, $(g,\cu)$ have the periodic
commuting property (cf Def. \ref{defi.periodicproperty}).
Theorem \ref{teo-uniqueness} implies that
$\cs_2 = \cs$.

Suppose now that $g, \cs, \cu$ is a discretized Anosov flow
and let $\cs_2$ preserved by $f$. 
Then $g$ also preserves $\cs_2$.
Again mixed behavior cannot occur, and now $\gt$ fixes
every leaf of $\wcu$, so it fixes every leaf of $\widetilde {\cW^{cs}_2}$.
It follows
that $g, \cs_2, \cu$ is also a discretized Anosov flow.
Then $\cs_2 = \cs$ follows from \cite[Lemma 7.6]{BFFP-3}.

The statement about curves tangent to $E^c$ is proved
from uniqueness of branching foliations \cite[Proposition 10.6]{BFP} 
as follows:
Let $c$ be a curve tangent to $E^c$.
Following previous notation let $f_0$ be a lift of a
a finite iterate of $f$ to a finite lift $M_0$ of $M$
so that all bundles are orientable in $M_0$ and $f_0$ preserves
the orientability of the bundles. In addition suppose the original
finite iterate of $f$ is homotopic to the identity.
Then $c$ lifts to $c'$ in $M_0$ tangent to the center bundle.
 \cite[Proposition 10.6]{BFP} requires the orientability of the bundles
which is attained by $f_0$ in $M_0$, hence $c'$ is obtained
as the intersection of a leaf of the center stable foliation and
a leaf of the center unstable foliation in $M_0$. But we proved
that these foliations in $M_0$ project to $\cs, \cu$ in $M$.
This proves the result for curves tangent to $E^c$.
\end{proof}

Immediate consequences are the following: 

\begin{cor}
Let $f: M \to M$ be a partially hyperbolic diffeomorphism in a hyperbolic 3-manifold. Then $f$ is a discretized Anosov flow if and only if the bundle $E^c$ is (uniquely) integrable. 
\end{cor}

This follows because in \cite[Theorem B]{BFFP-2} we prove that for 
double translations $E^c$ cannot integrate to a foliation. By the uniqueness properties given by Theorem \ref{t.unihyp} the result follows. 

One can also get a result in the direction of the plaque expansivity conjecture \cite{HPS} in a concrete setting. 

\begin{cor}\label{coro-plaqueexp}
Let $f: M \to M$ be a diffeomorphism of a hyperbolic 3-manifold so that $\cT$ is a one-dimensional normally hyperbolic foliation preserved
by $f$. Then $f$ is dynamically coherent and plaque expansive.  
\end{cor}

We refer the reader to \cite{HPS} for a definition of $\cT$ being
a one-dimensional normally hyperbolic foliation, which in particular
implies that $f$ is partially hyperbolic, and that the tangent
space of the foliation $\cT$ is the center bundle.

\begin{proof} Theorem \ref{t.unihyp} shows that $f$ preserves a unique pair of 
branching foliations $\cs, \cu$ and any curve tangent to $E^c$ is contained
in the intersection
of a leaf of $\cs$ and a leaf of $\cu$.
It follows that $\cT$ has to be the center foliation associated
with these branching foliations.
Since $\cT$ is a foliation (as opposed to a branching one dimensional
foliation) it follows that $\cs, \cu$ are also foliations, and
do not have branching.
This shows that $f$ is dynamically coherent.  

Using Theorem \ref{teoB} we get that an iterate of $f$ is a discretized Anosov flow. These are plaque expansive \cite{Martinchich}. 
\end{proof}

\section{Seifert manifolds: Proof of Theorem C}\label{s.seifert}
In this section we consider a partially hyperbolic diffeomorphism $f: M \to M$ where $M$ is a Seifert manifold and such that the induced action of $f$ in the base is pseudo-Anosov. As in the statement of Theorem C, we will assume that $M$ is Seifert over a hyperbolic orientable orbifold $\Sigma$. We note that in contrast with the hyperbolic case (Theorem \ref{teoB}) the arguments here do not rely on \cite{BFFP-3} and this result can be considered self contained.

In \cite[\S 7]{HPS} it is shown that under these hypothesis, the manifold $M$ is orientable and the bundles $E^s, E^c, E^u$ of $f$ are also orientable. Moreover, up to considering an iterate, $f^k$ it follows that $f^k$ preserves orientation of all bundles and thus we can apply Theorem \ref{teo-bi} to get branching foliations $\cs$ and $\cu$ invariant under $f^k$. (Note that one can take $k=2$.)

Using \cite[\S 5.3]{HaPS} we get that the branching foliations are \emph{horizontal}, in particular, they are $\RR$-covered, uniform and by hyperbolic leaves. Moreover, it follows that in $\mt$, the universal cover of $M$ the action of $\delta \in \pi_1(M)$ associated to the fiber of the circle bundle acts freely on the leaf space of both $\wcs$ and $\wcu$. Using Thurston's classification of surface diffeomorphisms \cite{Thurston-surfaces} one deduces: 

\begin{prop}\label{prop.fullpASeif}
The pairs $(f^k, \cs)$ and $(f^k, \cu)$ have full pseudo-Anosov behavior. 
\end{prop}

\begin{proof}
Since $\cs$ is horizontal, $\delta$ acts freely on $\wcs$.  As explained in Remark \ref{example2}, 
for any lift $\tilde{f}$ of $f^k$ to $\mt$, and for large enough $|m|$, then $(\delta^m \tilde{f}, \delta)$ is a good pair (cf.Definition \ref{good}).
Notice that $\delta$ acts as the identity on the
universal circle of $\cs$. 
Any pair obtained is an admissible pair (cf. Definition \ref{defi.admis}). 
It is easy to see they have the periodic
commuting property (Definition \ref{defi.periodicproperty}).

Finally one can check the full pseudo-Anosov behavior of 
$(f^k, \cs)$ (cf. Definition \ref{defi.fullpA}) using  \cite[Lemmas 6.2 and 6.4]{Casson} the same way as in Proposition \ref{p.fullpAhyperbolic}.  The same argument applies to $\cu$. 
\end{proof}

We deduce from Theorem \ref{teo.CAFabstract}: 

\begin{cor}\label{cor.collapsedSeifor}
The diffeomorphism $f$ is a collapsed Anosov flow.  
\end{cor}

\begin{proof}
It follows from Theorem \ref{teo.CAFabstract} and the analysis above that $f^k$ is a collapsed Anosov flow with respect to the branching foliations $\cs$, $\cu$. We must show that these branching foliations are also $f$ invariant and this concludes. But since $f(\cs)$ is also $f^k$-invariant and the argument of Proposition \ref{prop.fullpASeif} applies, we can invoque Theorem \ref{teo-uniqueness} to deduce that $f(\cs)=\cs$. The same argument applies to $\cu$ and this completes the proof of the Corollary. 
\end{proof}

\begin{remark}\label{rem.uniqueSeif}
One also obtains unique integrability results analogous to those of Theorem \ref{t.unihyp}. We remark that since the argument in Proposition \ref{prop.fullpASeif} applies to \emph{any} branching foliation invariant under $f^k$, we can use the results of \cite[Proposition 10.6]{BFP} in order to deduce that the curves in the branching foliation obtained as intersection of $\cs$ and $\cu$ are \emph{all} the complete curves tangent to $E^c$. 
\end{remark}

\section{Further results}\label{s.further}
In this section we give a couple of applications of pseudo-Anosov pairs to partially hyperbolic diffeomorphisms in other 3-manifolds or isotopy classes to show the flexibility of the tools developed here. We hope other applications can be found. 

\subsection{General partially hyperbolic diffeomorphisms homotopic to the identity}\label{ss.generalPH}
 Theorem \ref{teo-regpA} in \cite{CalegariPA,Fen2002} for atoroidal manifolds has been extended recently by the first author to more general manifolds \cite{FenleyFlow}. In particular, it will allow us to extract the following result that holds in a larger class of 3-manifolds:    

\begin{teo}\label{t.JSJatoroidal}
Let $\cF$ be a transversely oriented,
$\RR$-covered, uniform foliation on a 3-manifold with an atoroidal piece. Then, there exists a deck transformation $\gamma \in \pi_1(M)$ which acts as a translation on the leaf space of $\Ft$ and the induced action in the universal circle $\Su$ of $\cF$ has exactly exactly $4$ fixed points:
two super attracting and two super repelling fixed points. 
\end{teo}

As a consequence, we get: 

\begin{teo}\label{t.doubletrans}
Let  $f: M \to M$ is a partially hyperbolic diffeomorphism homotopic to the identity on a 3-manifold having some atoroidal piece in the JSJ decomposition preserving a branching foliation $\cs$ so that the good lift $\ft$ of $f$ is a translation on the leaf space of $\wcs$. Then both the center (branching) foliation and the strong stable foliation have small visual measure inside the leaves of $\cs$ (cf. Theorem \ref{teo.curvesgeneral}).  
In particular for any ray $r$ of a
center leaf $c$ in a leaf $L$ of $\wcs$, then $r$ accumulates
in a single point in $S^1(L)$.
\end{teo}
\begin{proof}
By translation we mean it has no fixed points on the leaf space
of $\wcs$. This was analyzed in \cite[Proposition 4.6]{BFFP-3},
where it is proved that this implies that $\cs$ is $\R$-covered
and uniform.
The translation of $\ft$ also implies 
that $\cs$ is transversely orientable.
We can apply Theorem \ref{t.JSJatoroidal} and we get that $(f,\cs)$ has the periodic commuting property and it has at least one (regular) pA pair. Therefore, Theorem \ref{teo.curvesgeneral} applies and we get the statement. 
\end{proof}

We note that for discretized Anosov flows the center foliation also has small visual measure in center stable leaves, but the strong stable foliation does not, which looks as something quite remarkable about Theorem
\ref{t.doubletrans} that needs to be better understood. The previous result complements well with \cite[Theorem 1.2]{BFFP-3}. 

We now explain the proof of Theorem \ref{t.JSJatoroidal}. 
This is proved in \cite[Proposition 5.2]{FenleyFlow}.  
We give an alternate proof of super attracting/repelling behavior
which uses less of the transverse
regulating flow and the transverse lamination and relies only on large
scale geometry. 
We put this alternate proof here as it may be useful 
in other contexts.
In particular the proofs of Claim \ref{claim1} and 
Claim \ref{claim.superattracting} 
work even when the deck transformation $\gamma$
has two fixed points in the universal circle $\Su$.
The super attracting property proved in 
\cite[Proposition 5.2]{FenleyFlow} only works for $\gamma$ associated 
with an orbit of the flow which necessarily has (up to finite
iterate) at least four fixed points in $\Su$.

\begin{proof}[Proof of Theorem \ref{t.JSJatoroidal}]
For simplicity we assume that $M$ is orientable,
which can be accomplished by taking a double cover.
We will divide the proof in three steps. 
We assume some background on 3-manifolds, see \cite[Appendix A]{BFFP-2} and \cite[Appendix A]{BFFP-3}.  First we show the following claim reminiscent of \cite[Lemma 8.5]{BFFP-2}. Recall that for leaves $L,E \in \Ft$ we have a quasi-isometry $\tau_{L,E}: L \to E$ given by Proposition \ref{prop.quasiisome}. 

\begin{claim} \label{claim1}
Let $\gamma \in \pi_1(M)$ be a deck transformation of $M$ acting increasingly in the leaf space of $\Ft$ and such that $\gamma$ fixes an atoroidal piece $P$ and does not fix the lift of a JSJ tori. Then for every $R>0$ there is $K>0$ such that if $L \in \Ft$ is some leaf and we denote $g: L \to L$ to be the quasi-isometry given by $\gamma \circ \tau_{L,\gamma^{-1}L} : L \to L$ then there is a disk $D$ of radius $R$ in $L$ such that if $y \notin D$ then $d(y,g(y))> K$. 
\end{claim}
\begin{proof}
Notice that $\gamma$ is in $\pi_1(P)$ and does not represent
a peripheral curve in $P$.
The proof is the same as \cite[Lemma 8.5]{BFFP-2} once one notices that the hypothesis on $\gamma$ forces the existence of an axis for the action on the atoroidal piece (which admits a hyperbolic structure). This also follows from an argument similar to Lemma \ref{lema.superattracting} using the laminations constructed in \cite{FenleyFlow}. 
\end{proof}

Now, using some hyperbolic geometry on the leaves we can show: 

\begin{claim}\label{claim.superattracting} 
If $\gamma$ is a deck transformation as in the previous claim, then every fixed point of $\gamma$ acting on $\Su$ is either super attracting or super repelling. 
\end{claim}

\begin{proof} 
Take $\xi \in \Su$ and assume that it is fixed by the action of $\gamma$. Consider a leaf $L \in \Ft$ and a geodesic ray $r_0$ whose endpoint is $\hat \xi=\Theta_L(\xi)$ in $S^1(L)$. 
Let $g = \gamma \circ \tau_{L,\gamma^{-1}(L)}$. 
It extends to a homeomorphism
of $L \cup S^1(L)$ still denoted by $g$.
Since the action of $\gamma$ in $\Su$ is given by the action of $g$ as defined above in $S^1(L)$ via the identification of $\Theta_L$ we get that $g(r_0)$ is a quasi-geodesic ray that also lands in $\hat \xi$. 
Let $r_1$ be the geodesic with same starting point and ideal point
as $g(r_0)$. Notice that $r_1$ is asymptotic with $r_0$.

Fix a sequence of neighborhoods of $\hat \xi$ in $S^1(L)$ given by intervals $[a_n,b_n]$ in $S^1(L)$ so that the geodesics $\alpha_n$ joining $a_n,b_n$ converge to $\hat \xi$ and are orthogonal to $r_0$. 
It follows that $g(\alpha_n)$ is a quasigeodesic which makes a uniform (coarse) angle with $g(r_0)$. 
In other words if $\ell_n$ is the geodesic in $L$ with same ideal
points as $g(\alpha_n)$ then the angle between $\ell_n$ and $r_0$ is
bounded below by $a_0 > 0$.
This is because if the angle goes to $0$, then  one gets points $x_n, y_n$ in
$r_1, \ell_n$ respectively  which are very close in $L$
and very far away from the
intersection of $r_1, \ell_n$. 
In addition $x_n$ converging to $\Theta_L(\xi)$.
Pulling back by $g^{-1}$ (using that $r_1, r_0$ are asymptotic)
one gets points in $r_0, \alpha_n$ which
are boundedly close in $L$ but the points in $\alpha_n$ very far
from $r_0$. This is a contradiction to $g$ being a quasi-isometry..

Using the previous claim we obtain the desired result. See \cite[Lemma A.10]{BFFP-4} for a similar argument in a slightly different setting.
\end{proof}

Finally, \cite[Proposition 5.2]{FenleyFlow} gives a deck transformation fixing an atoroidal piece and with at least four fixed points at infinity. This completes the proof of Theorem \ref{t.JSJatoroidal}.
\end{proof}

\begin{remark}\label{r.dtransincoh}
Note that in the setting of the Theorem \ref{t.doubletrans} we also get that $f$ cannot be dynamically coherent (see Remark \ref{rem.incoherence}). 
\end{remark}

Let us now prove Theorem \ref{teoD} (we assume familiarity with some arguments from \cite{BFFP-2,BFFP-3}).

\begin{proof}[Proof of Theorem \ref{teoD}]
Under the assumptions of Theorem \ref{teoD}, it is shown in \cite{BFFP-3} that if $f$ is not a discretized Anosov flow, then $f$ is not dynamically coherent and one of the two branching foliations is $\RR$-covered, uniform and the good lift $\ft$ of $f$ acts as a translation in the leaf space. This implies as in Example \ref{example1} that we have a good pseudo-Anosov pair. Therefore, we can apply Theorem \ref{teo.curvesgeneral} to the corresponding foliation to deduce Theorem \ref{teoD}. 
\end{proof}

\subsection{Results in Seifert manifolds with only one pseudo-Anosov component}\label{ss.seifertgral}

There is also a partial statement similar to Theorem \ref{t.doubletrans} where we replace Theorem \ref{t.JSJatoroidal} with the results in \cite[Appendix A]{BFFP-2}.

\begin{teo}\label{t.seifpA}
Let  $f: M \to M$ be a partially hyperbolic diffeomorphism of a Seifert manifold so that the induced action on the base has some pseudo-Anosov component preserving branching foliations $\cs$ and $\cu$, which
are horizontal.  Then, both the center (branching) foliation and the strong stable foliation have small visual measure inside the leaves of $\cs$ (cf. Theorem \ref{teo.curvesgeneral}).  
Similarly for the center and unstable foliations in center unstable
leaves.
\end{teo}

We remark that the desired properties 
are independent of taking a finite
cover and lift of iterate, so we can assume orientation properties. We note that this result is new even for the examples of \cite{BGHP} where this behavior of the strong foliations was unknown.  
The horizontality condition implies in particular that $\cs, \cu$ are
$\R$-covered, which is needed to apply the results in this article (see \cite{HaPS} for conditions under which the assumption is met). Note that incoherence in this setting (cf. Remark \ref{r.dtransincoh}) was shown in \cite{BFFP-4}. Theorem \ref{teoE} is a direct consequence of Theorem \ref{t.seifpA}.

\end{document}